\documentclass[10pt]{article}
\usepackage{multicol}
\usepackage{indentfirst}
\usepackage{latexsym}
\usepackage{bm}
\usepackage{graphicx}
\usepackage{subfigure}
\usepackage[top=1in, bottom=1in, left=1.25in, right=1.25in]{geometry}

\usepackage{epsfig,dsfont,amssymb,amsmath,amsthm,amsfonts,amsbsy,mathrsfs,amsthm}
\usepackage{psfrag}
\usepackage{subfigure}
\usepackage{epstopdf}
\usepackage{hyperref}
\usepackage{float}
\usepackage{array}
\usepackage{color}
\usepackage{booktabs}

\theoremstyle{plain}
\newtheorem{theorem}{Theorem}
\newtheorem{lemma}{Lemma}
\newtheorem{corollary}{Corollary}
\newtheorem{proposition}{Proposition}

\theoremstyle{definition}
\newtheorem{definition}{Definition}

\newtheorem*{problem*}{Problem}

\theoremstyle{remark}
\newtheorem{remark}{Remark}

\newtheorem{case}{Case}[theorem]

\newtheorem*{solution*}{Solution}

\DeclareMathOperator*{\argmin}{arg\,min}
\renewcommand{\Re}{\mathrm{Re}\,}
\renewcommand{\Im}{\mathrm{Im}\,}

\begin{document}

\title{Solvability of the Stokes Immersed Boundary Problem in\\ Two Dimensions}

\author{Fang-Hua Lin, Jiajun Tong\\[5pt]Courant Institute}
\date{}

\maketitle
\begin{abstract}
We study coupled motion of a 1-D closed elastic string immersed in a 2-D Stokes flow,
known as the Stokes immersed boundary problem in two dimensions. Using the fundamental
solution of the Stokes equation and the Lagrangian coordinate of the string, we write
the problem into a contour dynamic formulation, which is a nonlinear non-local equation
solely keeping track of evolution of the string configuration. We prove existence and uniqueness of local-in-time solution starting from an arbitrary
initial configuration that is an $H^{5/2}$-function in the Lagrangian coordinate
satisfying the so-called well-stretched assumption. We also prove that when the
initial string configuration is sufficiently close to an equilibrium, which is an
evenly parameterized circular configuration, then global-in-time solution
uniquely exists and it will converge to an equilibrium configuration exponentially as
$t\rightarrow +\infty$. The technique in this paper may also apply to the Stokes
immersed boundary problem in three dimensions.
\end{abstract}


\section{Introduction}

The immersed boundary method was initially formulated by Peskin
\cite{peskin1972flowPhD,peskin1972flow} in early 1970s to study flow patterns around
heart valves, and later it develops into a generally effective method to solve
fluid-structure interaction problems \cite{peskin2002immersed}.
It gives birth to numerous studies of the numerical methods, along with applications in physics, biology and medical sciences.
See \cite{peskin2002immersed, mittal2005immersed} and the references therein.
Various mathematical analysis have also been performed based on the model formulation
itself, e.g.\;\cite{stockie1995stability,stockie1997analysis,mori2008convergence}.
From the analysis point of view, the immersed boundary problem is intriguing on its own right.
It is nonlinear by nature, featuring free moving boundary and singular forcing, which are not well-studied in the classic mathematical theory of hydrodynamics \cite{temam1984navier}.

In this paper, we shall consider Stokes immersed boundary problem in two dimensions.
It models the scenario where there is a 1-D closed elastic string (or fibre) immersed and moving in the 2-D Stokes flow: the string exerts force on the fluid and generates the flow, while the flow in turn moves the string and changes its configuration.
The mathematical formulation will be given below.
We will prove solvability of the string motion and its asymptotic behavior near equilibrium.
Much of the analysis in this paper also applies to immersed boundary problems in three dimensions.

A similar type of problems on one- \cite{solonnikov1977solvability,solonnikov1986solvability,shibata2007free} or two-phase \cite{denisova1991solvability,tanaka1993global,giga1994global,denisova1994problem,denisova1994solvability,shimizu2011local,kohne2013qualitative,solonnikov2014theory} incompressible fluid motion has been extensively studied.
In these settings, the space is occupied by one incompressible viscous fluid and the vacuum, or by two immiscible incompressible viscous fluids; the fluids move with or without surface tension on their interface.
Solvability results have been established in various function spaces.
The main difference between these problems and ours is that only the geometry (such as length, area and curvature) of the interface is involved
there in determining the force balance at the interface.
In particular, it does not depend on how the immersed string or membrane is parametrized.
Consequently, one can use either Eulerian or Lagrangian approach to study the evolution of interfaces.
However, in the immersed boundary problems, elastic strings or membranes have their internal structures and their dynamics also depends on constitutive law of elasticity, which varies from case to case.
In other words, intrinsic parametrization of the immersed boundary and its elastic deformation should play a role.
Indeed, immersed boundaries with identical overall shape can generate force differently.
One can easily construct a 1-D closed string with a circular shape, yet far more
stretched at some point than somewhere else.
In this case, we shall see that the force on the string is not everywhere pointing inward normal to the string.
This suggests that a pure Eulerian approach employed in many mathematical studies of free boundary problems in hydrodynamics (e.g.\;\cite{bertozzi1993global}) would not suffice.
One needs to keep track of the configuration of the immersed boundary, which is typical in the nonlinear elasticity problems, and different techniques need to be used.

\subsection{The Stokes immersed boundary problem in two dimensions}
Consider a 1-D neutrally buoyant massless elastic closed string immersed in 2-D Stokes flow.
The string is modeled as a Jordan curve $\Gamma_t$ parameterized by $X(s,t)$, where $s\in\mathbb{T}$ is the Lagrangian coordinate (or the material coordinate) and $t\geq 0$ is the time variable. Here, $\mathbb{T}\triangleq \mathbb{R}/2\pi\mathbb{Z}$ is the 1-D torus equipped with the induced metric.
We always assume that at least $X(\cdot,t)\in H^2(\mathbb{T})$ for all $t$.
The flow field in the immersed boundary problem is determined by
\begin{equation}
\begin{split}
&\;-\mu_0\Delta u +\nabla p = f(x,t),\quad x\in\mathbb{R}^2,\;t>0,\\
&\;\mathrm{div}\, u = 0,\\
&\;|u|,|p|\rightarrow 0\mbox{ as }|x|\rightarrow \infty.
\end{split}
\label{eqn: stokes equation}
\end{equation}
Here $u(x,t)$ is the velocity field in $\mathbb{R}^2$ and $p$ is the pressure; $\mu_0>0$ is the dynamic viscosity; $f(x,t)$ is the elastic force exerted on the fluid generated by the string, given by \cite{peskin2002immersed}
\begin{equation}
f(x,t) = \int_\mathbb{T} F(s,t) \delta (x-X(s,t))\,ds.
\label{eqn: force in the immersed boundary problem general form}
\end{equation}
Here $\delta$ is the 2-D delta measure, which means the force is only supported on the string.
$F(s,t)$ is the force in the Lagrangian formulation; it is given by
\begin{equation}
F(s,t) = \frac{\partial}{\partial s}\left(\mathcal{T}(|X_s|)\frac{X_s}{|X_s|}\right),\quad \mathcal{T}(|v|) = \mathcal{E}'(|v|).
\label{eqn: force in the immersed boundary problem general form Lagrangian}
\end{equation}
where $X_s = \partial X/\partial s$, $\mathcal{T}$ is the tension in the string and $\mathcal{E}$ is the elastic energy density. In the following discussion, we shall take
\begin{equation}
\mathcal{E}(|v|) = k_0|v|^2/2.
\label{eqn: elastic energy density}
\end{equation}
In this case, each infinitesimal segment of the string behaves like a Hookean spring with elasticity coefficient $k_0>0$, and thus
$F(s,t) = k_0X_{ss}(s,t)$.
It will be clear below that most of the discussion in this paper can also apply to more general elastic energy of other forms.
The model is closed by the kinematic equation of the string,
\begin{equation}
\frac{\partial X}{\partial t}(s,t) = u(X(s,t), t),
\label{eqn: kinematic equation of membrane}
\end{equation}
which means the string moves with the flow.

For simplicity, we shall take $\mu_0 = k_0 = 1$ in the rest of the paper.
Indeed, one can easily normalize both coefficients simultaneously by properly redefining $u$, $p$ and the time variable $t$.
We shall always omit the $t$-dependence whenever it is convenient; and we shall also write $X'(s')$ and $X''(s')$ in the places of $X_s(s',t)$ and $X_{ss}(s',t)$ respectively.

\subsection{Contour dynamic formulation}\label{section: contour dynamic formulation}
The starting point of the analysis in this paper is the following proposition.
It rewrites the original immersed boundary problem \eqref{eqn: stokes equation}-\eqref{eqn: kinematic equation of membrane} that is in mixed Eulerian and Lagrangian formulation into a pure Lagrangian formulation, which we will call \emph{contour dynamic formulation}.
\begin{proposition}\label{prop: tranform into contour dynamic formulation}
Under the assumptions that $X(\cdot,t)\in H^2(\mathbb{T})$ for all $t$, and that there $\exists\,\lambda>0$, s.t.\;$\forall\,s_1,s_2\in\mathbb{T}$,
\begin{equation}
|X(s_1,t)-X(s_2,t)|\geq \lambda|s_1-s_2|,
\label{eqn: well_stretched assumption}
\end{equation}
where $|s_1-s_2|$ is the distance between $s_1$ and $s_2$ on $\mathbb{T}$, the evolution of $X(s,t)$ in the 2-D Stokes immersed boundary problem \eqref{eqn: stokes equation}-\eqref{eqn: kinematic equation of membrane} is equivalently given by
\begin{equation}
X_t(s,t)=\mathcal{L}X(s,t)+g_X(s,t),\quad X(s,0) = X_0(s),
\label{eqn: contour dynamic formulation of the immersed boundary problem}
\end{equation}
where $\mathcal{L}\triangleq-\frac{1}{4}(-\Delta)^{1/2}$, and
\begin{align}
g_X(s,t) = &\;\int_{\mathbb{T}} \Gamma_0(s,s',t)\,ds' +\frac{1}{4}(-\Delta)^{1/2}X(s,t),\label{eqn: definition of g_X}\\
\Gamma_0(s,s',t) = &\;-\partial_{s'}[G(X(s,t)-X(s',t))](X'(s',t)-X'(s,t)).\label{eqn: introduce the notation Gamma_0}
\end{align}
Here $(-\Delta)^{1/2}$ on $\mathbb{T}$ is understood as a Fourier multiplier or equivalently the following singular integral
\begin{equation}
(-\Delta)^{1/2}Y(s) \triangleq -\frac{1}{\pi} \mathrm{p.v.}\int_\mathbb{T}\frac{Y(s')-Y(s)}{4\sin^2\left(\frac{s'-s}{2}\right)}\,ds',
\end{equation}
and
\begin{equation}
G(x) = \frac{1}{4\pi}\left(-\ln |x| Id +\frac{x \otimes x}{|x|^2}\right)\label{eqn: 2D stokeslet}
\end{equation}
is the fundamental solution of the 2-D Stokes equation for the velocity field \cite{pozrikidis1992boundary}.
\end{proposition}
We call \eqref{eqn: well_stretched assumption} \emph{well-stretched assumption};
\eqref{eqn: contour dynamic formulation of the immersed boundary problem} is called \emph{the contour dynamic formulation} of the immersed boundary problem.
The proof of Proposition \ref{prop: tranform into contour dynamic formulation} is left to Section \ref{section: justification of contour dynamic formulation}.
In the sequel, we shall focus on \eqref{eqn: contour dynamic formulation of the immersed boundary problem} and prove existence and uniqueness of its solutions and their properties.
Estimates of the velocity field $u_X(x,t)$ can be easily obtained based on that; see Lemma \ref{lemma: the velocity field is continuous} below.
Note that the subscript of $u_X$ stresses that it is determined by $X(s,t)$; see Section \ref{section: justification of contour dynamic formulation} for more details.

\subsection{Main results}
Let us introduce a notation before we state the main results of the paper.
With $T>0$, define
\begin{equation}
\Omega_{T} = \left\{Y(s,t)\in L^{\infty}_T H^{5/2}\cap L^2_T H^{3}(\mathbb{T}):\;Y_t(s,t)\in L^2_T H^2(\mathbb{T})\right\}.
\label{eqn: define the primary function space to prove the local existence}
\end{equation}
It is equipped with the norm
\begin{equation*}
\|Y(s,t)\|_{\Omega_{T}} \triangleq \|Y\|_{L^{\infty}_T {H}^{5/2}(\mathbb{T})}+\|Y\|_{L^2_T {H}^{3}(\mathbb{T})}+\|Y_t\|_{L^2_T {H}^{2}(\mathbb{T})}.
\end{equation*}
Here $L^{\infty}_T {H}^{5/2}(\mathbb{T}) = L^\infty([0,T];{H}^{5/2}(\mathbb{T}))$, and $L^2_T {H}^{3}(\mathbb{T})$ and $L^2_T {H}^{2}(\mathbb{T})$ have similar meanings.
Then we are able to prove the local well-posedness of the immersed boundary problem \eqref{eqn: contour dynamic formulation of the immersed boundary problem}.
\begin{theorem}[Existence of the local-in-time solution]\label{thm: local in time existence}
Suppose $X_0(s) \in H^{5/2}(\mathbb{T})$, s.t.\;there exists some $\lambda>0$,
\begin{equation}
|X_0(s_1)-X_0(s_2)|\geq \lambda|s_1-s_2|,\quad \forall\, s_1, s_2\in \mathbb{T}.
\label{eqn: bi Lipschitz assumption in main thm}
\end{equation}
Then there exists $T_0 = T_0(\lambda, \|X_0\|_{\dot{H}^{5/2}})\in(0,+\infty]$ and a solution $X(s,t)\in \Omega_{T_0}\cap C_{[0,T_0]}H^{5/2}(\mathbb{T})$ of the immersed boundary problem \eqref{eqn: contour dynamic formulation of the immersed boundary problem}, satisfying that
\begin{equation}
\|X\|_{L^\infty_{T_0} \dot{H}^{5/2}\cap L^2_{T_0} \dot{H}^{3}(\mathbb{T})}\leq 4\|X_0\|_{\dot{H}^{5/2}(\mathbb{T})},\quad \|X_t\|_{L^2_{T_0} \dot{H}^{2}(\mathbb{T})}\leq \|X_0\|_{\dot{H}^{5/2}(\mathbb{T})},
\label{eqn: a priori estimate for the local solution in the main theorem}
\end{equation}
and that for $\forall\,s_1,s_2\in\mathbb{T}$ and $t\in[0,T_0]$,
\begin{equation}
\left|X(s_1,t) - X(s_2,t)\right| \geq \frac{\lambda}{2}|s_1 - s_2|.
\label{eqn: uniform bi lipschitz constant of the local solution in the main theorem}
\end{equation}
\end{theorem}
We write $C_{[0,T_0]}H^{5/2}(\mathbb{T})$ instead of $C_{T_0}H^{5/2}(\mathbb{T})$ to stress continuity up to the end points of the time interval.

\begin{theorem}[Uniqueness of the local-in-time solution]\label{thm: local in time uniqueness}
Suppose $X_0(s) \in H^{5/2}(\mathbb{T})$ satisfies \eqref{eqn: bi Lipschitz assumption in main thm} with some $\lambda>0$.
Given an arbitrary $c\in(0,1)$, the immersed boundary problem \eqref{eqn: contour dynamic formulation of the immersed boundary problem} has at most one solution $X\in\Omega_T$ satisfying that $\forall\,s_1,s_2\in\mathbb{T}$ and $\forall\,t\in[0,T]$,
\begin{equation}
|X(s_1,t)-X(s_2,t)|\geq c\lambda|s_1-s_2|.
\label{eqn: bi lipschitz assumption in uniqueness thm}
\end{equation}
In particular, the local-in-time solution obtained in Theorem \ref{thm: local in time existence} is unique in $\Omega_{T_0}$.
\end{theorem}

To state the results on the global existence of solutions near equilibrium configurations and its exponential convergence, we need the following definition.
\begin{definition}\label{def: closest equilbrium state}
Assume $Y(s) \in H^{5/2}(\mathbb{T})$ defines a Jordan curve in the plane, s.t.\;the area of domain enclosed by $Y$ is $\pi R_Y^2$ with $R_Y>0$, i.e.,
\begin{equation}
\frac{1}{2}\int_{\mathbb{T}} Y(s)\times Y'(s)\,ds = \pi R_Y^2.
\label{eqn: enclosed area is pi}
\end{equation}
We call $R_Y$ the \emph{effective radius} of $Y(s)$.
Define
\begin{equation}
Y_{\theta,x}(s) = (R_Y\cos (s+\theta), R_Y\sin(s+\theta))^T + x
\label{eqn: define a parameterization of the candidate equilibrium}
\end{equation}
with $\theta\in[0,2\pi)$ and $x\in \mathbb{R}^2$.
Let
\begin{equation}
(\theta_*,x_*) =\argmin_{\theta\in[0,2\pi), x\in\mathbb{R}^2}\int_{\mathbb{T}}|Y(s)-Y_{\theta,x}(s)|^2\,ds.
\label{eqn: define closest equilibrium and optimal parameters}
\end{equation}
Then $Y_*(s) \triangleq Y_{\theta_*,x_*}(s)$ is called \emph{the closest equilibrium configuration} to $Y(s)$.
\end{definition}
Properties of the closest equilibrium configuration will be discussed in Section \ref{section: global existence}.
Now we have
\begin{theorem}[Existence and uniqueness of global-in-time solution near equilibrium]\label{thm: global existence near equilibrium}
There exist universal $\varepsilon_*, \xi_*>0$, such that for $\forall\, X_0(s)\in H^{5/2}(\mathbb{T})$ satisfying
\begin{align}
\|X_0(s) - X_{0*}(s)\|_{\dot{H}^{5/2}(\mathbb{T})}\leq &\;\varepsilon_* R_{X_0},\label{eqn: closeness condition of H 2.5 norm}\\
\|X_0(s) - X_{0*}(s)\|_{\dot{H}^{1}(\mathbb{T})}\leq &\;\xi_* R_{X_0},\label{eqn: closeness condition of H 1 norm}
\end{align}
with $X_{0*}(s)$ being the closest equilibrium configuration to $X_0(s)$, there exists a unique solution $X(s,t)\in C_{[0,+\infty)}H^{5/2}\cap L^2_{[0,+\infty),loc}H^3(\mathbb{T})$ satisfying $X_t(s,t)\in L^2_{[0,+\infty),loc}H^2(\mathbb{T})$ for the immersed boundary problem \eqref{eqn: contour dynamic formulation of the immersed boundary problem}.
It satisfies the following estimates
\begin{align}
\|X-X_{*}\|_{L^{\infty}_{[0,+\infty)}\dot{H}^{5/2}(\mathbb{T})}\leq &\; \sqrt{2}\varepsilon_* R_{X_0},\label{eqn: estimates on the distance to the equilibrium for the global solution in all time intervals}\\
\left|X(s_1,t) - X(s_2,t)\right| \geq &\;\frac{1}{2\pi}|s_1 - s_2|,\quad \forall \,t\in[0,+\infty),\;s_1,s_2\in\mathbb{T}.\label{eqn: well-stretched constant estimates for the global solution in all time intervals}
\end{align}
In particular,
\begin{equation}
\|X\|_{L^{\infty}_{[0,+\infty)}\dot{H}^{5/2}(\mathbb{T})}\leq CR_{X_0}
\label{eqn: uniform bound of H 2.5 norm for the global solution}
\end{equation}
for some universal $C$.
\end{theorem}

\begin{theorem}[Exponential convergence to the equilibriums]\label{thm: exponential convergence}
Let $X_0\in H^{5/2}(\mathbb{T})$ satisfy all the assumptions in Theorem \ref{thm: global existence near equilibrium} and let $X$ be the unique global solution of \eqref{eqn: contour dynamic formulation of the immersed boundary problem} starting from $X_0$ obtained in Theorem \ref{thm: global existence near equilibrium}.
There exist universal constants $\xi_{**}, \alpha_*>0$, such that if in addition
\begin{equation*}
\|X_0(s) - X_{0*}(s)\|_{\dot{H}^{1}(\mathbb{T})}\leq \xi_{**} R_{X_0},\label{eqn: closeness condition of H 1 norm for exp convergence}
\end{equation*}
then
\begin{enumerate}
\item With some universal constant $C>0$,
\begin{equation}
\begin{split}
&\;\|X-X_{*}\|_{\dot{H}^{5/2}(\mathbb{T})}(t) \\
&\;\quad \leq Ce^{-\alpha_* t}\max\{\|X_0-X_{0*}\|_{\dot{H}^{5/2}(\mathbb{T})},\|X_0-X_{0*}\|_{\dot{H}^{1}(\mathbb{T})}(|\ln \|X_0-X_{0*}\|_{\dot{H}^{1}(\mathbb{T})}|+1)^2\}\\
&\;\quad \triangleq Ce^{-\alpha_* t} B(X_0).
\end{split}
\label{eqn: exp convergence in H2.5 norm}
\end{equation}
\item 
There exists an equilibrium configuration $X_\infty\triangleq x_\infty+(R_{X_0}\cos(s+\theta_\infty), R_{X_0}\sin(s+\theta_\infty))^T$, such that 
\begin{equation}
\|X(t)-X_{\infty}\|_{\dot{H}^{5/2}(\mathbb{T})}\leq CB(X_0)e^{-\alpha_* t},
\label{eqn: exp convergence to a fixed configuration}
\end{equation}
where $C$ is a universal constant and $B(X_0)$ is defined in \eqref{eqn: exp convergence in H2.5 norm}.
\end{enumerate}
\end{theorem}

The rest of the paper is organized as follows.
In Section \ref{section: justification of contour dynamic formulation}, the reformulation in Proposition \ref{prop: tranform into contour dynamic formulation} is justified.
We also discuss properties of the flow field and law of energy dissipation in the system; their proofs are left to the Appendix \ref{appendix section: study of the flow field}.
In Section \ref{section: a priori estimates}, we will prove a priori estimates necessary for proving the local well-posedness of the contour dynamic formulation \eqref{eqn: contour dynamic formulation of the immersed boundary problem}.
In particular, in Section \ref{section: preliminary a priori estimates}, we prove some preliminary estimates as building blocks of more complicated bounds in Section \ref{section: a priori estimates of the immersed boundary problem}, which is devoted to finding out derivatives of $g_X$ and proving its $H^2$-estimate.
In Section \ref{section: local existence and uniqueness}, we will establish the local well-posedness of \eqref{eqn: contour dynamic formulation of the immersed boundary problem}.
In Section \ref{section: global existence}, we will show global-in-time existence of solutions of \eqref{eqn: contour dynamic formulation of the immersed boundary problem} provided that the initial configuration is sufficiently close to an equilibrium configuration.
In Section \ref{section: exp convergence}, we will first prove a lower bound of the rate of energy dissipation in Section \ref{section: lower bound for energy dissipation rate} when the solution is close to an equilibrium.
Based on that, we will show exponential convergence of the solution to an equilibrium configuration in Section \ref{section: proof of exponential convergence to equilibrium configurations}.
Some other auxiliary results will be stated and proved in the Appendix \ref{appendix section: estimates involving L} and \ref{appendix section: auxiliary calculations}.

\section{Problem Reformulation and the Flow Field}\label{section: justification of contour dynamic formulation}

\subsection{Proof of Proposition \ref{prop: tranform into contour dynamic formulation}}\label{section: proof of contour dynamic formulation}
We first justify Proposition \ref{prop: tranform into contour dynamic formulation}, which reformulates the original immersed boundary problem \eqref{eqn: stokes equation}-\eqref{eqn: kinematic equation of membrane} into the contour dynamic formulation \eqref{eqn: contour dynamic formulation of the immersed boundary problem}.
Some of the arguments are redundant for proving the proposition itself, but we still derive them here as they will be useful in proving Lemma \ref{lemma: the velocity field is continuous} and Lemma \ref{lemma: energy estimate} below.

\begin{proof}[Proof of Proposition \ref{prop: tranform into contour dynamic formulation}]
In 2-D stationary Stokes flow, the velocity field $u$ and the pressure $p$ are instantaneously determined by the forcing $f$ through fundamental solutions
\begin{equation}
G(x) = \frac{1}{4\pi}\left(-\ln |x| Id +\frac{x \otimes x}{|x|^2}\right),\quad Q(x) = \frac{x}{2\pi|x|^2},\label{eqn: fundamental solution for pressure for 2D Stokes equation}
\end{equation}
respectively \cite{pozrikidis1992boundary}, where $Id$ is the $2\times 2$-identity matrix. Hence,
\begin{equation}
\begin{split}
u_X(x,t) =&\;\int_{\mathbb{R}^2} G(x-y)f(y,t) \,dy=\int_{\mathbb{R}^2}\int_\mathbb{T} G(x-y)\delta(x-X(s',t))F(s',t) \,ds'dy\\
=&\;\int_{\mathbb{T}} G(x-X(s',t))X_{ss}(s',t) \,ds',
\end{split}
\label{eqn: expression for velocity field}
\end{equation}
This is well-defined for $x\not\in\Gamma_t$ and $X(\cdot,t)\in H^2(\mathbb{T})$. The subscript of $u_X$ stresses that it is determined by the configuration $X$.
For $x = X(s,t)\in\Gamma_t$, by \eqref{eqn: well_stretched assumption},
\begin{equation*}
|G(X(s)-X(s'))|\leq C(\lambda)(1+|\ln |s-s'||).
\end{equation*}
Hence, $G(X(s)-X(\cdot))\in L^2(\mathbb{T})$ and \eqref{eqn: expression for velocity field} is well-defined.


For $x\not \in \Gamma_t$, we do integration by parts in \eqref{eqn: expression for velocity field} and find that
\begin{equation}
u^i_X(x) = \int_{\mathbb{T}} -\partial_{s'} [G^{ij}(x-X(s'))][X'(s')-C_x]^j \,ds',
\label{eqn: expression of velocity field after integration by parts}
\end{equation}
where the superscripts stand for the indices of entries, and $C_x$ is any arbitrary constant vector independent of $s'$.
We may take $C_x = X'(s_x)$, where $s_x$ is defined by
\begin{equation}
|x-X(s_x)| = \inf_{s\in\mathbb{T}}|x-X(s)| = \mathrm{dist}(x,X(\mathbb{T})).
\label{eqn: definition of s_x}
\end{equation}
Note that $s_x$ may not be unique; pick an arbitrary one if it is the case.
Hence,
\begin{equation}
u_X(x) = 
\int_{\mathbb{T}} -\partial_{s'} [G(x-X(s'))](X'(s')-X'(s_x))\,ds'
\label{eqn: 2D velocity field}
\end{equation}
Similarly, by integration by parts and taking the indetermined constant to be $0$, we find for $x\not \in \Gamma_t$,
\begin{equation}
p_X(x,t) =\frac{1}{2\pi}\int_{\mathbb{T}}  \frac{|X'(s')|^2}{|X(s')-x|^2} - \frac{2[(X(s')-x)\cdot X'(s')]^2}{|X(s')-x|^4}\,ds'.
\label{eqn: 2D pressure field}
\end{equation}

For $x = X(s,t)\in \Gamma_t$, by \eqref{eqn: expression for velocity field},
\begin{equation*}
\begin{split}
u_X(X(s)) =&\;\lim_{\varepsilon \rightarrow 0^+}\int_{|s'-s|\geq \varepsilon} G(X(s)-X(s'))X''(s') \,ds'\\
=&\;\lim_{\varepsilon \rightarrow 0^+}\int_{|s'-s|\geq \varepsilon} -\partial_{s'}[G(X(s)-X(s'))](X'(s')-X'(s)) \,ds'\\
 &\;+\lim_{\varepsilon \rightarrow 0^+}G(X(s)-X(s-\varepsilon))(X'(s-\varepsilon)-X'(s)) \\
 &\;-\lim_{\varepsilon \rightarrow 0^+}G(X(s)-X(s+\varepsilon))(X'(s+\varepsilon)-X'(s)).
\end{split}
\end{equation*}
Using \eqref{eqn: well_stretched assumption} and the assumption that $X(\cdot,t)\in H^2(\mathbb{T})$, we find
\begin{equation*}
\begin{split}
|G(X(s)-X(s-\varepsilon))(X'(s-\varepsilon)-X'(s))|\leq &\;C(\lambda)(1+|\ln \varepsilon|)\varepsilon^{1/2}\|X'\|_{\dot{C}^{1/2}(\mathbb{T})}\\
\leq &\;C(\lambda)(1+|\ln \varepsilon|)\varepsilon^{1/2}\|X\|_{\dot{H}^2(\mathbb{T})}.
\end{split}
\end{equation*}
It goes to $0$ as $\varepsilon\rightarrow 0^+$. A similar bound holds for $|G(X(s)-X(s+\varepsilon))(X'(s+\varepsilon)-X'(s))|$. Hence,
\begin{equation}
\begin{split}
u_X(X(s))=&\;\mathrm{p.v.}\int_{\mathbb{T}} -\partial_{s'}[G(X(s)-X(s'))](X'(s')-X'(s)) \,ds'\\
=&\;\frac{1}{4\pi}\mathrm{p.v.}\int_{\mathbb{T}} \left[\frac{(X(s')-X(s))\cdot X'(s')}{|X(s')-X(s)|^2}Id \right.\\
&\;\quad- \frac{X'(s')\otimes (X(s')-X(s))+(X(s')-X(s))\otimes X'(s')}{|X(s')-X(s)|^2}\\
&\;\left.\quad+\frac{2(X(s')-X(s))\cdot X'(s') (X(s')-X(s))\otimes (X(s')-X(s))}{|X(s')-X(s)|^4}\right](X'(s')-X'(s)) \,ds'.
\end{split}
\label{eqn: velocity of membrane}
\end{equation}
In \eqref{eqn: introduce the notation Gamma_0}, we denoted the integrand in \eqref{eqn: velocity of membrane} by $\Gamma_0(s,s')$.
It is trivial to show that
\begin{equation*}
|\Gamma_0(s,s')|\leq C\lambda^{-1} |s'-s|^{-1/2}\|X\|_{\dot{C}^1(\mathbb{T})} \|X'\|_{\dot{C}^{1/2}(\mathbb{T})}\leq C\lambda^{-1}|s'-s|^{-1/2}\|X\|_{\dot{H}^2(\mathbb{T})}^2.
\end{equation*}
Hence, $\Gamma_0(s,s')$ is integrable, and the principal value integral in \eqref{eqn: velocity of membrane} can be replaced by the usual integral.
As a byproduct, we also find a bound for $u_X(X(s))$,
\begin{equation}
|u_X(X(s))|\leq C\lambda^{-1}\|X\|_{\dot{H}^2(\mathbb{T})}^2.
\label{eqn: a trivial L^infty bound for velocity}
\end{equation}
\eqref{eqn: velocity of membrane} together with \eqref{eqn: kinematic equation of membrane} gives \eqref{eqn: contour dynamic formulation of the immersed boundary problem}.
Once \eqref{eqn: contour dynamic formulation of the immersed boundary problem} is solved, we can recover $u$ and $p$ by \eqref{eqn: expression for velocity field} and \eqref{eqn: 2D pressure field}. The original immersed boundary problem is then solved.
This completes the proof of Proposition \ref{prop: tranform into contour dynamic formulation}.
\end{proof}

\begin{remark}
\eqref{eqn: velocity of membrane} can be equivalently written as
\begin{equation}
\begin{split}
u_X(X(s))=&\;\mathrm{p.v.}\int_{\mathbb{T}} -\partial_{s'}[G(X(s)-X(s'))]X'(s') \,ds'\\
=&\;\frac{1}{4\pi}\mathrm{p.v.}\int_{\mathbb{T}} \left[- \frac{|X'(s')|^2}{|X(s')-X(s)|^2}+\frac{2[(X(s')-X(s))\cdot X'(s')]^2}{|X(s')-X(s)|^4} \right](X(s')-X(s)) \,ds'.
\end{split}
\label{eqn: equivalent formualtion of membrane velocity}
\end{equation}
Indeed, under the assumptions $X(\cdot,t)\in H^2(\mathbb{T})$ and \eqref{eqn: well_stretched assumption}, 
\begin{equation}
\mathrm{p.v.}\int_{\mathbb{T}} -\partial_{s'}[G(X(s)-X(s'))]\,ds' =  \lim_{\varepsilon\rightarrow 0^+} G(X(s)-X(s+\varepsilon)) - G(X(s)-X(s-\varepsilon)) = 0.
\label{eqn: pv integral vanishes}
\end{equation}
To obtain the last convergence, we derive that, since $|X'(s)|\geq \lambda$,
\begin{equation*}
\ln \frac{|X(s)-X(s+\varepsilon)|}{|X(s)-X(s-\varepsilon)|} = \ln \frac{|X(s)-X(s+\varepsilon)|/\varepsilon}{|X(s)-X(s-\varepsilon)|/\varepsilon}\rightarrow \ln \frac{|X'(s)|}{|X'(s)|} = 0,
\end{equation*}
and similarly,
\begin{equation*}
\frac{(X(s)-X(s\pm\varepsilon))\otimes (X(s)-X(s\pm\varepsilon))}{|X(s)-X(s\pm \varepsilon)|^2}\rightarrow \frac{X'(s)\otimes X'(s)}{|X'(s)|^2}.
\end{equation*}
\eqref{eqn: equivalent formualtion of membrane velocity} can be viewed as taking $C_x = 0$ in \eqref{eqn: expression of velocity field after integration by parts}.
\qed
\end{remark}

\begin{remark}
The reason why we single out the term $\mathcal{L}X$ in Proposition \ref{prop: tranform into contour dynamic formulation} comes from the following suggestive calculation starting from \eqref{eqn: equivalent formualtion of membrane velocity}.
Note that the integrals in \eqref{eqn: velocity of membrane} and \eqref{eqn: equivalent formualtion of membrane velocity} give the same value, so we use them interchangeably.

Suppose $X(\cdot,t)$ is sufficiently smooth.
There is a singularity in the integrand of \eqref{eqn: equivalent formualtion of membrane velocity} as $s'\rightarrow s$.
Consider $s'$ very close to $s$ and we formally use $(s'-s)X'(s')$ to approximate $X(s')-X(s)$ in \eqref{eqn: equivalent formualtion of membrane velocity}.
In this way, when $|s'-s|$ is sufficiently small, we formally find
\begin{equation*}
-\partial_{s'}[G(X(s)-X(s'))]X'(s')\sim \frac{1}{4\pi} \frac{X'(s')}{s'-s}\sim -\frac{1}{4}\cdot \frac{X'(s')}{2\pi\tan\left(\frac{s-s'}{2}\right)},
\end{equation*}
which presumably accounts for the principal part of the singular integral in \eqref{eqn: equivalent formualtion of membrane velocity}.
Recall that the Hilbert transform $\mathcal{H}$ on $\mathbb{T}$ is defined as \cite{grafakos2008classical} 
\begin{equation*}
\mathcal{H}Y(s) = \frac{1}{2\pi}\mathrm{p.v.}\int_{\mathbb{T}}\cot\left(\frac{s-s'}{2}\right)Y(s').
\end{equation*}
Hence, if we take out $-\frac{1}{4}\mathcal{H}X' = \mathcal{L}X$ in \eqref{eqn: equivalent formualtion of membrane velocity}, what remains is \emph{expected} to be regular.
We shall see that $\mathcal{L}X$ provides nice dissipation property that helps prove well-posedness of \eqref{eqn: contour dynamic formulation of the immersed boundary problem}.
See Lemma \ref{lemma: improved Hs estimate and Hs continuity of semigroup solution} and Lemma \ref{lemma: a priori estimate of nonlocal eqn} for some relevant estimates.

It should be noted that the very idea has been adopted in early numerical literature to, for example, remove stiffness in computing the evolution of elastic immersed boundary in 2-D Stokes flow or the motion of interface with surface tension in 2-D incompressible irrotational flow.
See e.g.\;\cite{hou2008removing, hou1994removing} and references therein.
\qed
\end{remark}

\subsection{Regularity of the flow field and energy dissipation}\label{section: energy estimate}
As is mentioned above, once \eqref{eqn: contour dynamic formulation of the immersed boundary problem} is solved, we can obtain the flow field $u_X$ by \eqref{eqn: expression for velocity field}.
The following lemma characterizes its regularity.
\begin{lemma}\label{lemma: the velocity field is continuous}
Let $X(\cdot,t)\in H^2(\mathbb{T})$ and satisfy the well-stretched condition \eqref{eqn: well_stretched assumption}.
Then $u_X(\cdot,t)$ defined by \eqref{eqn: expression for velocity field} (or equivalently \eqref{eqn: 2D velocity field}, \eqref{eqn: velocity of membrane} and \eqref{eqn: equivalent formualtion of membrane velocity}) is continuous in $\mathbb{R}^2$.
Moreover, $\nabla u_X(\cdot ,t)\in L^2(\mathbb{R}^2)$.
\end{lemma}

\begin{remark}
That $u_X(x,t)$ is continuous throughout $\mathbb{R}^2$ agrees with the intuition that the string moves with the ambient flow, and there is no jump in velocity across the string.
\qed
\end{remark}

As a dissipative system, the Stokes immersed boundary problem enjoys a natural law of energy dissipation, which is useful in proving existence and asymptotic behavior of global solution near equilibrium in Section \ref{section: global existence} and Section \ref{section: exp convergence}.

\begin{lemma}\label{lemma: energy estimate}
Assume $X(s,t)\in C_{T}H^2(\mathbb{T})$ with $X_t(s,t)\in L^2_{T}H^1(\mathbb{T})$ is a solution of \eqref{eqn: contour dynamic formulation of the immersed boundary problem} with some $T>0$ satisfying \eqref{eqn: well_stretched assumption} with constant $\lambda >0$, and $u_X(x,t)$ is the corresponding velocity field defined by the Stokes equation \eqref{eqn: stokes equation}, with $\nabla u_X(x,t)\in L^\infty_{T}L^2(\mathbb{R}^2)$ (showed in \eqref{eqn: a trivial bound for the energy dissipation rate or H1 semi norm of velocity field} in the proof of Lemma \ref{lemma: the velocity field is continuous}).
Then
\begin{equation}
\frac{1}{2}\frac{d}{dt}\int_{\mathbb{T}}|X'(s,t)|^2\,ds = -\int_{\mathbb{R}^2}|\nabla u_X(x,t)|^2\,dx
\label{eqn: energy estimate on each time slice simplified version}
\end{equation}
holds in the scalar distribution sense, and
\begin{equation}
\frac{1}{2}\int_{\mathbb{T}}|X'(s,T)|^2\,ds - \frac{1}{2}\int_{\mathbb{T}}|X'(s,0)|^2\,ds =-\int_{0}^T\int_{\mathbb{R}^2}|\nabla u_X(x,t)|^2\,dxdt.
\label{eqn: energy estimate of Stokes immersed boundary problem}
\end{equation}
In particular, the total elastic energy of the string $\mathcal{E}_X \triangleq \frac{1}{2}\|X(\cdot,t)\|_{\dot{H}^1(\mathbb{T})}^2$ always decreases in $t$.
\end{lemma}

The proofs of these lemmas are technical.
We leave them to Appendix \ref{appendix section: study of the flow field}.

\section{A Priori Estimates}\label{section: a priori estimates}
In this section, we shall prove a priori estimates that are needed in proving well-posedness of \eqref{eqn: contour dynamic formulation of the immersed boundary problem}.

\subsection{Preliminaries}\label{section: preliminary a priori estimates}

First we introduce some notations that will be heavily used in the rest of the paper.
Suppose $X\in H^3(\mathbb{T})$. For $s,s'\in\mathbb{T}$, let $\tau = s'-s\in[-\pi,\pi)$. For $s'\not = s$, define
\begin{equation}
L(s,s') = \frac{X(s')-X(s)}{\tau},\quad M(s,s') = \frac{X'(s')-X'(s)}{\tau},\quad N(s,s') = \frac{L(s,s')-X'(s)}{\tau}.
\label{eqn: definition of L M N}
\end{equation}
and
\begin{equation}
L(s,s) = X'(s),\quad M(s,s) = X''(s),\quad N(s,s) =\frac{1}{2}X''(s).
\label{eqn: definition of L M N at s}
\end{equation}
It is straightforward to calculate that for $s'\not = s$,
\begin{equation}
\partial_s L(s,s') = N(s,s'),\quad \partial_s M(s,s') = \frac{M(s,s')-X''(s)}{\tau},\quad \partial_s N(s,s') = \frac{2N(s,s')-X''(s)}{\tau}.
\label{eqn: derivatives of L M N wrt s}
\end{equation}
In the sequel, we shall omit the arguments in $L(s,s')$, $M(s,s')$ and $N(s,s')$ whenever it is convenient.
Without assuming the well-stretched assumption \eqref{eqn: well_stretched assumption}, we have the following estimates for $L$, $M$ and $N$, which will be building blocks of more complicated estimates in Section \ref{section: a priori estimates of the immersed boundary problem}.

\begin{lemma}\label{lemma: estimates for L M N}
\begin{enumerate}
\item For $\forall\, 1\leq p\leq q \leq \infty$, $q>1$ and any interval $I\subset\mathbb{T}$ satisfying $0\in I$
\begin{align}
\|L(s,\cdot)\|_{L^p(s+I)} \leq &\;C|I|^{\frac{1}{p}-\frac{1}{q}}\|X'\|_{L^q(s+I)},\label{eqn: Lp estimate for L}\\
\|M(s,\cdot)\|_{L^p(s+I)} \leq &\;C|I|^{\frac{1}{p}-\frac{1}{q}}\|X''\|_{L^q(s+I)},\label{eqn: Lp estimate for M}\\
\|N(s,\cdot)\|_{L^p(s+I)} \leq &\;C|I|^{\frac{1}{p}-\frac{1}{q}}\|X''\|_{L^q(s+I)},\label{eqn: Lp estimate for N}\\
\|\partial_s M(s,\cdot)\|_{L^p(s+I)} \leq &\;C|I|^{\frac{1}{p}-\frac{1}{q}}\|X'''\|_{L^q(s+I)},\label{eqn: Lp estimate for M'}\\
\|\partial_s N(s,\cdot)\|_{L^p(s+I)} \leq &\;C|I|^{\frac{1}{p}-\frac{1}{q}}\|X'''\|_{L^q(s+I)},\label{eqn: Lp estimate for N'}
\end{align}
where the constants $C>0$ only depend on $p$ and $q$.
\item For $\forall\, 1< p\leq q \leq \infty$ and any interval $I\subset\mathbb{T}$ satisfying $0\in I$
\begin{align}
\|L(s,s')\|_{L^q_{s}(\mathbb{T})L^p_{s'}(s+I)} \leq &\;C|I|^{1/q}\|X'\|_{L^p(\mathbb{T})},\label{eqn: double Lp estimate for L}\\
\|M(s,s')\|_{L^q_{s}(\mathbb{T})L^p_{s'}(s+I)} \leq &\;C|I|^{1/q}\|X''\|_{L^p(\mathbb{T})},\label{eqn: double Lp estimate for M}\\
\|N(s,s')\|_{L^q_{s}(\mathbb{T})L^p_{s'}(s+I)} \leq &\;C|I|^{1/q}\|X''\|_{L^p(\mathbb{T})},\label{eqn: double Lp estimate for N}\\
\|\partial_s M(s,s')\|_{L^q_{s}(\mathbb{T})L^p_{s'}(s+I)} \leq &\;C|I|^{1/q}\|X'''\|_{L^p(\mathbb{T})},\label{eqn: double Lp estimate for M'}\\
\|\partial_s N(s,s')\|_{L^q_{s}(\mathbb{T})L^p_{s'}(s+I)} \leq &\;C|I|^{1/q}\|X'''\|_{L^p(\mathbb{T})},\label{eqn: double Lp estimate for N'}
\end{align}
where the constants $C>0$ only depend on $p$ and $q$.
\item Let $\mathcal{M}$ be the centered Hardy-Littlewood maximal operator on $\mathbb{T}$. Then for $\forall\, s,s'\in\mathbb{T}$,
\begin{equation}
|L(s,s')|\leq 2\mathcal{M} X'(s),\quad |M(s,s')|\leq 2\mathcal{M} X''(s),\quad |N(s,s')|\leq 2\mathcal{M} X''(s).\label{eqn: bound for L M N by maximal function}
\end{equation}
\item If $X\in C^2(\mathbb{T})$,
\begin{equation}
L(s,\cdot),M(s,\cdot),N(s,\cdot)\in C(\mathbb{T}).
\label{eqn: continuity of L M N}
\end{equation}
\item Moreover, if \eqref{eqn: well_stretched assumption} is satisfied with constant $\lambda>0$,
\begin{equation}
\lambda\leq |L(s,s')|\leq \|X'\|_{L^\infty},
\label{eqn: lower bound for L}
\end{equation}
and
\begin{equation}
\lambda \leq \min_{s\in\mathbb{T}}|X'(s)|.
\label{eqn: upper bound for lambda}
\end{equation}

\end{enumerate}

\begin{proof}
\eqref{eqn: lower bound for L} and \eqref{eqn: upper bound for lambda} are obvious.
To prove the $L^p$-estimates and the continuity of $L$, $M$ and $N$, we rewrite
\begin{align*}
&\;L(s,s') =\frac{1}{\tau} \int_0^{\tau} X'(s+\theta)\,d\theta = \int_0^1  X'(s+\tau\theta)\,d\theta,\\
&\;M(s,s') =\frac{1}{\tau} \int_0^{\tau} X''(s+\theta)\,d\theta = \int_0^1  X''(s+\tau\theta)\,d\theta,
\end{align*}
\begin{equation*}
\begin{split}
N(s,s') =&\;\frac{1}{\tau^2} \int_0^{\tau} (X'(s+\theta)-X'(s))\,d\theta = \frac{1}{\tau^2} \int_0^{\tau} \int_0^{\theta}  X''(s+\omega)\,d\omega d\theta\\
=&\;\frac{1}{\tau^2} \int_0^\tau \theta\int_0^1  X''(s+\theta\omega)\,d\omega d\theta=\int_0^1 \theta\int_0^1  X''(s+\tau\theta\omega)\,d\omega d\theta,
\end{split}
\end{equation*}
\begin{equation*}
\begin{split}
\partial_s M(s,s') =&\;\frac{1}{\tau^2} (X'(s')-X'(s)-\tau X''(s)) = \frac{1}{\tau^2} \int_{0}^{\tau} X''(s+\theta)-X''(s)\,d\theta\\
=&\;\frac{1}{\tau^2} \int_{0}^{\tau} \int_{0}^\theta X'''(s+\omega)\,d\omega d\theta = \int_0^1 \theta\int_0^1  X'''(s+\tau\theta\omega)\,d\omega d\theta,
\end{split}
\end{equation*}
and
\begin{equation*}
\begin{split}
\partial_s N(s,s') =&\;\frac{2}{\tau^3} \left(X(s')-X(s)-\tau X'(s)-\frac{1}{2}\tau^2 X''(s)\right)\\
=&\;\frac{2}{\tau^3}\left(\int_0^{\tau} X'(s+\theta)\,d\theta-\tau X'(s)-\frac{1}{2}\tau^2 X''(s)\right)\\
=&\;\frac{2}{\tau^3}\left(\int_0^{\tau} \int_0^{\theta} X''(s+\omega)\,d\omega d\theta-\frac{1}{2}\tau^2 X''(s)\right)\\
=&\;\frac{2}{\tau^3}\int_0^{\tau} \int_0^{\theta} \int_0^{\omega} X'''(s+\xi)\,d\xi d\omega d\theta\\
=&\;2\int_0^{1} \theta^2 \int_0^{1} \omega \int_0^{1} X'''(s+\tau\theta\omega\xi)\,d\xi d\omega d\theta.
\end{split}
\end{equation*}
\eqref{eqn: continuity of L M N} is immediate by the continuity of $X'$ and $X''$ at $s$.
To prove \eqref{eqn: bound for L M N by maximal function}, we use the above representation to derive that
\begin{align*}
|L(s,s')| \leq &\;\frac{1}{\tau} \int_0^{\tau} |X'(s+\theta)|\,d\theta \leq \frac{1}{\tau} \int_{-\tau}^{\tau} |X'(s+\theta)|\,d\theta \leq 2\mathcal{M}X'(s),\\
|M(s,s')| \leq &\;\frac{1}{\tau} \int_0^{\tau} |X''(s+\theta)|\,d\theta \leq \frac{1}{\tau} \int_{-\tau}^{\tau} |X''(s+\theta)|\,d\theta \leq 2\mathcal{M}X''(s),\\
|N(s,s')| \leq &\;\frac{1}{\tau^2} \int_0^{\tau} \int_0^{\theta}  |X''(s+\omega)|\,d\omega d\theta \leq \frac{1}{\tau} \int_0^{\tau}  |X''(s+\omega)|\,d\omega \leq 2\mathcal{M}X''(s).
\end{align*}

Now we turn to \eqref{eqn: Lp estimate for L}-
\eqref{eqn: double Lp estimate for N'}.
When $p = q =\infty$, \eqref{eqn: Lp estimate for L}-
\eqref{eqn: double Lp estimate for N'} 
immediately follow from the above representations.
When $1\leq p\leq q \leq \infty$, $p<\infty$ and $q>1$, we find that
\begin{equation*}
\begin{split}
\|L(s,\cdot)\|_{L^p(s+I)} = &\;\left(\int_{I} d\tau\left|\int_0^1  X'(s+\tau\theta)\,d\theta\right|^{p}\right)^{\frac{1}{p}}\\
\leq &\; C\int_0^1\left(\int_{I} d\tau\left|  X'(s+\tau\theta)\right|^{p}\right)^{\frac{1}{p}}\,d\theta\\
= &\; C\int_0^1\theta^{-\frac{1}{p}}\left(\int_{s+\theta I} ds'\left|  X'(s')\right|^{p}\right)^{\frac{1}{p}}\,d\theta\\
\leq &\; C\int_0^1\theta^{-\frac{1}{p}}|\theta I|^{\frac{1}{p}-\frac{1}{q}}\|X'\|_{L^q(s+I)}\,d\theta \leq C|I|^{\frac{1}{p}-\frac{1}{q}}\|X'\|_{L^q(s+I)}.
\end{split}
\end{equation*}
We applied Minkowski inequality in the second line and H$\mathrm{\ddot{o}}$lder's inequality in the fourth line; we also used the fact that $s+\theta I \subset s+I$.
This proves \eqref{eqn: Lp estimate for L}; \eqref{eqn: Lp estimate for M} could be proved in exactly the same way simply by replacing $X''$ by $X'''$. For \eqref{eqn: Lp estimate for N},
\begin{equation*}
\begin{split}
\|N(s,\cdot)\|_{L^p(s+I)} =&\; \left(\int_I d\tau\left|\int_0^1 \theta\int_0^1  X''(s+\tau\theta\omega)\,d\omega d\theta\right|^p\right)^{\frac{1}{p}}\\
\leq &\; C\int_0^1 \theta\int_0^1  \left(\int_I d\tau\left|X''(s+\tau\theta\omega)\right|^p\right)^\frac{1}{p} \,d\omega d\theta\\
= &\; C\int_0^1 \theta\int_0^1  \left(\frac{1}{\theta \omega}\int_{s+\theta \omega I} ds'\left|X''(s')\right|^p\right)^{\frac{1}{p}}\,d\omega d\theta\\
\leq &\; C\int_0^1 \int_0^1  \frac{\theta^{1-\frac{1}{p}}}{ \omega^{\frac{1}{p}}}|\theta\omega I|^{\frac{1}{p}-\frac{1}{q}}\|X''\|_{L^q(s+I)}\,d\omega d\theta\leq C|I|^{\frac{1}{p}-\frac{1}{q}}\|X''\|_{L^q(s+I)}.
\end{split}
\end{equation*}
\eqref{eqn: Lp estimate for M'} could be proved in exactly the same way simply by replacing $X''$ by $X'''$.
For \eqref{eqn: Lp estimate for N'},
\begin{equation*}
\begin{split}
\|\partial_s N(s,\cdot)\|_{L^p(s+I)} =&\; \left(\int_I d\tau\left|2\int_0^{1} \theta^2 \int_0^{1} \omega \int_0^{1} X'''(s+\tau\theta\omega\xi)\,d\xi d\omega d\theta\right|^p\right)^{\frac{1}{p}}\\
\leq &\; C\int_0^{1} \theta^2 \int_0^{1} \omega \int_0^{1} \left(\int_I d\tau|X'''(s+\tau\theta\omega\xi)|^p\right)^{\frac{1}{p}}\,d\xi d\omega d\theta\\
= &\; C\int_0^{1} \theta^2 \int_0^{1} \omega \int_0^{1} (\theta\omega\xi)^{-\frac{1}{p}}\left(\int_{s+\theta\omega\xi I} ds'|X'''(s')|^p\right)^{\frac{1}{p}}\,d\xi d\omega d\theta\\
\leq &\; C\int_0^{1} \theta^2 \int_0^{1} \omega \int_0^{1} (\theta\omega\xi)^{-\frac{1}{p}}|\theta\omega\xi I|^{\frac{1}{p}-\frac{1}{q}}\|X'''\|_{L^q(s+\theta\omega\xi I)}\,d\xi d\omega d\theta\\
\leq &\; C|I|^{\frac{1}{p}-\frac{1}{q}}\int_0^{1} \theta^2 \int_0^{1} \omega \int_0^{1} (\theta\omega\xi)^{-\frac{1}{q}}\|X'''\|_{L^q(s+I)}\,d\xi d\omega d\theta\\
\leq &\; C|I|^{\frac{1}{p}-\frac{1}{q}}\|X'''\|_{L^q(s+I)}.
\end{split}
\end{equation*}

For \eqref{eqn: double Lp estimate for L}, we first consider the case $p=q\in(1,\infty)$. \eqref{eqn: Lp estimate for L} implies that, $\|L(s,s')\|_{L^p_{s'}(s+I)}\leq C\|X'\|_{L^p(s+I)}$. Hence, by Fubini's Theorem,
\begin{equation*}
\|L(s,s')\|_{L^{p}_{s}(\mathbb{T})L^p_{s'}(s+I)}\leq C\left(\int_{\mathbb{T}}\|X'\|^p_{L^p(s+I)}\,ds\right)^{1/p}\leq C|I|^{1/p}\|X'\|_{L^p(\mathbb{T})}
\end{equation*}
On the other hand, by \eqref{eqn: Lp estimate for L}, $\|L(s,s')\|_{L^{\infty}_{s}(\mathbb{T})L^p_{s'}(s+I)}\leq C\|X'\|_{L^p(\mathbb{T})}$. Hence, by interpolation between $L^p$-spaces, we proved \eqref{eqn: double Lp estimate for L}. In a similar manner, we can prove \eqref{eqn: double Lp estimate for M}-\eqref{eqn: double Lp estimate for N'}.
\end{proof}
\end{lemma}

\subsection{$H^2$-estimate of $g_X$}\label{section: a priori estimates of the immersed boundary problem}

%
%

In Section \ref{section: local existence and uniqueness}, we will prove well-posedness of \eqref{eqn: contour dynamic formulation of the immersed boundary problem} via a fixed-point-type argument by making use of dissipation structure of the operator $\mathcal{L}$ (see Lemma \ref{lemma: improved Hs estimate and Hs continuity of semigroup solution} and Lemma \ref{lemma: a priori estimate of nonlocal eqn} in the Appendix \ref{appendix section: estimates involving L}).
In order to do that, in this section, we focus on the term $g_X$ in \eqref{eqn: contour dynamic formulation of the immersed boundary problem} and establish its $H^2$-estimate; recall that $g_X$ is defined in \eqref{eqn: definition of g_X}.
We are also going to prove an $H^2$-estimate of $g_{X_1}-g_{X_2}$, which will be used in proving the uniqueness of the local solution.

We start from a pointwise estimate of $g_X$.
\begin{lemma}\label{lemma: L infty estimate for g_X}
Suppose $X\in H^2(\mathbb{T})$ satisfies \eqref{eqn: well_stretched assumption} with some $\lambda>0$. Then
\begin{equation}
|g_X(s)|\leq \frac{C}{\lambda}\|X'\|_{L^2}\|X''\|_{L^2},
\label{eqn: L infty estimate for g_X}
\end{equation}
where $C>0$ is a universal constant.
\begin{proof}
Recall that $\Gamma_0(s,s')$ is defined in \eqref{eqn: introduce the notation Gamma_0}. By \eqref{eqn: velocity of membrane}, and the definitions of $L$ and $M$, we have
\begin{equation}
\begin{split}
\Gamma_0(s,s')=&\; \frac{1}{4\pi}\left(\frac{L\cdot X'(s')}{|L|^2}Id-\frac{X'(s')\otimes L + L\otimes X'(s')}{|L|^2}+\frac{2L\cdot X'(s')L\otimes L}{|L|^4}\right)M\\
=&\;\frac{1}{4\pi}\left(\frac{L\cdot X'(s')}{|L|^2}M-\frac{L\cdot M}{|L|^2}X'(s') -\frac{X'(s')\cdot M}{|L|^2}L+\frac{2L\cdot X'(s')L\cdot M}{|L|^4}L\right).
\end{split}
\label{eqn: simplification of integrand of g_X part 1}
\end{equation}
Hence, by \eqref{eqn: lower bound for L},
\begin{equation}
|\Gamma_0(s,s')|\leq C \frac{|M(s,s')||X'(s')|}{|L(s,s')|} \leq \frac{C}{\lambda} |M(s,s')||X'(s')|.
\label{eqn: pointwise estimate of integrand of g_X part 1}
\end{equation}
This implies by H$\mathrm{\ddot{o}}$lder's inequality and Lemma \ref{lemma: estimates for L M N} that
\begin{equation}
\left|\int_{\mathbb{T}} \Gamma_0(s,s')\,ds'\right| \leq \frac{C}{\lambda} \|X'\|_{L^2(\mathbb{T})}\|X''\|_{L^2(\mathbb{T})}.
\label{eqn: L infty estimate for g_X part 1}
\end{equation}
The other term in $g_X(s)$, $(-\Delta)^{1/2}X$, has mean zero on $\mathbb{T}$. By Gagliardo-Nirenberg interpolation inequality,
\begin{equation*}
\left|(-\Delta)^{1/2}X(s)\right|\leq C \left\|(-\Delta)^{1/2}X(s)\right\|_{\dot{H}^1}^{1/2}\left\|(-\Delta)^{1/2}X(s)\right\|_{L^2}^{1/2}\leq C \|X''\|_{L^2}^{1/2}\|X'\|_{L^2}^{1/2}.
\end{equation*}
Using \eqref{eqn: lower bound for L}, we find that
\begin{equation}
\left|(-\Delta)^{1/2}X(s)\right| \leq C\frac{\|X'\|_{L^\infty}}{\lambda}\|X''\|^{1/2}_{L^2}\|X'\|_{L^2}^{1/2}\leq \frac{C}{\lambda}\|X''\|_{L^2}\|X'\|_{L^2}
\label{eqn: L infty estimate for g_X part 2}
\end{equation}
\eqref{eqn: L infty estimate for g_X} is then proved by \eqref{eqn: L infty estimate for g_X part 1} and \eqref{eqn: L infty estimate for g_X part 2}.
\end{proof}
\begin{remark}
If we further assume $X\in H^3(\mathbb{T})\subset C^2(\mathbb{T})$, using the continuity of $L(s,\cdot)$ and $M(s,\cdot)$, it is not difficult to show in \eqref{eqn: simplification of integrand of g_X part 1} that
\begin{equation}
\lim_{s'\rightarrow s} \Gamma_0(s,s') = \frac{1}{4\pi} X''(s).
\label{eqn: limit of integrand of g_X part 1 at s}
\end{equation}
This will be useful below in proving Lemma \ref{lemma: derivative of g_X}.
\qed
\end{remark}
\end{lemma}

\begin{corollary}\label{coro: L2 estimate for g_X1-g_X2}
Let $X_1(s),X_2(s)\in H^2(\mathbb{T})$ both satisfy \eqref{eqn: well_stretched assumption} with some $\lambda>0$. Then
\begin{equation}
\|g_{X_1}(s)-g_{X_2}(s)\|_{L^2}\leq C\lambda^{-2} (\|X_1\|_{\dot{H}^2}+\|X_2\|_{\dot{H}^2})^2\|X_1-X_2\|_{\dot{H}^2},
\label{eqn: L2 estimate for g_X1-g_X2}
\end{equation}
where $C>0$ is a universal constant.
\begin{proof}
By the definition of $g_X$ in \eqref{eqn: definition of g_X} and \eqref{eqn: simplification of integrand of g_X part 1},
\begin{equation}
\begin{split}
&\;g_{X_1}(s)-g_{X_2}(s) \\
= &\;\int_\mathbb{T}ds'\,\frac{1}{4\pi}\left(\frac{L_1\cdot X_1'(s')}{|L_1|^2}M_1-\frac{L_1\cdot M_1}{|L_1|^2}X_1'(s') -\frac{X_1'(s')\cdot M_1}{|L_1|^2}L_1+\frac{2L_1\cdot X_1'(s')L_1\cdot M_1}{|L_1|^4}L_1\right)\\
&\;-\int_\mathbb{T}ds'\,\frac{1}{4\pi}\left(\frac{L_2\cdot X_2'(s')}{|L_2|^2}M_2-\frac{L_2\cdot M_2}{|L_2|^2}X_2'(s') -\frac{X_2'(s')\cdot M_2}{|L_2|^2}L_2+\frac{2L_2\cdot X_2'(s')L_2\cdot M_2}{|L_2|^4}L_2\right)\\
&\;-\mathcal{L}X_1(s)+\mathcal{L}X_2(s).
\end{split}
\label{eqn: difference of X_t at two moments}
\end{equation}
where $L_i$, $M_i$ and $X_i'$ denote the corresponding quantities associated with $X_i(\cdot)$; see definitions in \eqref{eqn: definition of L M N} and \eqref{eqn: definition of L M N at s}.
To make an $L^2$-estimate, for conciseness, we only consider a part of the difference above.
By \eqref{eqn: well_stretched assumption} and \eqref{eqn: lower bound for L},
\begin{equation*}
\begin{split}
&\;\left\|\int_\mathbb{T}ds'\,\frac{L_1\cdot X_1'(s')}{|L_1|^2}M_1 - \frac{L_2\cdot X_2'(s')}{|L_2|^2}M_2\right\|_{L^2}\\
\leq &\;\left\|\frac{L_1\cdot (X_1'-X_2')(s')}{|L_1|^2}M_1\right\|_{L^2_sL_{s'}^1}+\left\|\frac{L_1\cdot X_2'(s')}{|L_1|^2}(M_1-M_2)\right\|_{L^2_sL_{s'}^1}\\
&\;+\left\|\frac{(L_1-L_2)\cdot X_2'(s')}{|L_1|^2}M_2\right\|_{L^2_sL_{s'}^1}+\left\|L_2\cdot X_2'(s')M_2\frac{|L_2|^2-|L_1|^2}{|L_1|^2|L_2|^2}\right\|_{L^2_sL_{s'}^1}\\
\leq &\;C\lambda^{-2}\left(\|L_1\|_{L^\infty_s L^\infty_{s'}}\|X_1'-X_2'\|_{L^2}\|M_1\|_{L^2_s L^2_{s'}}+\|L_1\|_{L^4_s L^2_{s'}}\|X_2'\|_{L^\infty}\|M_1-M_2\|_{L^4_s L^2_{s'}}\right.\\
&\;\left.+\|L_1-L_2\|_{L^4_s L^2_{s'}}\|X_2'\|_{L^\infty}\|M_2\|_{L^4_s L^2_{s'}}\right).
\end{split}
\end{equation*}
By Lemma \ref{lemma: estimates for L M N} and Sobolev inequality,
\begin{equation*}
\begin{split}
&\;\left\|\int_\mathbb{T}ds'\,\frac{L_1\cdot X_1'(s')}{|L_1|^2}M_1 - \frac{L_2\cdot X_2'(s')}{|L_2|^2}M_2\right\|_{L^2}\\
\leq &\;C\lambda^{-2}\left(\|X_1'\|_{L^\infty}\|X_1'-X_2'\|_{L^2}\|X_1''\|_{L^2}+\|X_1'\|_{L^2}\|X_2'\|_{L^\infty}\|X_1''-X_2''\|_{L^2}\right.\\
&\;\left.+\|X_1'-X_2'\|_{L^2}\|X_2'\|_{L^\infty}\|X_2''\|_{L^2}\right)\\
\leq &\; C\lambda^{-2} (\|X_1\|_{\dot{H}^2}+\|X_2\|_{\dot{H}^2})^2\|X_1-X_2\|_{\dot{H}^2}.
\end{split}
\end{equation*}
Similarly,
\begin{equation*}
\begin{split}
&\;\left\|\int_\mathbb{T}ds'\,\frac{L_1\cdot X_1'(s')L_1\cdot M_1}{|L_1|^4}L_1 -\frac{L_2\cdot X_2'(s')L_2\cdot M_2}{|L_2|^4}L_2\right\|_{L^2}\\
\leq &\;\left\|\frac{L_1\cdot (X_1'(s')-X_2'(s'))L_1\cdot M_1}{|L_1|^4}L_1\right\|_{L^2_sL^1_{s'}}+\left\|\frac{L_1\cdot X_2'(s')L_1\cdot (M_1-M_2)}{|L_1|^4}L_1\right\|_{L^2_sL^1_{s'}}\\
&\;+\left\|\frac{(L_1-L_2)\cdot X_2'(s')L_1\cdot M_2}{|L_1|^4}L_1\right\|_{L^2_sL^1_{s'}}+\left\|\frac{L_2\cdot X_2'(s')L_1\cdot M_2}{|L_1|^2}L_1\frac{|L_2|^2-|L_1|^2}{|L_1|^2|L_2|^2}\right\|_{L^2_sL^1_{s'}}\\
&\;+\left\|\frac{L_2\cdot X_2'(s')(L_1-L_2)\cdot M_2}{|L_1|^2|L_2|^2}L_1\right\|_{L^2_sL^1_{s'}}+\left\|\frac{L_2\cdot X_2'(s')L_2\cdot M_2}{|L_1|^2|L_2|^2}(L_1-L_2)\right\|_{L^2_sL^1_{s'}}\\
&\;+\left\|\frac{L_2\cdot X_2'(s')L_2\cdot M_2}{|L_2|^2}L_2\frac{|L_2|^2-|L_1|^2}{|L_1|^2|L_2|^2}\right\|_{L^2_sL^1_{s'}}\\
\leq &\;C\lambda^{-2}\left(\|X_1'-X_2'\|_{L^2}\|L_1\|_{L^\infty_s L^\infty_{s'}}\| M_1\|_{L^2_sL^2_{s'}}+\|X_2'\|_{L^\infty}\|L_1\|_{L^4_s L^2_{s'}}\| M_1-M_2\|_{L^4_sL^2_{s'}}\right.\\
&\;+\left.\|X_2'\|_{L^\infty}\|M_2\|_{L^4_sL^2_{s'}} \|L_2-L_1\|_{L^4_sL^2_{s'}}\right)\\
\leq &\;C\lambda^{-2}\left(\|X_1'-X_2'\|_{L^2}\|X_1'\|_{L^\infty}\|X_1''\|_{L^2}+\|X_2'\|_{L^\infty}\|X_1'\|_{L^2}\| X_1''-X_2''\|_{L^2}\right.\\
&\;+\left.\|X_2'\|_{L^\infty}\|X_2''\|_{L^2} \|X_2'-X_1'\|_{L^2}\right)\\ \leq&\;C\lambda^{-2}(\|X_1\|_{\dot{H}^2}+\|X_2\|_{\dot{H}^2})^2\|X_1-X_2\|_{\dot{H}^2}.
\end{split}
\end{equation*}
We can estimate the other terms in \eqref{eqn: difference of X_t at two moments} in a similar fashion and obtain \eqref{eqn: L2 estimate for g_X1-g_X2}.
\end{proof}
\end{corollary}

In order to estimate $H^2$-norm of $g_X$, we find out its weak derivatives $g_X'$ and $g_X''$ in the following two lemmas.
\begin{lemma}\label{lemma: derivative of g_X}
Suppose $X\in H^3(\mathbb{T})$ and satisfies \eqref{eqn: well_stretched assumption} with some $\lambda>0$. Then
\begin{equation}
g'_X(s) = \mathrm{p.v.}\int_\mathbb{T}\left(-\partial_{ss'}[G(X(s)-X(s'))]-\frac{Id}{16\pi\sin^2\left(\frac{s'-s}{2}\right)}\right)(X'(s')-X'(s))\,ds'.\label{eqn: derivative of g_X}
\end{equation}

\begin{proof}
We define a cut-off function $\varphi(y)\in C^\infty (\mathbb{T})$ such that
\begin{enumerate}
\item $\varphi(y)= \varphi(-y)$, $\forall\,y\in\mathbb{T}$.
\item $\varphi(y)= 1$ for $|y|\leq 1$; $\varphi(y)= 0$ for $|y|\geq 2$; and $|\varphi'(y)|\leq C$.
\item $\varphi(y)$ is decreasing on $[0,\pi]$ and increasing on $[-\pi,0]$.
\end{enumerate}
Define $\psi_\varepsilon(y) = 1-\varphi\left(\frac{y}{\varepsilon}\right)$.
Let
\begin{equation*}
g_{X,1}(s) = \int_{\mathbb{T}} \Gamma_0(s,s')\,ds', \quad g_{X,1}^\varepsilon(s) = \int_{\mathbb{T}} \Gamma_0(s,s')\psi_\varepsilon(s'-s)\,ds'.
\end{equation*}
By \eqref{eqn: pointwise estimate of integrand of g_X part 1} and Lemma \ref{lemma: estimates for L M N},
\begin{equation}
|\Gamma_0(s,s')|\leq \frac{C}{\lambda} \|X''\|_{L^\infty}\|X'\|_{L^\infty},
\label{eqn: L infty estimate of integrand of g_X part 1}
\end{equation}
which implies that $g_{X,1},g_{X,1}^\varepsilon\in L^\infty(\mathbb{T})$, and $g_{X,1}^\varepsilon \rightarrow g_{X,1}$ in $L^\infty(\mathbb{T})$. In particular, for any test function $\eta\in C^\infty(\mathbb{T})$,
\begin{equation}
\lim_{\varepsilon\rightarrow 0}(\eta',g_{X,1}^\varepsilon) = (\eta',g_{X,1}),
\label{eqn: derivative of g_X1 test function convergence}
\end{equation}
where $(\cdot,\cdot)$ is the $L^2$-inner product on $\mathbb{T}$. Since there is no singularity in the integral in $g_{X,1}^\varepsilon$, we apply integration by parts on the left hand side above and exchange the derivative and the integral. We will obtain
\begin{equation}
\begin{split}
(\eta',g_{X,1}^\varepsilon) = &\;-(\eta, \partial_s g_{X,1}^\varepsilon)\\
=&\;-\left(\eta, \int_{\mathbb{T}} \partial_s\Gamma_0(s,s')\psi_\varepsilon(s'-s)\,ds'\right)+\left(\eta, \int_{\mathbb{T}} \Gamma_0(s,s')\psi'_\varepsilon(s'-s)\,ds'\right)\\
\triangleq &\; I_\varepsilon+II_\varepsilon
\end{split}
\label{eqn: derivative of g_X1 integration by parts}
\end{equation}
It is not difficult to show that
\begin{equation}
\lim_{\varepsilon\rightarrow 0}I_\varepsilon = -\left(\eta, \mathrm{p.v.}\int_{\mathbb{T}} \partial_s\Gamma_0(s,s')\,ds'\right).
\label{eqn: derivative of g_X1 term 1}
\end{equation}
On the other hand, since $\psi_\varepsilon'(\cdot-s)$ is of mean zero on $\mathbb{T}$ and $\|\psi_\varepsilon'(\cdot-s)\|_{L^1(\mathbb{T})} = 2$ due to the monotonicity assumption on $\varphi$, we have that
\begin{equation}
|II_{\varepsilon}|\leq 2\|\eta\|_{L^1} \mathrm{osc}_{s'\in [s-2\varepsilon, s+2\varepsilon]} \Gamma_0(s,s')\rightarrow 0,\quad \mbox{as }\varepsilon\rightarrow 0,
\label{eqn: derivative of g_X1 term 2}
\end{equation}
where the convergence comes from \eqref{eqn: limit of integrand of g_X part 1 at s}. Combining \eqref{eqn: derivative of g_X1 test function convergence}, \eqref{eqn: derivative of g_X1 integration by parts}, \eqref{eqn: derivative of g_X1 term 1} and \eqref{eqn: derivative of g_X1 term 2}, we find
\begin{equation}
\begin{split}
g'_{X,1}(s) = &\;\mathrm{p.v.}\int_{\mathbb{T}} \partial_s\Gamma_0(s,s')\,ds'\\
= &\;\mathrm{p.v.}\int_{\mathbb{T}} -\partial_{ss'}[G(X(s)-X(s'))](X'(s')-X'(s))\,ds'\\
&\;+ \mathrm{p.v.}\int_{\mathbb{T}} \partial_{s'}[G(X(s)-X(s'))]X''(s)\,ds'\\
= &\;\mathrm{p.v.}\int_{\mathbb{T}} -\partial_{ss'}[G(X(s)-X(s'))](X'(s')-X'(s))\,ds'.
\label{eqn: derivative of g_X part 1}
\end{split}
\end{equation}
We used \eqref{eqn: pv integral vanishes} in the last line.

For the other term in $g_X(s)$, namely $\frac{1}{4}(-\Delta)^{1/2}X$, we note that $(-\Delta)^{1/2}$ and the derivative commute since they are both Fourier multipliers. This gives
\begin{equation}
\partial_s\left(\frac{1}{4}(-\Delta)^{1/2}X\right) = \frac{1}{4}(-\Delta)^{1/2}X' = -\frac{1}{4\pi}\mathrm{p.v.}\int_\mathbb{T} \frac{X'(s')-X'(s)}{4\sin^2\left(\frac{s'-s}{2}\right)}\,ds'.
\label{eqn: derivative of g_X part 2}
\end{equation}
Combining \eqref{eqn: derivative of g_X part 1} and \eqref{eqn: derivative of g_X part 2}, we proved \eqref{eqn: derivative of g_X}.

\end{proof}
\end{lemma}

\begin{lemma}\label{lemma: second derivative of g_X}
Suppose $X\in H^3(\mathbb{T})$ and satisfies \eqref{eqn: well_stretched assumption} with some $\lambda>0$. Then
\begin{equation}
g''_X(s) = \mathrm{p.v.}\int_\mathbb{T}\partial_s\left[\left(-\partial_{ss'}[G(X(s)-X(s'))]-\frac{Id}{16\pi\sin^2\left(\frac{s'-s}{2}\right)}\right)(X'(s')-X'(s))\right]\,ds'.\label{eqn: second derivative of g_X}
\end{equation}
\begin{proof}
Denote the integrand of \eqref{eqn: derivative of g_X} by $\Gamma_1(s,s')$, i.e.
\begin{equation}
g'_X(s) = \mathrm{p.v.}\int_\mathbb{T} \Gamma_1(s,s')\,ds'.
\label{eqn: introduce the notation Gamma_1}
\end{equation}
What we are going to show in \eqref{eqn: second derivative of g_X} is exactly
\begin{equation*}
g''_X(s) = \mathrm{p.v.}\int_\mathbb{T} \partial_s\Gamma_1(s,s')\,ds'.
\end{equation*}
We claim that for $s\not = s'$,
\begin{equation}
\begin{split}
4\pi\Gamma_1(s,s') = &\;\frac{(X'(s)-L)\cdot N}{|L|^2}M - \frac{2(N\cdot L)(X'(s)\cdot L)}{|L|^4}M - \left(\frac{\tau^2 - 4\sin^2(\frac{\tau}{2})}{4\tau\sin^2(\frac{\tau}{2})}\right)M\\
&\;+\frac{(M-2N)\cdot M}{|L|^2}X'(s)+\frac{2(N\cdot L)( L\cdot M)}{|L|^4}X'(s)\\
&\; +\frac{2 (L\cdot M) (L\cdot (M-N)) (L\cdot X'(s))}{|L|^6}L+\frac{2 ((N-M)\cdot M)(L\cdot X'(s))}{|L|^4}L\\
&\;-\frac{6 (L\cdot M) (L\cdot X'(s')) (L\cdot N)}{|L|^6}L+\frac{2 (L\cdot M) (L\cdot X'(s'))}{|L|^4} N\\
&\;+\frac{2 (N\cdot M) (L\cdot X'(s'))}{|L|^4}L+\frac{2 (L\cdot M) (N\cdot X'(s'))}{|L|^4}L.
\end{split}
\label{eqn: simplified Gamma order 1}
\end{equation}
For conciseness, we leave its proof in Lemma \ref{lemma: simplification of Gamma_1(s,s')} in Appendix \ref{appendix section: auxiliary calculations}. With \eqref{eqn: simplified Gamma order 1} in hand, we use \eqref{eqn: lower bound for L} and \eqref{eqn: upper bound for lambda} to derive that,
\begin{equation}
\begin{split}
|\Gamma_1(s,s')| \leq &\;C \left(\frac{|X'(s)|+|X'(s')|}{\lambda^2}+\frac{1}{\lambda}\right)|M|(|M|+|N|)+C|\tau| |M|\\
\leq &\;C \frac{|X'(s)|+|X'(s')|}{\lambda^2}|M|(|M|+|N|)+C|\tau| |M|\\
\leq&\; \frac{C}{\lambda^2} \|X'\|_{L^\infty}\|X''\|^2_{L^\infty}.
\label{eqn: rough pointwise estimate of Gamma}
\end{split}
\end{equation}
By the continuity of $L$, $M$ and $N$, i.e.\;\eqref{eqn: continuity of L M N}, we also know by \eqref{eqn: simplified Gamma order 1} that
\begin{equation*}
\lim_{s'\rightarrow s}\Gamma_1(s,s') = 0.
\end{equation*}
To this end, we can prove \eqref{eqn: second derivative of g_X} by arguing as in the proof of Lemma \ref{lemma: derivative of g_X}. We omit the details.
\end{proof}
\end{lemma}

Similar to Corollary \ref{coro: L2 estimate for g_X1-g_X2}, one can prove that
\begin{corollary}\label{coro: H1 estimate for g_X1-g_X2}
Let $X_1(s),X_2(s)\in H^2(\mathbb{T})$ both satisfy \eqref{eqn: well_stretched assumption} with some $\lambda>0$. Then
\begin{equation}
\|g_{X_1}(s)-g_{X_2}(s)\|_{\dot{H}^1}\leq C\lambda^{-3} (\|X_1\|_{\dot{H}^2}+\|X_2\|_{\dot{H}^2})^3\|X_1-X_2\|_{\dot{H}^2},
\label{eqn: H1 estimate for g_X1-g_X2}
\end{equation}
where $C>0$ is a universal constant.
\begin{proof}
With \eqref{eqn: introduce the notation Gamma_1} and \eqref{eqn: simplified Gamma order 1} in hand, we simply argue as in the proof of Corollary \ref{coro: L2 estimate for g_X1-g_X2} to obtain the desired estimate.
We omit the details.
\end{proof}
\end{corollary}

The following lemma is devoted to $H^2$-estimate of $g_X$.
\begin{lemma}\label{lemma: H2 estimate of g_X}
Suppose $X\in H^3(\mathbb{T})$ and satisfies \eqref{eqn: well_stretched assumption} with some $\lambda>0$. Then for $\forall\,\delta\in(0,\pi)$, 
\begin{equation}
\begin{split}
\|g_X''\|_{L^2(\mathbb{T})} \leq 
&\;C\left(\delta^{1/2}\lambda^{-2}\|X\|_{\dot{H}^{3}}\|X\|_{\dot{H}^{5/2}}^2+(|\ln \delta|+1)\lambda^{-3}\|X\|_{\dot{H}^{5/2}}^4\right),
\end{split}
\label{eqn: H2 estimate of g_X}
\end{equation}
where $C>0$ is a universal constant.
\begin{proof}
By Lemma \ref{lemma: second derivative of g_X}, we look into $\partial_s \Gamma_1(s,s')$.
We take $s$-derivative of \eqref{eqn: simplified Gamma order 1} and use $\partial_s L = N$ in \eqref{eqn: derivatives of L M N wrt s} to find that
\begin{equation}
\begin{split}
|\partial_s \Gamma_1(s,s')| \leq &\;C\frac{|X'(s)|+|X'(s')|}{\lambda^3}|M||N|(|M|+|N|)+C \frac{|X'(s)|+|X'(s')|}{\lambda^2}|\partial_s M|(|M|+|N|)\\
&\;+C \frac{|X'(s)|+|X'(s')|}{\lambda^2}|M|(|\partial_s M|+|\partial_s N|)+C \frac{|X''(s)|}{\lambda^2}|M|(|M|+|N|)\\
&\;+C|M|+C|M-X''(s)|.
\end{split}
\label{eqn: pointwise estimate for s-derivative of Gamma}
\end{equation}
The following estimate is also useful by substituting \eqref{eqn: derivatives of L M N wrt s} into the above formula
\begin{equation}
\begin{split}
|\partial_s \Gamma_1(s,s')|\leq &\;C\frac{|X'(s)|+|X'(s')|}{\lambda^3}|M||N|(|M|+|N|)\\
&\;+C \frac{|X'(s)|+|X'(s')|}{\lambda^2}\frac{|M|+|X''(s)|}{|\tau|}(|M|+|N|)\\
&\;+C \frac{|X'(s)|+|X'(s')|}{\lambda^2}|M|\frac{|M|+|N|+|X''(s)|}{|\tau|}\\
&\;+C \frac{|X''(s)|}{\lambda^2}|M|(|M|+|N|)+C|M|+C|M-X''(s)|\\
\leq &\;C \lambda^{-3} (|X'(s)|+|X'(s')|)|M||N|(|M|+|N|)\\
&\;+C \lambda^{-2}|X''(s)||M|(|M|+|N|)+C|M|+C|X''(s)|\\
&\;+C \lambda^{-2}|\tau|^{-1}(|X'(s)|+|X'(s')|)(|M|+|X''(s)|)(|M|+|N|).
\end{split}
\label{eqn: pointwise estimate for s-derivative of Gamma far field}
\end{equation}
In order to prove \eqref{eqn: H2 estimate of g_X}, we split $g_X''$, an integral of $\partial_s\Gamma_1$ with respect to $s'$, into two terms --- the integral in a neighborhood of the singularity at $s'=s$, and the rest. To be more precise, for $\forall\,\delta\in(0,\pi)$, we have
\begin{equation}
\|g_X''\|_{L^2(\mathbb{T})} \leq \left\|\int_{B_\delta(s)}|\partial_s \Gamma_1(s,s')|\,ds'\right\|_{L^2(\mathbb{T})}+\left\|\int_{B^c_\delta(s)}|\partial_s \Gamma_1(s,s')|\,ds'\right\|_{L^2(\mathbb{T})} \triangleq I_\delta + II_\delta.
\label{eqn: splitting of g_X''}
\end{equation}
For $I_\delta$, we use \eqref{eqn: pointwise estimate for s-derivative of Gamma}. Applying Lemma \ref{lemma: estimates for L M N} with $I = B_\delta(0)$, we obtain that

\begin{equation*}
\begin{split}
I_\delta \leq &\;C\lambda^{-3}\left\|\int_{B_{\delta}(s)}(|X'(s)|+|X'(s')|)|M||N|(|M|+|N|)\,ds'\right\|_{L^2(\mathbb{T})}\\
&\;+C\lambda^{-2} \left\|\int_{B_{\delta}(s)}(|X'(s)|+|X'(s')|)|\partial_s M|(|M|+|N|)\,ds'\right\|_{L^2(\mathbb{T})}\\
&\;+C \lambda^{-2}\left\|\int_{B_{\delta}(s)}(|X'(s)|+|X'(s')|)|M|(|\partial_s M|+|\partial_s N|)\,ds'\right\|_{L^2(\mathbb{T})}\\
&\;+C \lambda^{-2}\left\|\int_{B_{\delta}(s)}|X''(s)||M|(|M|+|N|)\,ds'\right\|_{L^2(\mathbb{T})}+C\left\|\int_{B_{\delta}(s)}|M|+|X''(s)|\,ds'\right\|_{L^2(\mathbb{T})}\\
\leq &\;C\lambda^{-3}\||X'(s)|+|X'(s')|\|_{L^\infty_s(\mathbb{T})L^\infty_{s'}(B_\delta(s))}\|M\|_{L^6_s(\mathbb{T})L^3_{s'}(B_\delta(s))}\|N\|_{L^6_s(\mathbb{T})L^3_{s'}(B_\delta(s))}\\
&\;\quad\cdot\||M|+|N|\|_{L^6_s(\mathbb{T})L^3_{s'}(B_\delta(s))}\\
&\;+C\lambda^{-2} \||X'(s)|+|X'(s')|\|_{L^\infty_s(\mathbb{T})L^\infty_{s'}(B_\delta(s))}\|\partial_s M\|_{L^2_s(\mathbb{T})L^2_{s'}(B_\delta(s))}\||M|+|N|\|_{L^\infty_s(\mathbb{T})L^2_{s'}(B_\delta(s))}\\
&\;+C \lambda^{-2}\||X'(s)|+|X'(s')|\|_{L^\infty_s(\mathbb{T})L^\infty_{s'}(B_\delta(s))}\|M\|_{L^\infty_s(\mathbb{T})L^2_{s'}(B_\delta(s))}\||\partial_s M|+|\partial_s N|\|_{L^2_s(\mathbb{T})L^2_{s'}(B_\delta(s))}\\
&\;+C \lambda^{-2}\|X''(s)\|_{L^3_s(\mathbb{T})L^3_{s'}(B_\delta(s))}\|M\|_{L^{12}_s(\mathbb{T})L^3_{s'}(B_\delta(s))}\||M|+|N|\|_{L^{12}_s(\mathbb{T})L^3_{s'}(B_\delta(s))}\\
&\;+C\||M|+|X''(s)|\|_{L^2_s(\mathbb{T})L^2_{s'}(B_\delta(s))}\\
\leq &\;C\lambda^{-3}\|X'\|_{L^\infty(\mathbb{T})}\left(\delta^{1/6}\|X''\|_{L^3(\mathbb{T})}\right)^3+C\lambda^{-2} \|X'\|_{L^\infty(\mathbb{T})}\delta^{1/2}\|X'''\|_{L^2(\mathbb{T})}\|X''\|_{L^2(\mathbb{T})}\\
&\;+C \lambda^{-2}\|X'\|_{L^\infty(\mathbb{T})}\|X''\|_{L^2(\mathbb{T})}\delta^{1/2}\|X'''\|_{L^2(\mathbb{T})}\\
&\;+C \lambda^{-2}\delta^{1/3}\|X''\|_{L^3(\mathbb{T})}\delta^{1/12}\|X''\|_{L^3(\mathbb{T})}\delta^{1/12}\|X''\|_{L^3(\mathbb{T})}+C\delta^{1/2}\|X''\|_{L^2(\mathbb{T})}\\
\leq &\; C\delta^{1/2}(\lambda^{-3}\|X'\|_{L^\infty}\|X''\|_{L^3}^3+ \lambda^{-2}\|X'\|_{L^\infty}\|X'''\|_{L^2}\|X''\|_{L^2}+ \lambda^{-2}\|X''\|_{L^3}^3+ \|X''\|_{L^2})\\
\leq &\; C\delta^{1/2}(\lambda^{-3}\|X\|_{\dot{H}^{5/2}}^4+ \lambda^{-2}\|X\|_{\dot{H}^{5/2}}^2\|X\|_{\dot{H}^3}).
\end{split}
\end{equation*}
We used \eqref{eqn: lower bound for L} and Sobolev inequality in the last line. For $II_\delta$, we used \eqref{eqn: pointwise estimate for s-derivative of Gamma far field}. 
Applying Lemma \ref{lemma: estimates for L M N} with $I = \mathbb{T}$, we obtain that
\begin{equation*}
\begin{split}
II_\delta\leq &\;C\lambda^{-3}\left\|\int_{B^c_{\delta}(s)}  (|X'(s)|+|X'(s')|)|M||N|(|M|+|N|)\,ds'\right\|_{L^2(\mathbb{T})}\\
&\;+C  \lambda^{-2}\left\|\int_{B^c_{\delta}(s)}|X''(s)||M|(|M|+|N|)\,ds'\right\|_{L^2(\mathbb{T})}+C\left\|\int_{B^c_{\delta}(s)} |M|+|X''(s)|\,ds'\right\|_{L^2(\mathbb{T})}\\
&\;+C \lambda^{-2} \left\|\int_{B^c_{\delta}(s)} |\tau|^{-1}(|X'(s)|+|X'(s')|)(|M|+|X''(s)|)(|M|+|N|)\,ds'\right\|_{L^2(\mathbb{T})}\\
\leq &\;C\lambda^{-3}\||X'(s)|+|X'(s')|\|_{L^\infty_s(\mathbb{T})L^\infty_{s'}(\mathbb{T})}\|M\|_{L^6_s(\mathbb{T})L^3_{s'}(\mathbb{T})}\|N\|_{L^6_s(\mathbb{T})L^3_{s'}(\mathbb{T})}\||M|+|N|\|_{L^6_s(\mathbb{T})L^3_{s'}(\mathbb{T})}\\
&\;+C  \lambda^{-2}\|X''(s)\|_{L^3_s(\mathbb{T})L^3_{s'}(\mathbb{T})}\|M\|_{L^{12}_s(\mathbb{T})L^3_{s'}(\mathbb{T})}\||M|+|N|\|_{L^{12}_s(\mathbb{T})L^3_{s'}(\mathbb{T})}\\
&\;+C\|M\|_{L^2_s(\mathbb{T})L^2_{s'}(\mathbb{T})}+C\|X''(s)\|_{L^2_s(\mathbb{T})L^2_{s'}(\mathbb{T})}\\
&\;+C \lambda^{-2} \|(s'-s)^{-1}\|_{L^\infty_s(\mathbb{T}) L^1_{s'}(B^c_\delta(s))}\||X'(s)|+|X'(s')|\|_{L^\infty_s(\mathbb{T})L^\infty_{s'}(\mathbb{T})}\\
&\;\quad\cdot\||M|+|X''(s)|\|_{L^4_s(\mathbb{T})L^\infty_{s'}(\mathbb{T})}
\||M|+|N|\|_{L^4_s(\mathbb{T})L^\infty_{s'}(\mathbb{T})}\\
\leq &\;C\lambda^{-3}\|X'\|_{L^\infty(\mathbb{T})}\|X''\|_{L^3(\mathbb{T})}^3+C  \lambda^{-2}\|X''\|_{L^3(\mathbb{T})}\|X''\|_{L^3(\mathbb{T})}^2+C\|X''\|_{L^2(\mathbb{T})}\\
&\;+C \lambda^{-2} (|\ln \delta|+1)\|X'\|_{L^\infty(\mathbb{T})}\|\mathcal{M}X''\|_{L^4(\mathbb{T})}^2\\
\leq &\; C(\lambda^{-3}\|X'\|_{L^\infty}\|X''\|_{L^3}^3+ \lambda^{-2}\|X''\|_{L^3}^3+ \|X''\|_{L^2}+(|\ln \delta|+1)\lambda^{-2}\|X'\|_{L^\infty}\|X''\|_{L^4}^2)\\
\leq &\; C(|\ln \delta|+1)\lambda^{-3}\|X\|_{\dot{H}^{5/2}}^4.
\end{split}
\end{equation*}
Here we used Lemma \ref{lemma: estimates for L M N} and Sobolev inequality.
Combining the above two estimates of $I_\delta$ and $II_\delta$, we proved \eqref{eqn: H2 estimate of g_X}.
\end{proof}
\end{lemma}
\begin{remark}
It is clear from the proof that, the goal of splitting $g_X''$ into two parts in \eqref{eqn: splitting of g_X''} is to introduce a small parameter $\delta$ in front of $\|X\|_{\dot{H}^3}$ in \eqref{eqn: H2 estimate of g_X}.
This will be useful in the proof of local well-posedness.
See Section \ref{section: local existence and uniqueness}.
\qed
\end{remark}

We can also show that
\begin{corollary}\label{coro: H2 estimate for g_X0-g_X1}
Let $X_1(s),X_2(s)\in H^3(\mathbb{T})$ both satisfy \eqref{eqn: well_stretched assumption} with some $\lambda>0$. Then for $\forall\, \delta\in(0,\pi)$, and $\forall\,\mu >0$, 
\begin{equation}
\begin{split}
&\;\left\|g''_{X_1}(s)-g''_{X_2}(s)\right\|_{L^2}\\
\leq &\; C_\mu\left[\delta^{1/2}\lambda^{-2}(\|X_1\|_{\dot{H}^{5/2}}+\|X_2\|_{\dot{H}^{5/2}})^2\|X_1-X_2\|_{\dot{H}^3}\right.\\
&\;+\delta^{1/2}\lambda^{-3}(\|X_1\|_{\dot{H}^{5/2}}+\|X_2\|_{\dot{H}^{5/2}})^2(\|X_1\|_{\dot{H}^3}+\|X_2\|_{\dot{H}^3})\|X_1-X_2\|_{\dot{H}^2}\\
&\;\left.+(|\ln \delta|+1)\lambda^{-4}(\|X_1\|_{\dot{H}^{5/2}}+\|X_2\|_{\dot{H}^{5/2}})^4\|X_1-X_2\|_{\dot{W}^{2,2+\mu}}\right],
\end{split}
\label{eqn: H2 estimate for g_X0-g_X1}
\end{equation}
where $C_\mu>0$ is a universal constant depending on $\mu$.
\begin{proof}
We take $s$-derivative in \eqref{eqn: simplified Gamma order 1} first, and argue as in the proofs of Corollary \ref{coro: L2 estimate for g_X1-g_X2} and Lemma \ref{lemma: H2 estimate of g_X}. The calculation is unnecessarily long but tedious.
We omit the details here.
\end{proof}
\end{corollary}

\section{Existence and Uniqueness of the Local-in-time Solution}\label{section: local existence and uniqueness}
To this end, we are able to prove the local well-posedness of \eqref{eqn: contour dynamic formulation of the immersed boundary problem}.
\subsection{Existence}\label{section: local existence}

\begin{proof}[Proof of Theorem \ref{thm: local in time existence} (existence of the local-in-time solution)]
For $\forall\,Y\in L^1(\mathbb{T})$, we split it into its mean $\bar{Y}$ and its oscillation $\tilde{Y}$, i.e.
\begin{equation*}
\bar{Y} \triangleq \frac{1}{2\pi}\int_\mathbb{T} Y(s)\,ds,\quad \tilde{Y}(s)\triangleq Y(s)-\bar{Y}.
\end{equation*}
Then \eqref{eqn: contour dynamic formulation of the immersed boundary problem} could be split into two equations as well, one for $\tilde{X}$ and the other for $\bar{X}$. Namely,
\begin{equation}
\begin{split}
&\;\partial_t \tilde{X}(s,t)= \mathcal{L}\tilde{X}(s,t) + \widetilde{g_{\tilde{X}}}(s,t),\quad s\in \mathbb{T}, t> 0,\\
&\;\tilde{X}(s,0) = \widetilde{X_0}(s),
\end{split}
\label{eqn: equation for oscillation of X in the main thm}
\end{equation}
and
\begin{equation}
\frac{d}{dt}\bar{X}(t) = \overline{g_{\tilde{X}}} = \frac{1}{2\pi}\int_\mathbb{T}g_{\tilde{X}}(s,t)\,ds,\quad \bar{X}(0) = \overline{X_0}.
\label{eqn: equation for mean of X in the main thm}
\end{equation}

We first consider the existence of solutions of the $\tilde{X}$-equation \eqref{eqn: equation for oscillation of X in the main thm}.
Given $X_0$, with $T>0$ to be determined, we define
\begin{equation}
\begin{split}
\Omega_{0,T}(X_0) = &\;\left\{Y(s,t)\in\Omega_{T}:\;\int_\mathbb{T}Y(s,t)\,ds \equiv 0,\;\|Y_t(s,t)\|_{L^2_T \dot{H}^2(\mathbb{T})}\leq \|X_0\|_{\dot{H}^{5/2}(\mathbb{T})},\right.\\
&\;\qquad\left.\left\|Y(s,t)-\mathrm{e}^{t\mathcal{L}}\widetilde{X_0}\right\|_{L^{\infty}_T \dot{H}^{5/2}\cap L^2_T \dot{H}^{3}(\mathbb{T})} \leq \|X_0\|_{\dot{H}^{5/2}(\mathbb{T})},\;Y(s,0)=\widetilde{X_0}(s)\right\}.
\end{split}
\label{eqn: definition of the neighbourhood of X0 used in the proof of local existence}
\end{equation}
The subscript $0$ stresses that functions in $\Omega_{0,T}(X_0)$ has mean zero on $\mathbb{T}$.
We remark that only the seminorms are used, since the mean of $X_0$ is irrelevant in the equation for $\tilde{X}$, which is always this case in the sequel.
$\Omega_{0,T}(X_0)$ is non-empty.
Indeed, by Lemma \ref{lemma: improved Hs estimate and Hs continuity of semigroup solution} and Lemma \ref{lemma: a priori estimate of nonlocal eqn}, $\mathrm{e}^{t\mathcal{L}}\widetilde{X_0}\in \Omega_{0,T}(X_0)$.
It is also convex and closed in $\Omega_T$.
By Aubin-Lions lemma, $\Omega_{0,T}(X_0)$ is compact in $C_T H^2(\mathbb{T})$.

By Lemma \ref{lemma: a priori estimate of nonlocal eqn}, for $\forall\, Y\in \Omega_{0,T}(X_0)$, $\|Y\|_{L^{\infty}_T \dot{H}^{5/2}\cap L^2_T \dot{H}^3(\mathbb{T})} \leq 4 \|X_0\|_{\dot{H}^{5/2}(\mathbb{T})}$.
Moreover, by taking $T$ sufficiently small, we will have
\begin{equation}
\left|Y(s_1,t) - Y(s_2,t)\right| \geq \frac{\lambda}{2}|s_1 - s_2|,\quad \forall\,s_1, s_2\in\mathbb{T},\;t\in[0,T].
\label{eqn: uniform bi lipschitz constant in the neighborhood}
\end{equation}
In fact, if we assume $C_1 \|X_0\|_{\dot{H}^{5/2}(\mathbb{T})}T^{1/2}\leq \lambda/2$, with $C_1$ being a universal constant coming from Sobolev inequality that will be clear below,
\begin{equation*}
\begin{split}
||Y(s_1,t) - Y(s_2,t)| - |X_0(s_1) - X_0(s_2)||\leq &\;|(Y-X_0)(s_1,t) - (Y-X_0)(s_2,t)|\\
\leq &\;C_1 \|Y-X_0\|_{C_T \dot{H}^2(\mathbb{T})} |s_1 - s_2|\\
\leq &\;C_1 \|X_0\|_{\dot{H}^{5/2}(\mathbb{T})}T^{1/2} |s_1 - s_2|\leq \frac{\lambda}{2}|s_1-s_2|.
\end{split}
\end{equation*}
Here we used the assumptions that $\|Y_t(s,t)\|_{L^2_T \dot{H}^2(\mathbb{T})}\leq \|X_0\|_{\dot{H}^{5/2}(\mathbb{T})}$ and $Y(s,0)=\widetilde{X_0}(s)$.
Then \eqref{eqn: uniform bi lipschitz constant in the neighborhood} follows from \eqref{eqn: bi Lipschitz assumption in main thm} and the triangle inequality.

Under the above assumption, we define a map $V: \Omega_{0,T}(X_0) \rightarrow \Omega_{0,T}(X_0)$ as follows, with $T$ to be determined. For given $Y(s,t) \in \Omega_{0,T}(X_0)$, let $Z \triangleq VY$ solve
\begin{equation}
\partial_t Z(s,t)= \mathcal{L}Z(s,t) + \widetilde{g_Y}(s,t),\quad s\in \mathbb{T}, t\in[0,T],\quad Z(s,0) = \widetilde{X_0}(s).
\label{eqn: equation to define the map V}
\end{equation}
To show $V$ is well-defined, we first claim that $Z\in \Omega_T$.
In fact, for $Y\in\Omega_{0,T}(X_0)$, by Lemma \ref{lemma: H2 estimate of g_X},
\begin{equation}
\begin{split}
\|\widetilde{g_Y}\|_{L^2_T\dot{H}^2(\mathbb{T})} \leq&\;C\left(\delta^{1/2}\lambda^{-2}\|Y\|_{L^2_T\dot{H}^{3}(\mathbb{T})}\|Y\|_{L^\infty_T\dot{H}^{5/2}(\mathbb{T})}^2+T^{1/2}(|\ln \delta|+1)\lambda^{-3}\|Y\|_{L^\infty_T\dot{H}^{5/2}(\mathbb{T})}^4\right)\\
\leq &\;C_2\lambda^{-3}\|X_0\|_{\dot{H}^{5/2}(\mathbb{T})}^4(\delta^{1/2}+T^{1/2}(|\ln \delta|+1)).
\end{split}
\label{eqn: estimate for the source term in local existence}
\end{equation}
In the last line, we used \eqref{eqn: lower bound for L} and Sobolev inequality to have that $\lambda\leq C\|X_0\|_{\dot{H}^{5/2}(\mathbb{T})}$.
Then for $\forall\, T>0$, Lemma \ref{lemma: a priori estimate of nonlocal eqn} gives the existence and uniqueness of the solution $Z\in \Omega_T$, which satisfies
\begin{equation}
\|\partial_t Z\|_{L^2_{T} \dot{H}^2(\mathbb{T})} \leq \frac{1}{2}\|X_0\|_{\dot{H}^{5/2}(\mathbb{T})}+\|\widetilde{g_Y}\|_{L^2_{T}\dot{H}^2(\mathbb{T})}.
\label{eqn: bound of Z_t}
\end{equation}
$Z$ obviously has mean zero for all time.

Now consider $W =Z-\mathrm{e}^{t\mathcal{L}}\widetilde{X_0}$, which solves
\begin{equation*}
\partial_t W(s,t)= \mathcal{L}W(s,t) + \widetilde{g_Y}(s,t),\quad W(s,0) = 0.
\end{equation*}
By Lemma \ref{lemma: a priori estimate of nonlocal eqn} and \eqref{eqn: estimate for the source term in local existence}, we find that
\begin{equation}
\|W\|_{L^\infty_{T}\dot{H}^{5/2}\cap L^2_{T} \dot{H}^{3}(\mathbb{T})} \leq 6\|\widetilde{g_Y}\|_{L^2_T\dot{H}^{2}(\mathbb{T})}\leq 6C_2\lambda^{-3}\|X_0\|_{\dot{H}^{5/2}(\mathbb{T})}^4(\delta^{1/2}+T^{1/2}(|\ln \delta|+1)).
\label{eqn: bound on W}
\end{equation}
To this end, we first take $\delta \leq \delta_0(\lambda,\|X_0\|_{\dot{H}^{5/2}})$ sufficiently small, s.t.
\begin{equation*}
C_2\lambda^{-3}\|X_0\|_{\dot{H}^{5/2}(\mathbb{T})}^4\delta^{1/2} \leq \frac{1}{12}\|X_0\|_{\dot{H}^{5/2}(\mathbb{T})},
\end{equation*}
and then assume $T \leq T_0(\lambda,\|X_0\|_{\dot{H}^{5/2}},\delta)$ sufficiently small as well, s.t.
\begin{equation}
\begin{split}
&\;C_1 \|X_0\|_{\dot{H}^{5/2}(\mathbb{T})}T^{1/2}\leq \frac{1}{2}\lambda,\\
&\;C_2\lambda^{-3}\|X_0\|_{\dot{H}^{5/2}(\mathbb{T})}^4 T^{1/2}(|\ln \delta|+1) \leq \frac{1}{12}\|X_0\|_{\dot{H}^{5/2}(\mathbb{T})}.
\end{split}
\label{eqn: constraints on existence time T}
\end{equation}
This implies $\|Z-\mathrm{e}^{t\mathcal{L}}\widetilde{X_0}\|_{L^\infty_{T}\dot{H}^{5/2}\cap L^2_{T} \dot{H}^{3}(\mathbb{T})} \leq \|X_0\|_{\dot{H}^{5/2}(\mathbb{T})}$ by \eqref{eqn: bound on W}.
Also by \eqref{eqn: estimate for the source term in local existence} and \eqref{eqn: bound of Z_t}
\begin{equation*}
\|\partial_t Z\|_{L^2_{T} \dot{H}^2(\mathbb{T})} \leq \frac{1}{2}\|X_0\|_{\dot{H}^{5/2}(\mathbb{T})}+\frac{1}{6}\|X_0\|_{\dot{H}^{5/2}(\mathbb{T})} \leq \|X_0\|_{\dot{H}^{5/2}(\mathbb{T})}.
\end{equation*}
Hence, $V$ is well-defined from $\Omega_{0,T}(X_0)$ to itself.
We note that the upper bound of valid $T$, which is $T_0$, essentially only depends on $\lambda$ and $\|X_0\|_{\dot{H}^{5/2}(\mathbb{T})}$.

By Aubin-Lions lemma, $V(\Omega_{0,T}(X_0))$ is compact in $C_{T} H^2(\mathbb{T})$. By Schauder fixed point theorem, there is a fixed point of the map $V$ in $V(\Omega_{0,T}(X_0))\subset \Omega_{0,T}(X_0)$, denoted by $\tilde{X}\in \Omega_{T}$, which is a solution of \eqref{eqn: equation for oscillation of X in the main thm}.
It satisfies
\begin{equation}
\|\tilde{X}\|_{L^\infty_{T} \dot{H}^{5/2}\cap L^2_{T} \dot{H}^{3}(\mathbb{T})}\leq 4\|X_0\|_{\dot{H}^{5/2}(\mathbb{T})},\quad \|\partial_t \tilde{X}\|_{L^2_{T} \dot{H}^2(\mathbb{T})} \leq \|X_0\|_{\dot{H}^{5/2}(\mathbb{T})},
\label{eqn: a priori estimate for the local solution}
\end{equation}
and
\begin{equation}
\left|\tilde{X}(s_1,t) - \tilde{X}(s_2,t)\right| \geq \frac{\lambda}{2}|s_1 - s_2|,\quad \forall\,s_1,s_2\in\mathbb{T},\;t\in[0,T].
\label{eqn: uniform bi lipschitz constant of the local solution}
\end{equation}

To this end, we turn to the ODE \eqref{eqn: equation for mean of X in the main thm} for $\bar{X}$.
By Lemma \ref{lemma: L infty estimate for g_X}, for $\forall\, s\in\mathbb{T}$ and $t\in[0,T]$,
\begin{equation*}
|\overline{g_{\tilde{X}}}|\leq \|g_{\tilde{X}}(s,t)\|_{L^\infty(\mathbb{T})}\leq \frac{C}{\lambda}\|\tilde{X}\|_{\dot{H}^1}\|\tilde{X}\|_{\dot{H}^2}\leq C\lambda^{-1}\|X_0\|_{\dot{H}^{5/2}(\mathbb{T})}^2.
\end{equation*}
It is then easy to show that \eqref{eqn: equation for mean of X in the main thm} admits a unique solution $\bar{X}(t)\in C^{0,1}([0,T])$ once $\tilde{X}$ is given.
The solution for \eqref{eqn: contour dynamic formulation of the immersed boundary problem} is thus given by $X(s,t) = \bar{X}(t)+\tilde{X}(s,t)$. This proves the existence of the local-in-time solutions in $\Omega_{T}$. \eqref{eqn: a priori estimate for the local solution in the main theorem} and \eqref{eqn: uniform bi lipschitz constant of the local solution in the main theorem} follow from \eqref{eqn: a priori estimate for the local solution} and \eqref{eqn: uniform bi lipschitz constant of the local solution} respectively.

That $X\in L^2_{T}H^3(\mathbb{T})$ together with $X_t\in L^2_{T}H^2(\mathbb{T})$ implies that $X$ is almost everywhere equal to a continuous function valued in $H^{5/2}(\mathbb{T})$, i.e.\;$X$ could be realized as an element in $C([0,T];H^{5/2}(\mathbb{T}))$. This can be proved by classic arguments (see Temam \cite{temam1984navier}, \S\,1.4 of Chapter III).
\end{proof}

\subsection{Uniqueness}\label{section: local uniqueness}
\begin{proof}[Proof of Theorem \ref{thm: local in time uniqueness} (uniqueness of the local-in-time solution)]
Suppose $X_1,X_2\in \Omega_T$ are two solutions of \eqref{eqn: contour dynamic formulation of the immersed boundary problem}, both satisfying the assumption \eqref{eqn: bi lipschitz assumption in uniqueness thm}.
Let
\begin{equation}
R = \|X_1\|_{L_{T}^\infty\dot{H}^{5/2}\cap L_{T}^2 \dot{H}^3(\mathbb{T})} + \|X_2\|_{L_{T}^\infty\dot{H}^{5/2}\cap L_{T}^2 \dot{H}^3(\mathbb{T})}\geq C(c)\lambda
\label{eqn: uniform bound in uniqueness thm}
\end{equation}
and $Q(s,t) \triangleq X_1-X_2$. Then $\tilde{Q} = \widetilde{X_1}-\widetilde{X_2}$ solves
\begin{equation*}
\partial_t \tilde{Q}(s,t)= \mathcal{L}\tilde{Q}(s,t) + \widetilde{g_{\widetilde{X_1}}}(s,t)-\widetilde{g_{\widetilde{X_2}}}(s,t),\quad \tilde{Q}(s,0) = 0,\quad (s,t)\in\mathbb{T}\times [0,T].
\end{equation*}
By Corollary \ref{coro: H2 estimate for g_X0-g_X1} with $\mu = 2$ and Sobolev inequality, with $t\in(0,T]$ to be determined
\begin{equation}
\begin{split}
&\;\left\|g''_{X_1}(s)-g''_{X_2}(s)\right\|_{L_{t}^2L^2}\\
\leq &\;C\left[
\delta^{1/2}\lambda^{-2}(\|X_1\|_{L^\infty_t\dot{H}^{5/2}}+\|X_2\|_{L^\infty_t\dot{H}^{5/2}})^2\|X_1-X_2\|_{L^2_t\dot{H}^3}\right.\\
&\;+\delta^{1/2}\lambda^{-3}(\|X_1\|_{L^\infty_t\dot{H}^{5/2}}+\|X_2\|_{L^\infty_t\dot{H}^{5/2}})^2(\|X_1\|_{L^2_t\dot{H}^3}+\|X_2\|_{L^2_t\dot{H}^3})\|X_1-X_2\|_{L^\infty_t\dot{H}^2}\\
&\;\left.+(|\ln \delta|+1)\lambda^{-4}t^{1/2}(\|X_1\|_{L^\infty_t\dot{H}^{5/2}}+\|X_2\|_{L^\infty_t\dot{H}^{5/2}})^4\|X_1-X_2\|_{L^\infty_t\dot{W}^{2,4}}\right],\\
\leq &\; C(c)\left[\delta^{1/2}\lambda^{-2}R^2\|Q\|_{L^2_t\dot{H}^3}+\delta^{1/2}\lambda^{-3}R^3\|Q\|_{L^\infty_t\dot{H}^2}+(|\ln \delta|+1)\lambda^{-4}R^4t^{1/2}\|Q\|_{L^\infty_t\dot{H}^{5/2}}\right]\\
\leq &\; C(c)[\delta^{1/2}+(|\ln \delta|+1)t^{1/2}]\lambda^{-4}R^4\|\tilde{Q}\|_{L^{\infty}_{t}\dot{H}^{5/2}\cap L^2_{t}\dot{H}^3(\mathbb{T})}.
\end{split}
\label{eqn: space time estimate for the difference of solutions in proving uniqueness}
\end{equation}
Here we repeatedly used \eqref{eqn: uniform bound in uniqueness thm}. By Lemma \ref{lemma: a priori estimate of nonlocal eqn},
\begin{equation*}
\|\tilde{Q}\|_{L^{\infty}_{t}\dot{H}^{5/2}\cap L^2_{t}\dot{H}^3(\mathbb{T})}\leq C(c)[\delta^{1/2}+(|\ln \delta|+1)t^{1/2}]\lambda^{-4}R^4\|\tilde{Q}\|_{L^{\infty}_{t}\dot{H}^{5/2}\cap L^2_{t}\dot{H}^3(\mathbb{T})}.
\end{equation*}
Hence, we first take $\delta = \delta_*(\lambda, R, c)$ sufficiently small and then take  $t=t_*(\lambda, R, c)$ sufficiently small,
such that $C(c)[\delta^{1/2}+(|\ln \delta|+1)t^{1/2}]\lambda^{-4}R^4\in(0,1)$.
This implies that
\begin{equation*}
\|\tilde{Q}\|_{L^{\infty}_{t_*}\dot{H}^{5/2}\cap L^2_{t_*}\dot{H}^3(\mathbb{T})}=0,
\end{equation*}
i.e.\;$\tilde{Q}(s,t) = 0$ for $t\in[0,t_*]$. Since \eqref{eqn: bi lipschitz assumption in uniqueness thm} and \eqref{eqn: uniform bound in uniqueness thm} are uniform throughout $[0,T]$, the above argument is also true for arbitrary initial time, i.e.\;provided that $\tilde{Q}(s,t_0) = 0$ for some $t_0\in [0,T]$, then $\tilde{Q}(s,t) = 0$ for $t\in[t_0,\min\{t_0+t_*,T\}]$. Hence, $\tilde{Q}(s,t) \equiv 0$ for $t\in[0,T]$, i.e.\;$\widetilde{X_1}(s,t) \equiv \widetilde{X_2}(s,t)$.

Recall that in \eqref{eqn: equation for mean of X in the main thm}, the solution $\bar{X}(t)$ is uniquely determined in $C^{0,1}([0,T])$ by $\tilde{X}(s,t)$. This implies that $\overline{X_1}(t)\equiv \overline{X_2}(t)$, and thus $X_1(s,t) \equiv X_2(s,t)$ for $(s,t)\in\mathbb{T}\times [0,T]$.
This proves the uniqueness under the assumption \eqref{eqn: bi lipschitz assumption in uniqueness thm}.

The uniqueness of the local-in-time solution obtained in Theorem \ref{thm: local in time existence} follows immediately.
\end{proof}

\section{Existence and Uniqueness of Global-in-time Solutions near Equilibrium Configurations}\label{section: global existence}
In this section, we will prove that existence of global solution provided that the initial string configuration is sufficiently close to equilibrium.
The closeness is measured using the difference between a string configuration $Y$ and its closet equilibrium configuration $Y_*$ (see Definition \ref{def: closest equilbrium state}).
We start with several remarks on the definition of the closest equilibrium configuration.


\begin{remark}\label{remark: mass center of the equilibrium agrees with initial data}
In the definition of $Y_*$, we have $x_* = \frac{1}{2\pi}\int_{\mathbb{T}} Y(s)\,ds$.
This can be seen from the Fourier point of view.
Assume $Y(s) = \sum_{k\in\mathbb{Z}} \hat{Y}_k \mathrm{e}^{iks}$, where $\hat{Y}_k$'s are complex-valued 2-vectors. By Parseval's identity,
\begin{equation}
\begin{split}
\frac{1}{2\pi}\int_{\mathbb{T}}|Y(s)-Y_{\theta,x}(s)|^2\,ds = &\;|\hat{Y}_0-x|^2 +\sum_{k\in\mathbb{Z},|k|\geq 2} |\hat{Y}_k|^2\\
&\;+ \left|\hat{Y}_{1}-
R_Y \mathrm{e}^{i\theta}\left(
\begin{array}{c}
\frac{1}{2}\\\frac{1}{2i}
\end{array}
\right)
\right|^2
+\left|\hat{Y}_{-1}-
R_Y \mathrm{e}^{-i\theta}\left(
\begin{array}{c}
\frac{1}{2}\\-\frac{1}{2i}
\end{array}
\right)
\right|^2.
\end{split}
\label{eqn: L2 difference of Y and its closest equilibrium using Parseval}
\end{equation}
In order to achieve its minimum, we should take $x_* = \hat{Y}_0 = \frac{1}{2\pi}\int_{\mathbb{T}} Y(s)\,ds$. In the sequel, we shall denote $Y_\theta(s) \triangleq Y_{\theta,x_*}(s)$ and only minimize $\|Y-Y_\theta\|_{L^2(\mathbb{T})}$ with respect to $\theta$.
\qed
\end{remark}
\begin{remark}\label{remark: L2 closest is also Hs closest}
Although $Y_*$ is defined to be the closest to $Y$ in the $L^2$-distance among all $Y_\theta $, it is also the closest in the $H^s$-sense for all $s\geq 0$. Indeed, by Parseval's identity,
\begin{equation*}
\begin{split}
\frac{1}{2\pi}\|Y-Y_{\theta}\|^2_{\dot{H}^s(\mathbb{T})} = &\;\sum_{k\in\mathbb{Z},|k|\geq 2} |k|^{2s}|\hat{Y}_k|^2+ \left|\hat{Y}_{1}-
R_Y \mathrm{e}^{i\theta}\left(
\begin{array}{c}
\frac{1}{2}\\\frac{1}{2i}
\end{array}
\right)
\right|^2
+\left|\hat{Y}_{-1}-
R_Y \mathrm{e}^{-i\theta}\left(
\begin{array}{c}
\frac{1}{2}\\-\frac{1}{2i}
\end{array}
\right)
\right|^2\\
=&\;\frac{1}{2\pi}\|Y-Y_{\theta}\|^2_{L^2(\mathbb{T})}+\sum_{k\in\mathbb{Z},|k|\geq 2} (|k|^{2s}-1)|\hat{Y}_k|^2.
\end{split}
\end{equation*}
The last term in the last line is constant with respect to $\theta$, which implies that $\theta_*$ also optimizes $\|Y-Y_{\theta}\|_{\dot{H}^s(\mathbb{T})}$.
\qed
\end{remark}

The following lemma establishes the equivalence of the $H^1$-distance and the energy difference between a string configuration $Y$ and its closest equilibrium configuration $Y_*$.
Recall that the elastic energy of $Y$ is $\|Y\|_{\dot{H}(\mathbb{T})}^2/2$ (see Lemma \ref{lemma: energy estimate}).
The motivation is that we wish to transform the global coercive bound on the energy difference, which comes from \eqref{eqn: energy estimate of Stokes immersed boundary problem}, into a bound for more convenient quantity $\|Y-Y_*\|_{\dot{H}^1}$.
\begin{lemma}\label{lemma: estimates concerning closest equilbrium}
We have the following estimates for $Y$ and its closest equilibrium configuration $Y_*$:
\begin{equation}
\frac{1}{2}\left(\|Y'(s)\|_{L^2(\mathbb{T})}^2-\|Y_*'(s)\|_{L^2(\mathbb{T})}^2\right)\leq \|Y'(s)-Y_*'(s)\|_{L^2(\mathbb{T})}^2\leq 4\left(\|Y'(s)\|_{L^2(\mathbb{T})}^2-\|Y_*'(s)\|_{L^2(\mathbb{T})}^2\right).
\label{eqn: difference in H1 bounded by difference in energy}
\end{equation}

\begin{proof}
Without loss of generality, we assume that $(\theta_*,x_*) = (0,0)$. Otherwise we simply make a translation and rotation of $Y$.
Define $D(s) = Y(s)-Y_*(s)$. By Remark \ref{remark: mass center of the equilibrium agrees with initial data} and the above assumption, $D(s)$ is of mean zero on $\mathbb{T}$.

We first prove the upper bound. By the definition of $\theta_*$ and $Y_*$, we know that
\begin{equation}
0 = \left.\frac{d}{d\theta}\right|_{\theta = \theta_*}\int_{\mathbb{T}}\left|Y(s)-Y_\theta(s)\right|^2\,ds = -2\int_\mathbb{T} (Y-Y_*)\cdot Y'_*\,ds = -2\int_\mathbb{T} D\cdot Y'_*\,ds,
\label{eqn: equation for the optimal approximated equilibrium}
\end{equation}
and
\begin{equation}
\begin{split}
0 \leq &\; \left.\frac{d^2}{d\theta^2}\right|_{\theta = \theta_*}\int_{\mathbb{T}}\left|Y(s)-Y_\theta(s)\right|^2\,ds = -2\int_\mathbb{T} -Y'_*\cdot Y'_*+(Y-Y_*)\cdot Y_*''\,ds\\
= &\;2\int_\mathbb{T} |Y'_*|^2+(Y-Y_*)\cdot Y_*\,ds = 2\int_\mathbb{T} |Y'_*|^2+D\cdot Y_*\,ds\\
= &\;4\pi R_Y^2+2\int_\mathbb{T} D\cdot Y_*\,ds\\
\label{eqn: second order equation for the optimal approximated equilibrium}
\end{split}
\end{equation}
Here we used $Y_{*}'' = -Y_{*}$.
Moreover, since $Y$ and $Y_*$ have the same effective radius, by \eqref{eqn: enclosed area is pi},
\begin{equation*}
0 =\int_\mathbb{T}Y\times Y'\,ds - \int_\mathbb{T}Y_*\times Y_*'\,ds =\int_\mathbb{T}D\times Y_*'+ Y_*\times D' + D\times D'\,ds.
\end{equation*}
Since $Y_*'(s) = (-R_Y\sin s, R_Y\cos s) = Y_*^\perp(s)$ and $Y_*'' = - Y_*$, it is further simplified to be
\begin{equation}
0 =\int_\mathbb{T}D\cdot Y_*+ Y_*'\cdot D' + D\times D'\,ds = \int_\mathbb{T}D\cdot Y_*- Y_*''\cdot D + D\times D'\,ds = \int_\mathbb{T}2D\cdot Y_* + D\times D'\,ds.
\label{eqn: constraint on deviation from volume conservation}
\end{equation}

In the sequel, we shall write $Y_*$ and $D$ in terms of their Fourier coefficients. With the assumption that $(\theta_*,x_*) = (0,0)$, we have
\begin{align*}
&\;Y_*(s) = R_Y\left(
\begin{array}{c}
\frac{1}{2}\\-\frac{i}{2}
\end{array}
\right)\mathrm{e}^{is} + R_Y\left(
\begin{array}{c}
\frac{1}{2}\\\frac{i}{2}
\end{array}
\right)\mathrm{e}^{-is},\\
&\;Y'_*(s) = R_Y\left(
\begin{array}{c}
\frac{i}{2}\\\frac{1}{2}
\end{array}
\right)\mathrm{e}^{is} + R_Y\left(
\begin{array}{c}
-\frac{i}{2}\\\frac{1}{2}
\end{array}
\right)\mathrm{e}^{-is}.
\end{align*}
Assume $D(s) = \sum_{k\in \mathbb{Z}}\hat{D}_k \mathrm{e}^{iks}$, where $\hat{D}_k$'s are complex-valued $2$-vectors satisfying $\hat{D}_{-k} = \overline{\hat{D}_{k}}$. Hence, \eqref{eqn: equation for the optimal approximated equilibrium} could be rewritten as
\begin{equation*}
0= \hat{D}_1\cdot \overline{\left(
\begin{array}{c}
\frac{i}{2}\\\frac{1}{2}
\end{array}
\right)} + \hat{D}_{-1}\cdot \overline{\left(
\begin{array}{c}
-\frac{i}{2}\\\frac{1}{2}
\end{array}
\right)} = \left(-\frac{i}{2}\hat{D}_{1,1}+ \frac{1}{2}\hat{D}_{1,2}\right)+\overline{\left(-\frac{i}{2}\hat{D}_{1,1}+ \frac{1}{2}\hat{D}_{1,2}\right)},
\end{equation*}
where $\hat{D}_{1,1}$ and $\hat{D}_{1,2}$ represent the first and the second component of $\hat{D}_1$ respectively. This implies that
\begin{equation}
0 = \Re(-i\hat{D}_{1,1}+\hat{D}_{1,2}) = \Im \hat{D}_{1,1}+\Re \hat{D}_{1,2}.
\label{eqn: simplified equation for the optimal approximated equilibrium}
\end{equation}
Similarly, the terms in \eqref{eqn: constraint on deviation from volume conservation} could be rewritten as follows:
\begin{equation}
\begin{split}
\int_\mathbb{T} 2D\cdot Y_*\,ds = &\;4\pi R_Y\hat{D}_1\cdot \overline{\left(
\begin{array}{c}
\frac{1}{2}\\-\frac{i}{2}
\end{array}
\right)} + 4\pi R_Y\hat{D}_{-1}\cdot\overline{\left(
\begin{array}{c}
\frac{1}{2}\\\frac{i}{2}
\end{array}
\right)}\\
= &\;4\pi R_Y(\Re \hat{D}_{1,1} - \Im \hat{D}_{1,2}),
\end{split}
\label{eqn: inner product of D and Y_star}
\end{equation}
and
\begin{equation}
\begin{split}
\int_\mathbb{T}D\times D'\,ds = &\;2\pi \sum_{k\in\mathbb{Z}}\hat{D}_k\times \overline{\left(ik \hat{D}_k\right)}\\
=&\;2\pi \sum_{k\in\mathbb{Z}} -ik \left(\hat{D}_{k,1}\overline{\hat{D}_{k,2}} - \hat{D}_{k,2}\overline{\hat{D}_{k,1}}\right)\\
=&\;2\pi \sum_{k\in\mathbb{Z}} 2k \Im\left(\hat{D}_{k,1}\overline{\hat{D}_{k,2}}\right)\\
=&\;4\pi \sum_{k\in\mathbb{Z}} k (\Im \hat{D}_{k,1}\Re \hat{D}_{k,2}-\Re \hat{D}_{k,1}\Im \hat{D}_{k,2})\\
\leq &\;2\pi \sum_{\genfrac{}{}{0pt}{}{k\in\mathbb{Z}}{|k|\geq 2}} |k| |\hat{D}_{k}|^2+ 4\pi (\Im \hat{D}_{1,1}\Re \hat{D}_{1,2}-\Re \hat{D}_{1,1}\Im \hat{D}_{1,2})\\
&\;-4\pi (\Im \hat{D}_{-1,1}\Re \hat{D}_{-1,2}-\Re \hat{D}_{-1,1}\Im \hat{D}_{-1,2})\\
= &\;2\pi \sum_{\genfrac{}{}{0pt}{}{k\in\mathbb{Z}}{|k|\geq 2}} |k| |\hat{D}_{k}|^2+ 8\pi (\Im \hat{D}_{1,1}\Re \hat{D}_{1,2}-\Re \hat{D}_{1,1}\Im \hat{D}_{1,2}).
\end{split}
\label{eqn: cross product of D and D'}
\end{equation}
Here we used the fact that $\hat{D}_{-1} = \overline{\hat{D}_1}$.
By \eqref{eqn: second order equation for the optimal approximated equilibrium} and \eqref{eqn: inner product of D and Y_star}, we know that
\begin{equation}
-\Re \hat{D}_{1,1} + \Im \hat{D}_{1,2}\leq R_Y.
\label{eqn: constraints on the coefficients from the second order condition of optimal approximation}
\end{equation}
We calculate that
\begin{equation}
\|Y'(s) - Y_*'(s)\|_{L^2}^2 = \|D'(s)\|_{L^2}^2 = 2\pi\sum_{k\in\mathbb{Z}} k^2|\hat{D}_k|^2,
\label{eqn: H1 norm of deviation in terms of Fourier coefficients}
\end{equation}
and
\begin{equation}
\begin{split}
\|Y'(s)\|_{L^2}^2 - \|Y_*'(s)\|_{L^2}^2 = &\;\int_\mathbb{T}(Y_*'+D')\cdot(Y_*'+D') - Y_*'\cdot Y_*'\,ds = \int_\mathbb{T}2Y_*'\cdot D'+D'\cdot D'\,ds\\
= &\;\int_\mathbb{T}-2Y_*''\cdot D+D'\cdot D'\,ds= \int_\mathbb{T}2Y_*\cdot D+D'\cdot D'\,ds\\
= &\;\int_\mathbb{T}2D\cdot Y_*\,ds+2\pi \sum_{k\in\mathbb{Z}} k^2 |\hat{D}_k|^2.
\end{split}
\label{eqn: expression for excess energy}
\end{equation}
\begin{case}
If $\int_\mathbb{T}2D\cdot Y_*\,ds\geq 0$, we readily proved the upper bound in \eqref{eqn: difference in H1 bounded by difference in energy} by comparing \eqref{eqn: H1 norm of deviation in terms of Fourier coefficients} and \eqref{eqn: expression for excess energy}.
\end{case}
\begin{case}
If $\int_\mathbb{T}2D\cdot Y_*\,ds < 0$, by \eqref{eqn: inner product of D and Y_star}, $\Re \hat{D}_{1,1} - \Im \hat{D}_{1,2} <0$. Then by \eqref{eqn: constraint on deviation from volume conservation}, \eqref{eqn: inner product of D and Y_star}, \eqref{eqn: cross product of D and D'} and \eqref{eqn: expression for excess energy},
\begin{equation*}
\begin{split}
\|Y'(s)\|_{L^2}^2 - \|Y_*'(s)\|_{L^2}^2 = &\;2\pi \sum_{k\in\mathbb{Z}} k^2 |\hat{D}_k|^2 +\frac{3}{2}\int_\mathbb{T}2D\cdot Y_*\,ds-\int_\mathbb{T}D\cdot Y_*\,ds\\
=&\;2\pi \sum_{k\in\mathbb{Z}} k^2 |\hat{D}_k|^2 -\frac{3}{2}\int_\mathbb{T}D\times D'\,ds-\int_\mathbb{T}D\cdot Y_*\,ds\\
\geq &\; 2\pi \sum_{\genfrac{}{}{0pt}{}{k\in\mathbb{Z}}{|k|\geq 2}} k^2 |\hat{D}_k|^2 -3\pi \sum_{\genfrac{}{}{0pt}{}{k\in\mathbb{Z}}{|k|\geq 2}} |k| |\hat{D}_k|^2+2\pi(|\hat{D}_1|^2+|\hat{D}_{-1}|^2)\\
&\;-12\pi(\Im \hat{D}_{1,1}\Re \hat{D}_{1,2}-\Re \hat{D}_{1,1}\Im \hat{D}_{1,2})\\
&\;-2\pi R_Y(\Re \hat{D}_{1,1}- \Im \hat{D}_{1,2})\\
\geq &\; \pi \sum_{\genfrac{}{}{0pt}{}{k\in\mathbb{Z}}{|k|\geq 2}} (2k^2-3|k|) |\hat{D}_k|^2+4\pi|\hat{D}_1|^2 \\
&\;+12\pi\Re \hat{D}_{1,1}\Im \hat{D}_{1,2}-2\pi R_Y(\Re \hat{D}_{1,1}- \Im \hat{D}_{1,2}).
\end{split}
\end{equation*}
In the last line, we used the fact that $\hat{D}_{-1} = \overline{\hat{D}_1}$ and $\Im \hat{D}_{1,1}\Re \hat{D}_{1,2}\leq 0$ due to \eqref{eqn: simplified equation for the optimal approximated equilibrium}.

If $\Re \hat{D}_{1,1}$ and $\Im \hat{D}_{1,2}$ have the same sign, then $12\pi\Re \hat{D}_{1,1}\Im \hat{D}_{1,2}-2\pi R_Y(\Re \hat{D}_{1,1}- \Im \hat{D}_{1,2})\geq 0$. Hence,
\begin{equation*}
\|Y'(s)\|_{L^2}^2 - \|Y_*'(s)\|_{L^2}^2 \geq \pi \sum_{k\in\mathbb{Z}} \frac{1}{2}k^2 |\hat{D}_k|^2+3\pi|\hat{D}_1|^2 \geq \frac{1}{4}\|D'(s)\|_{L^2}^2.
\end{equation*}

Otherwise, if $\Re \hat{D}_{1,1}$ and $\Im \hat{D}_{1,2}$ have different signs, i.e., $\Re \hat{D}_{1,1}\leq 0$ and $-\Im \hat{D}_{1,2}\leq 0$ since $\Re \hat{D}_{1,1} - \Im \hat{D}_{1,2} <0$, we know that
\begin{equation*}
\begin{split}
\|Y'(s)\|_{L^2}^2 - \|Y_*'(s)\|_{L^2}^2 \geq &\; \pi \sum_{k\in\mathbb{Z}} \frac{1}{2}k^2 |\hat{D}_k|^2+3\pi|\hat{D}_1|^2\\
&\;-12\pi|\Re \hat{D}_{1,1}||\Im \hat{D}_{1,2}|+4\pi R_Y\sqrt{|\Re \hat{D}_{1,1}||\Im \hat{D}_{1,2}|}.
\end{split}
\end{equation*}
Also, by \eqref{eqn: constraints on the coefficients from the second order condition of optimal approximation},
\begin{equation}
|\Re \hat{D}_{1,1}||\Im \hat{D}_{1,2}|\leq \frac{1}{4}(-\Re \hat{D}_{1,1}+\Im \hat{D}_{1,2})^2\leq \frac{1}{4}R_Y^2.
\end{equation}
This implies
\begin{equation*}
\begin{split}
&\;3\pi|\hat{D}_1|^2-12\pi|\Re \hat{D}_{1,1}||\Im \hat{D}_{1,2}|+4\pi R_Y\sqrt{|\Re \hat{D}_{1,1}||\Im \hat{D}_{1,2}|}\\
\geq &\; 3\pi|\hat{D}_1|^2 - 3\pi\left(|\Re \hat{D}_{1,1}|^2+|\Im \hat{D}_{1,2}|^2\right)-6\pi|\Re \hat{D}_{1,1}||\Im \hat{D}_{1,2}|+4\pi R_Y\sqrt{|\Re \hat{D}_{1,1}||\Im \hat{D}_{1,2}|}\\
\geq &\;\pi\sqrt{|\Re \hat{D}_{1,1}||\Im \hat{D}_{1,2}|}\left(4R_Y-6\sqrt{|\Re \hat{D}_{1,1}||\Im \hat{D}_{1,2}|}\right)\\
\geq &\;0.
\end{split}
\end{equation*}
Therefore,
\begin{equation}
\|Y'(s)\|_{L^2}^2 - \|Y_*'(s)\|_{L^2}^2 \geq \frac{\pi}{2}\sum_{k\in\mathbb{Z}}k^2 |\hat{D}_k|^2 = \frac{1}{4}\|D'(s)\|_{L^2}^2.
\label{eqn: the excess energy can bound the H1 difference}
\end{equation}
This proves the upper bound in \eqref{eqn: difference in H1 bounded by difference in energy}.
\end{case}

Now we turn to the lower bound.
By \eqref{eqn: constraint on deviation from volume conservation} and \eqref{eqn: expression for excess energy},
\begin{equation*}
\begin{split}
\|Y'(s)\|_{L^2}^2 - \|Y_*'(s)\|_{L^2}^2 = &\;\int_\mathbb{T}2D\cdot Y_*\,ds+\|D'\|_{L^2}^2 = -\int_\mathbb{T}D\times D'\,ds+\|D'\|_{L^2}^2\\
\leq &\;\|D\|_{L^2}\|D'\|_{L^2}+\|D'\|_{L^2}^2 \leq 2\|D'\|_{L^2}^2.
\end{split}
\end{equation*}
Here we used the fact that $D$ has mean zero on $\mathbb{T}$.

This completes the proof.
\end{proof}
\begin{remark} 
As a byproduct, we know that for any Jordan curve $Y(s)\in H^1(\mathbb{T})$, $\|Y_*\|_{\dot{H}^1(\mathbb{T})}\leq \|Y\|_{\dot{H}^1(\mathbb{T})}$. The equality holds if and only if $Y = Y_*$.
Hence, Lemma \ref{lemma: estimates concerning closest equilbrium} implies that the string configuration having a circular shape and uniform parameterization has the lowest elastic energy among all the $H^1$-configurations that enclose the same area.
This can also be showed by isoperimetric inequality and Cauchy-Schwarz inequality.
\qed
\end{remark}
\end{lemma}

Let $X$ be a local solution of \eqref{eqn: contour dynamic formulation of the immersed boundary problem} obtained in Theorem \ref{thm: local in time existence}.
By Lemma \ref{lemma: energy estimate} and Lemma \ref{lemma: estimates concerning closest equilbrium}, we readily have global bound on $\|X-X_*\|_{\dot{H}^1(\mathbb{T})}(t)$.
It would be very ideal if we could show that  $\|X-X_*\|_{\dot{H}^{5/2}(\mathbb{T})}(t)$ can not be (always) big when $\|X-X_*\|_{\dot{H}^1(\mathbb{T})}(t)$ is small.
The following lemma is an effort in this direction, which is crucial in proving Theorem \ref{thm: global existence near equilibrium}.

\begin{lemma}\label{lemma: bound and decay for H2.5 difference when energy difference is small}
Suppose $T\in(0,1]$ and $X_0\in H^{5/2}(\mathbb{T})$. Let $X(s,t)\in \Omega_T$ be a (local) solution of \eqref{eqn: contour dynamic formulation of the immersed boundary problem}, s.t.
\begin{equation}
\|X\|_{L^\infty_{T}\dot{H}^{5/2}\cap L^2_{T}\dot{H}^3(\mathbb{T})} \leq R<+\infty,
\label{eqn: uniform bound in the lemma for the small energy regularity}
\end{equation}
and for some $\lambda>0$,
\begin{equation}
|X(s_1,t)-X(s_2,t)|\geq \lambda|s_1-s_2|,\quad \forall\,s_1,s_2\in\mathbb{T},\;t\in[0,T].
\label{eqn: uniform bi lipschitz constant in the lemma for the small energy regularity}
\end{equation}

\begin{enumerate}
\item There exists $T_*=T_*(T,R,\lambda)\in(0,T]$, s.t.
\begin{equation}
\|X-X_{*}\|_{L^{\infty}_{T_*}\dot{H}^{5/2}(\mathbb{T})}^2\leq 2\|X_0-X_{0*}\|_{\dot{H}^{5/2}(\mathbb{T})}^2,
\label{eqn: upper bound for the growth of H2.5 norm in a short period of time in the statement of the lemma}
\end{equation}
where $X_*(\cdot,t)$ and $X_{0*}$ are the closest equilibrium configuration to $X(\cdot,t)$ and $X_0(\cdot)$ respectively.
\item Given $T'\in(0,T_*]$, there exist a constant $c_* = c_*(R,\lambda, T')>0$, s.t.\;if 
\begin{equation}
\|X_0-X_{0*}\|_{\dot{H}^{5/2}(\mathbb{T})}\geq c_* \|X-X_{0*}\|_{L^\infty_{T'}\dot{H}^1(\mathbb{T})},
\label{eqn: condition in the small energy lemma initial H2.5 norm is much larger than H1 norm on the whole interval}
\end{equation}
then there exists $t_*\in[T'/4,T']$, s.t.
\begin{equation}
\|X-X_{0*}\|^2_{\dot{H}^{5/2}(\mathbb{T})}(t_*)\leq \mathrm{e}^{- t_*/4}\|X_0-X_{0*}\|^2_{\dot{H}^{5/2}(\mathbb{T})}.
\label{eqn: a lower H2.5 norm could be found}
\end{equation}
In particular,
\begin{equation}
\|X-X_{*}\|^2_{\dot{H}^{5/2}(\mathbb{T})}(t_*)\leq \mathrm{e}^{-t_*/4}\|X_0-X_{0*}\|^2_{\dot{H}^{5/2}(\mathbb{T})}.
\label{eqn: a lower H2.5 norm with updated approximation could be found}
\end{equation}
\end{enumerate}

\begin{proof}
It is easy to see that $X_{0*}(s,t) \equiv X_{0*}(s)\in H^{5/2}(\mathbb{T})$ is the (unique, by Theorem \ref{thm: local in time uniqueness}) solution for \eqref{eqn: contour dynamic formulation of the immersed boundary problem} starting from $X_{0*}$.
Consider $\tilde{X}-\widetilde{X_{0*}}$, which satisfies
\begin{equation}
\begin{split}
&\;\partial_t (\tilde{X}-\widetilde{X_{0*}})= \mathcal{L}(\tilde{X}-\widetilde{X_{0*}}) + (\widetilde{g_X} - \widetilde{g_{X_{0*}}}),\quad s\in \mathbb{T}, t\in[0,T],\\
&\;(\tilde{X}-\widetilde{X_{0*}})(s,0) = (X_0-X_{0*})(s).
\label{eqn: equation for the difference from the initial steady state}
\end{split}
\end{equation}
Similar to \eqref{eqn: space time estimate for the difference of solutions in proving uniqueness}, we use the assumptions \eqref{eqn: uniform bound in the lemma for the small energy regularity} and \eqref{eqn: uniform bi lipschitz constant in the lemma for the small energy regularity} and Corollary \ref{coro: H2 estimate for g_X0-g_X1} with $\mu = 2$ to find that for $\forall\, t\in[0,T]$ with $T\leq 1$ and $\forall\, \delta\in(0,1]$,
\begin{equation}
\left\|g_{X}-g_{X_{0*}}\right\|_{L_{t}^2\dot{H}^2}
\leq C\left(\delta^{1/2}\|X-X_{0*}\|_{L^2_t\dot{H}^3}+[\delta^{1/2}+(|\ln \delta|+1)t^{1/2}]\|X-X_{0*}\|_{L^\infty_t\dot{W}^{2,4}}\right),
\label{eqn: H2 estimate of the difference between solution and equilibrium solution}
\end{equation}
where $C = C(R,\lambda)$.
By Lemma \ref{lemma: a priori estimate of nonlocal eqn} and the interpolation inequality, for $\forall\, t\in[0,T]$ and $\forall\, \delta \in(0,1]$,
\begin{equation*}
\begin{split}
&\;\|\tilde{X}-\widetilde{X_{0*}}\|_{\dot{H}^{5/2}}^2(t)+\frac{1}{4}\|\tilde{X}-\widetilde{X_{0*}}\|_{L^2_{t} \dot{H}^{3}}^2\\
\leq &\; \|X_0-X_{0*}\|_{\dot{H}^{5/2}}^2+ 4\|\widetilde{g_X} - \widetilde{g_{X_{0*}}}\|_{L_{t}^2 \dot{H}^{2}}^2\\
\leq &\; \|X_0-X_{0*}\|_{\dot{H}^{5/2}}^2\\
&\;+ C_3 \left(\delta\|X-X_{0*}\|_{L^2_t\dot{H}^3}^2+[\delta+(|\ln \delta|+1)^2t]\|X-X_{0*}\|_{L^\infty_t\dot{H}^{1}}^{1/3}\|X-X_{0*}\|_{L^\infty_t\dot{H}^{5/2}}^{5/3}\right),
\end{split}
\end{equation*}
where $C_3 = C_3(R,\lambda)$ is a constant. For simplicity, let us assume $C_3(R,\lambda)\geq 1$.
Take $\delta = t\leq 1$ with $t\in[0,T_*]$ and we find that
\begin{equation}
\begin{split}
&\;\|\tilde{X}-\widetilde{X_{0*}}\|_{\dot{H}^{5/2}}^2(t)+\left(\frac{1}{4}-C_3 t\right)\|\tilde{X}-\widetilde{X_{0*}}\|_{L^2_{t} \dot{H}^{3}}^2\\
\leq &\; \|X_0-X_{0*}\|_{\dot{H}^{5/2}}^2+ 2C_3 t(|\ln t|+1)^2 \|\tilde{X}-\widetilde{X_{0*}}\|_{L^{\infty}_{t}\dot{H}^{1}}^{1/3}\|\tilde{X}-\widetilde{X_{0*}}\|_{L^{\infty}_{t}\dot{H}^{5/2}}^{5/3},
\end{split}
\label{eqn: estimates on the difference between X and its closest equilibrium configuration}
\end{equation}
Now we take $T_*\leq T\leq 1$ sufficiently small, s.t.
\begin{equation}
8C_3(R,\lambda) T_*(|\ln T_*|+1)^2  \leq 1
\label{eqn: expression for T* in the lemma for small energy regularity}
\end{equation}
and $x(|\ln x|+1)^2$ is increasing in $[0,T_*]$.
In this way, $C_3t\leq C_3 T_*(|\ln T_*|+1)^2\leq 1/8$.
By \eqref{eqn: estimates on the difference between X and its closest equilibrium configuration},
\begin{equation}
\begin{split}
&\;\|\tilde{X}-\widetilde{X_{0*}}\|_{\dot{H}^{5/2}}^2(t)+\frac{1}{8}\|\tilde{X}-\widetilde{X_{0*}}\|_{L^2_{t} \dot{H}^{3}}^2\\
\leq &\;\|\tilde{X}-\widetilde{X_{0*}}\|_{\dot{H}^{5/2}}^2(t)+\left(\frac{1}{4}-C_3 t\right)\|\tilde{X}-\widetilde{X_{0*}}\|_{L^2_{t} \dot{H}^{3}}^2\\
\leq &\;  \|X_0-X_{0*}\|_{\dot{H}^{5/2}}^2+ 2C_3 t(|\ln t|+1)^2 \|\tilde{X}-\widetilde{X_{0*}}\|_{L^{\infty}_{t}\dot{H}^{1}}^{1/3}\|\tilde{X}-\widetilde{X_{0*}}\|_{L^{\infty}_{t}\dot{H}^{5/2}}^{5/3}\\
\leq &\; \|X_0-X_{0*}\|_{\dot{H}^{5/2}}^2+ \frac{1}{4}\|\tilde{X}-\widetilde{X_{0*}}\|_{L^{\infty}_{T_*}\dot{H}^{5/2}}^2.
\label{eqn: energy estimate in global existence by applying a priori estimates in Appendix}
\end{split}
\end{equation}
By taking supremum in $t\in[0,T_*]$ on the left hand side, we find that
\begin{equation}
\|\tilde{X}-\widetilde{X_{0*}}\|_{L^{\infty}_{T_*}\dot{H}^{5/2}(\mathbb{T})}^2\leq \frac{4}{3}\|X_0-X_{0*}\|_{\dot{H}^{5/2}(\mathbb{T})}^2.
\label{eqn: upper bound for the growth of H2.5 norm in a short period of time in the proof}
\end{equation}
In view of Remark \ref{remark: L2 closest is also Hs closest}, \eqref{eqn: upper bound for the growth of H2.5 norm in a short period of time in the statement of the lemma} immediately follows with $T_*$ defined in \eqref{eqn: expression for T* in the lemma for small energy regularity}.

Next we shall prove the second part of the Lemma for given $T'\in(0,T_*]$.
Putting \eqref{eqn: upper bound for the growth of H2.5 norm in a short period of time in the proof} back into the third line of \eqref{eqn: energy estimate in global existence by applying a priori estimates in Appendix} and take $t= T'$, we find that
\begin{equation}
\begin{split}
&\;\|\tilde{X}-\widetilde{X_{0*}}\|_{\dot{H}^{5/2}}^2(T')+\frac{1}{8}\|\tilde{X}-\widetilde{X_{0*}}\|_{L^2_{T'} \dot{H}^{3}}^2\\
\leq &\; \|X_0-X_{0*}\|_{\dot{H}^{5/2}}^2+2C_3T'(|\ln T'|+1)^2 \left(\frac{4}{3}\right)^{5/6}\|X_0-X_{0*}\|_{\dot{H}^{5/2}}^{5/3}\|\tilde{X}-\widetilde{X_{0*}}\|_{L^{\infty}_{T'}\dot{H}^1}^{1/3}\\
\leq &\; \|X_0-X_{0*}\|_{\dot{H}^{5/2}}^2+4C_3 T'(|\ln T'|+1)^2  c^{-1/3}\|X_0-X_{0*}\|_{\dot{H}^{5/2}}^{2},
\label{eqn: equation for J before introducing the notation J}
\end{split}
\end{equation}
In the last inequality, we introduce the notation $c = \|X_0-X_{0*}\|_{\dot{H}^{5/2}(\mathbb{T})}/\|X-X_{0*}\|_{L^\infty_{T'} \dot{H}^1(\mathbb{T})}$.
Denote $J(t) = \|\tilde{X}-\widetilde{X_{0*}}\|_{\dot{H}^{5/2}(\mathbb{T})}^2(t)$. By interpolation, for $\forall\,t\in[0,T']$,
\begin{equation*}
\begin{split}
J(t)^{4/3} = \|\tilde{X}-\widetilde{X_{0*}}\|_{\dot{H}^{5/2}(\mathbb{T})}^{8/3}(t) 
\leq &\;\|\tilde{X}-\widetilde{X_{0*}}\|_{L^\infty_{T'}\dot{H}^{1}(\mathbb{T})}^{2/3}\|\tilde{X}-\widetilde{X_{0*}}\|_{\dot{H}^{3}(\mathbb{T})}^2(t)\\
= &\;c^{-2/3}J(0)^{1/3}\|\tilde{X}-\widetilde{X_{0*}}\|_{\dot{H}^{3}(\mathbb{T})}^2(t).
\end{split}
\end{equation*}
We multiply both sides of \eqref{eqn: equation for J before introducing the notation J} by $c^{-2/3}J(0)^{1/3}$ and find that 
\begin{equation}
c^{-2/3}J(0)^{1/3} J(T')+\frac{1}{8}\int_0^{T'} J(\omega)^{4/3}\,d\omega\leq (c^{-2/3}+4C_3 T'(|\ln T'|+1)^2  c^{-1})J(0)^{4/3}.
\label{eqn: simplified evolution equation for J the H2.5 difference}
\end{equation}

Now suppose the statement of the Lemma is false. Namely, for $\forall\,c>0$, there exists a solution $X^{(c)}(s,t)$ with $t\in[0,T']$, starting from some $X_0^{(c)}(s)\in H^{5/2}(\mathbb{T})$, satisfying \eqref{eqn: uniform bound in the lemma for the small energy regularity}, \eqref{eqn: uniform bi lipschitz constant in the lemma for the small energy regularity} and that
\begin{equation*}
\|X^{(c)}_0-X^{(c)}_{0*}\|_{\dot{H}^{5/2}(\mathbb{T})}\geq c \|X^{(c)}-X^{(c)}_{0*}\|_{L^\infty_{T'}\dot{H}^1(\mathbb{T})},
\end{equation*}
while for $\forall\,t\in [T'/4,T']$,
\begin{equation*}
J^{(c)}(t) = \|X^{(c)}-X^{(c)}_{0*}\|_{\dot{H}^{5/2}}^2(t)> \mathrm{e}^{-t/4}\|X^{(c)}_0-X^{(c)}_{0*}\|_{\dot{H}^{5/2}}^2 = \mathrm{e}^{- t/4}J^{(c)}(0).
\end{equation*}
Since \eqref{eqn: simplified evolution equation for J the H2.5 difference} holds with $J$ replaced by $J^{(c)}$, we find that
\begin{equation*}
c^{-2/3}\mathrm{e}^{- T'/4}J^{(c)}(0)^{4/3} +\frac{1}{8}\int_{T'/4}^{T'} \mathrm{e}^{-\omega/3}J^{(c)}(0)^{4/3}\,d\omega < (c^{-2/3}+4C_3 T'(|\ln T'|+1)^2  c^{-1})J^{(c)}(0)^{4/3},
\end{equation*}
which implies that
\begin{equation}
\frac{3}{8}\left(\mathrm{e}^{-T'/12} - \mathrm{e}^{-T'/3}\right) < c^{-2/3}\left(1-\mathrm{e}^{-T'/4}\right)+4C_3 T'(|\ln T'|+1)^2  c^{-1}.
\label{eqn: constraint for c}
\end{equation}
Since $T'\leq T_*\leq 1$,
\begin{equation*}
\mathrm{e}^{-T'/12} - \mathrm{e}^{-T'/3} > \frac{1}{4}T' \mathrm{e}^{-T'/3}> \frac{1}{6}T',\quad 1-\mathrm{e}^{-T'/4} < \frac{1}{4}T'.
\end{equation*}
Then \eqref{eqn: constraint for c} implies that
\begin{equation}
\frac{c}{4} - c^{1/3}<16C_3(|\ln T'|+1)^2.
\label{eqn: equation for c before introducing definitions of constants}
\end{equation}
Let $c_+$ be the unique positive real number such that the equality is achieved in \eqref{eqn: equation for c before introducing definitions of constants}.
Then we have
\begin{equation*}
\frac{c_+}{4} = 16C_3(|\ln T'|+1)^2 +c_+^{1/3}\leq 16C_3(|\ln T'|+1)^2 + \frac{c_+}{27}+2,
\end{equation*}
which implies that
\begin{equation}
c_+\leq C_4(R,\lambda)(|\ln T'|+1)^2.
\label{eqn: introducing C_4}
\end{equation}
Here $C_4\geq 1$ is some constant depending only on $R$ and $\lambda$; it will be used in the proof of Theorem \ref{thm: global existence near equilibrium} and Theorem \ref{thm: exponential convergence}.
Therefore, if $c\geq C_4(R,\lambda)(|\ln T'|+1)^2$, \eqref{eqn: equation for c before introducing definitions of constants} does not hold, which is a contradiction.
Hence, we proved \eqref{eqn: a lower H2.5 norm could be found} with
\begin{equation}
c_*(R,\lambda,T') = C_4(R,\lambda)(|\ln T'|+1)^2.
\label{eqn: defintion of c_*}
\end{equation}
 \eqref{eqn: a lower H2.5 norm with updated approximation could be found} immediately follows from \eqref{eqn: a lower H2.5 norm could be found} by virtue of Remark \ref{remark: L2 closest is also Hs closest}.
This completes the proof.
\end{proof}
\begin{remark}
Taking smaller $\mu$ in \eqref{eqn: H2 estimate of the difference between solution and equilibrium solution} can give sharper bound for $c_*$ in \eqref{eqn: defintion of c_*}, but that is not necessary for the remaining results. 
\qed
\end{remark}
\end{lemma}

Using Lemma \ref{lemma: bound and decay for H2.5 difference when energy difference is small}, we are able to prove Theorem \ref{thm: global existence near equilibrium}.
\begin{proof}[Proof of Theorem \ref{thm: global existence near equilibrium} (existence and uniqueness of global solution near equilibrium)]
With no loss of generality, we assume $R_{X_0} = 1$; otherwise, simply rescale $X_0$ by a factor of $R_{X_0}^{-1}$. Note that the contour dynamic formulation \eqref{eqn: contour dynamic formulation of the immersed boundary problem} is translation and scaling-invariant. Moreover, we note that the effective radius of $X(\cdot,t)$ is invariant in time, since the flow is volume-preserving. 

Define
\begin{equation}
S_\varepsilon = \left\{Z(s)\in H^{5/2}(\mathbb{T}):\,
R_Z = 1,\,\|Z-Z_* \|_{\dot{H}^{5/2}(\mathbb{T})}\leq \varepsilon\right\}
\label{eqn: def of data close to equilibrium}
\end{equation}
We claim that there exists a universal constant $\varepsilon_0$ which will be clear below, for $\forall\, Z(s)\in S_{\varepsilon_0}$, $\|Z\|_{\dot{H}^{5/2}}\leq C$ for some universal constant $C$, and
\begin{equation}
|Z(s_1)-Z(s_2)|\geq \frac{1}{\pi}|s_1 -s_2|,\quad \forall\, s_1,s_2\in \mathbb{T}.
\label{eqn: uniform lower bound for lambda in the proof of global existence}
\end{equation}
In fact,
\begin{equation*}
\begin{split}
|Z(s_1)-Z(s_2)|\geq &\;|Z_*(s_1)-Z_*(s_2)|-|(Z_*-Z)(s_1)-(Z_*-Z)(s_2)|\\
\geq &\;\frac{2}{\pi}|s_1 -s_2| - \|Z-Z_*\|_{\dot{C}^1(\mathbb{T})}|s_1-s_2|\\
\geq &\;\left(\frac{2}{\pi}-C_5\varepsilon_0\right)|s_1 -s_2|,
\end{split}
\end{equation*}
where $C_5>0$ is a universal constant coming from Sobolev inequality.
Hence, it suffices to take $\varepsilon_0=\min\{(C_5 \pi)^{-1},1\}$;
that $\|Z\|_{\dot{H}^{5/2}}\leq C$ is obvious.

The above uniform estimates, together with Theorem \ref{thm: local in time existence} and Theorem \ref{thm: local in time uniqueness}, imply that there is a universal constant $T_0\in(0,1)$, s.t.\;for $\forall\, X_0\in S_{\varepsilon_0}$, there is a unique solution $X(s,t)$ for \eqref{eqn: contour dynamic formulation of the immersed boundary problem} in $C_{[0,T_0]}H^{5/2}\cap L^2_{T_0}H^3(\mathbb{T})$ starting from $X_0$, s.t.
\begin{equation}
\|X\|_{L^\infty_{T_0} \dot{H}^{5/2}\cap L^2_{T_0} \dot{H}^{3}(\mathbb{T})}\leq 4\|X_0\|_{\dot{H}^{5/2}(\mathbb{T})}\leq 4(\|X_{0*}\|_{\dot{H}^{5/2}(\mathbb{T})}+\varepsilon_0) \triangleq C_6,
\label{eqn: uniform bound of the family of solution}
\end{equation}
where $C_6$ is a universal constant. Moreover, for $\forall\, s_1,s_2\in\mathbb{T}$ and $t\in[0,T_0]$,
\begin{equation}
\left|X(s_1,t) - X(s_2,t)\right| \geq \frac{1}{2\pi}|s_1 - s_2|.
\label{eqn: uniform bi lipschitz constant of the family of solution}
\end{equation}
That is, $X(s,t)$ satisfies the assumption of Lemma \ref{lemma: bound and decay for H2.5 difference when energy difference is small} with $T = T_0$, $R = C_6$, and $\lambda =(2\pi)^{-1}$, which are all universal constants. Hence, by Lemma \ref{lemma: bound and decay for H2.5 difference when energy difference is small}, there exists a universal constant $T_* = T_*(T_0, C_6, 1/(2\pi))\in(0,T_0]$ such that
\begin{equation*}
\|X-X_*\|_{L^\infty_{T_*}\dot{H}^{5/2}(\mathbb{T})}\leq \sqrt{2}\|X_0-X_{0*}\|_{\dot{H}^{5/2}(\mathbb{T})}.
\end{equation*}

To this end, we shall first investigate $\|\tilde{X}-\widetilde{X_{0*}}\|_{L^{\infty}_{[0,t]}\dot{H}^1(\mathbb{T})}$. Using the equation for $\tilde{X}$ (see \eqref{eqn: equation for oscillation of X in the main thm}), we find that for $\forall\, t\in[0,T_0]$,
\begin{equation}
\begin{split}
\|\tilde{X}-\widetilde{X_{0*}}\|_{L^{\infty}_{t}\dot{H}^1(\mathbb{T})} \leq &\; \|\tilde{X}-\widetilde{X_0}\|_{L^{\infty}_{t}\dot{H}^1(\mathbb{T})}+\|\widetilde{X_0}-\widetilde{X_{0*}}\|_{\dot{H}^1(\mathbb{T})}\\
\leq &\; \int_0^t\|\partial_t \tilde{X}\|_{\dot{H}^1(\mathbb{T})}(\tau)\,d\tau+\|\widetilde{X_0}-\widetilde{X_{0*}}\|_{\dot{H}^1(\mathbb{T})}\\
\leq &\; \int_0^t\|\mathcal{L}\tilde{X}\|_{\dot{H}^1(\mathbb{T})}(\tau)+\|\widetilde{g_{\tilde{X}}}\|_{\dot{H}^1(\mathbb{T})}(\tau)\,d\tau+\|\widetilde{X_0}-\widetilde{X_{0*}}\|_{\dot{H}^1(\mathbb{T})}.
\label{eqn: estimate for H1 difference of X and X0*}
\end{split}
\end{equation}
In order to give an estimate for $\|\widetilde{g_{\tilde{X}}}\|_{\dot{H}^1}$, we should go back to \eqref{eqn: introduce the notation Gamma_1} and \eqref{eqn: rough pointwise estimate of Gamma} and apply Lemma \ref{lemma: estimates for L M N}. Indeed, with \eqref{eqn: uniform bound of the family of solution} and \eqref{eqn: uniform bi lipschitz constant of the family of solution}, we have
\begin{equation*}
\|g_{\tilde{X}}'\|_{L^2(\mathbb{T})}(t)\leq C \|\tilde{X}\|_{\dot{H}^2(\mathbb{T})}^2(t)\|\tilde{X}'\|_{L^\infty(\mathbb{T})}(t)\leq C,
\end{equation*}
where $C$ is a universal constant.
Hence, by \eqref{eqn: uniform bound of the family of solution}, \eqref{eqn: estimate for H1 difference of X and X0*} and Lemma \ref{lemma: estimates concerning closest equilbrium},
\begin{equation}
\begin{split}
\|\tilde{X}-\widetilde{X_{0*}}\|_{L^{\infty}_{t}\dot{H}^1} \leq &\;\int_0^t C\left(\|\tilde{X}\|_{\dot{H}^2}(\tau)+1\right)\,d\tau+\|\widetilde{X_0}-\widetilde{X_{0*}}\|_{\dot{H}^{1}}\\
\leq &\;C_7 t+2\left(\|\widetilde{X_0}\|_{\dot{H}^1}^2-\|\widetilde{X_{0*}}\|_{\dot{H}^1}^2\right)^{1/2}\\
\triangleq &\;C_7t +2\zeta_{X_0},
\end{split}
\label{eqn: bound for H1 norm of the difference to the equilibrium}
\end{equation}
where $C_7$ is a universal constant. 
Here we applied Lemma \ref{lemma: estimates concerning closest equilbrium}
and defined $\zeta_{X_0}^2 = \|X_0\|_{\dot{H}^1}^2 - \|X_{0*}\|_{\dot{H}^1}^2$.
The above estimate is true as long as $X_0\in S_{\varepsilon_0}$ and $t\in[0,T_0]$.

In what follows, we shall prove the Theorem with
\begin{equation}
\varepsilon_* = \varepsilon_0 = \min\{(C_5 \pi)^{-1},1\}
\label{eqn: definition of epsilon* in the global existence}
\end{equation}
We also take $\xi_* \leq T_*/2$ such that
\begin{equation}
2 C_4(C_7+2) (|\ln (2\xi_*)|+1)^2(2\xi_*)\leq \varepsilon_*,
\label{eqn: definition of xi* in the global existence}
\end{equation}
where $T_* = T_*(T_0, C_6, 1/(2\pi))\in(0,T_0]$ given by Lemma \ref{lemma: bound and decay for H2.5 difference when energy difference is small}, $C_4 = C_4(C_6,1/(2\pi))$ defined in \eqref{eqn: introducing C_4} and $C_7$ defined in \eqref{eqn: bound for H1 norm of the difference to the equilibrium} are all universal constants.
Hence, both $\varepsilon_*$ and $\xi_*$ are universal.
We fix $T' = 2\xi_*$, which is also a universal constant.
By Lemma \ref{lemma: estimates concerning closest equilbrium} and the assumption that $\|X_0-X_{0*}\|_{\dot{H}^1} \leq \xi_*$,
\begin{equation}
\zeta_{X_0}^2 \leq 2\|X_0-X_{0*}\|^2_{\dot{H}^1}\leq 2\xi_*^2 \leq T'^2 \leq T_*^2.
\label{eqn: T' is smaller than T_*}
\end{equation}

We are going to use mathematical induction to show existence of the global solution.
First we focus on the local solution $X(s,t)$ for $t\in[0,T']$.
By \eqref{eqn: bound for H1 norm of the difference to the equilibrium} and \eqref{eqn: T' is smaller than T_*},
\begin{equation}
\|X-X_{0*}\|_{L^\infty_{T'}\dot{H}^1}\leq (C_7+2)T'.
\label{eqn: a final bound in the proof of global existence for H1 norm of X-X0* in 0 to T'}
\end{equation}
We apply Lemma \ref{lemma: bound and decay for H2.5 difference when energy difference is small} to obtain the constant $c_* = c_*(C_6,1/(2\pi),T')$, and claim that the assumption \eqref{eqn: condition in the small energy lemma initial H2.5 norm is much larger than H1 norm on the whole interval} holds if $\|X_0-X_{0*}\|_{\dot{H}^{5/2}(\mathbb{T})}\geq \varepsilon_*/2$.
In fact, by \eqref{eqn: defintion of c_*}, \eqref{eqn: definition of xi* in the global existence} and \eqref{eqn: a final bound in the proof of global existence for H1 norm of X-X0* in 0 to T'},
\begin{equation}
\begin{split}
c_*\|X-X_{0*}\|_{L^\infty_{T'}\dot{H}^1}\leq &\; C_4(|\ln T'|+1)^2\cdot (C_7+2)T'= C_4(C_7+2) (|\ln (2\xi_*)|+1)^2(2\xi_*)\leq \varepsilon_*/2.
\end{split}
\label{eqn: proof of the threshold of H2.5 norm in the proof of global existence}
\end{equation}
Therefore, if $\|X_0-X_{0*}\|_{\dot{H}^{5/2}(\mathbb{T})} \in[ \varepsilon_*/2, \varepsilon_*]$, by \eqref{eqn: proof of the threshold of H2.5 norm in the proof of global existence} and Lemma \ref{lemma: bound and decay for H2.5 difference when energy difference is small}, there exists $t_1\in [T'/4,T']$, s.t.
\begin{equation*}
\|X-X_*\|_{\dot{H}^{5/2}}(t_1)\leq e^{-t_1/8}\|X_0-X_{0*}\|_{\dot{H}^{5/2}}\leq \varepsilon_*.
\end{equation*}
Otherwise, if $\|X_0-X_{0*}\|_{\dot{H}^{5/2}(\mathbb{T})} \leq \varepsilon_*/2$, by the fact that $T'\leq T_*$ and Lemma \ref{lemma: bound and decay for H2.5 difference when energy difference is small}, there exists $t_1\in [T'/4,T']$, s.t.
\begin{equation*}
\|X-X_{*}\|_{\dot{H}^{5/2}(\mathbb{T})}(t_1)\leq \|X-X_{*}\|_{L^{\infty}_{T'}\dot{H}^{5/2}(\mathbb{T})}\leq \sqrt{2}\|X_0-X_{0*}\|_{\dot{H}^{5/2}(\mathbb{T})} \leq \varepsilon_*.
\end{equation*}
This implies that for all $X_0\in S_{\varepsilon_*}$, we can always find $t_1\in [T'/4,T']$, such that the unique local solution of \eqref{eqn: contour dynamic formulation of the immersed boundary problem} in $C_{[0,t_1]}H^{5/2}\cap L^2_{t_1}H^{3}(\mathbb{T})$ satisfies that
\begin{align*}
&\;\|X-X_{*}\|_{L^{\infty}_{t_1}\dot{H}^{5/2}(\mathbb{T})}\leq  \sqrt{2}\varepsilon_*,\\
&\;|X(s_1,t) - X(s_2,t)| \geq \frac{1}{2\pi}|s_1 - s_2|,\quad \forall \,t\in[0,t_1],\;s_1,s_2\in\mathbb{T},\\
&\;X(t_1)\in S_{\varepsilon_*}.
\end{align*}
We note that $T'$ is a universal constant.

Suppose we have found $t_k$'s for $k\leq n$, satisfying that for $\forall\, k =1,\cdots,n$,
\begin{enumerate}
\item $t_k\in[T'/4,T']$.
\item There exists a unique solution $X$ of \eqref{eqn: contour dynamic formulation of the immersed boundary problem} in $C_{[0,T_{n}]}H^{5/2}\cap L^2_{T_{n}}H^{3}(\mathbb{T})$, where $T_k = \sum_{i=1}^k t_i$ for $i = 1,\cdots,n$, such that
\begin{align}
&\;\|X-X_{*}\|_{L^{\infty}_{[0,T_k]}\dot{H}^{5/2}(\mathbb{T})}\leq  \sqrt{2}\varepsilon_*, \label{eqn: estimates on the distance to the equilibrium for the global solution in first k-th time intervals}\\
&\;|X(s_1,t) - X(s_2,t)| \geq \frac{1}{2\pi}|s_1 - s_2|,\quad \forall \,t\in[0,T_k],\;s_1,s_2\in\mathbb{T},\label{eqn: well-stretched constant estimates for the global solution in first k-th time intervals}\\
&\;X(\cdot,T_k) \in S_{\varepsilon_*}.
\end{align}
\end{enumerate}
Now let us restart the equation at $t = T_n$. To be more precise, we consider
\begin{equation*}
\partial_t X(s,t)= \mathcal{L}X(s,t) + g_X(s,t),\quad s\in \mathbb{T}, t\geq T_n,
\end{equation*}
with $X(\cdot,T_n)\in S_{\varepsilon_*}= S_{\varepsilon_0}$ given.
As before, there exists a unique local solution $X(s,t)$ for $t\in [T_n,T_n+T_0]$ satisfying the uniform estimates \eqref{eqn: uniform bound of the family of solution} and \eqref{eqn: uniform bi lipschitz constant of the family of solution} for solutions starting in $S_{\varepsilon_*}$.
Moreover, with $T_*$ and $T'$ defined as before,
\begin{equation*}
\|X-X_{*}\|_{L^{\infty}_{[T_n,T_n+T']}\dot{H}^{5/2}(\mathbb{T})}\leq \sqrt{2}\|X_{T_n}-(X_{T_n})_*\|_{\dot{H}^{5/2}(\mathbb{T})}.
\end{equation*}
By \eqref{eqn: bound for H1 norm of the difference to the equilibrium},
\begin{equation}
\|X-(X_{T_n})_*\|_{L^{\infty}_{[T_n,T_n+T']}\dot{H}^1} \leq C_7 T'+2\left(\|X_{T_n}\|_{\dot{H}^1}^2-\|(X_{T_n})_*\|_{\dot{H}^1}^2\right)^{1/2} = C_7 T'+2\zeta_{X_{T_n}},
\label{eqn: crude form of the bound for H1 norm of the difference to the equilibrium for later time}
\end{equation}
where $X_{T_n}(s) \triangleq X(s,T_n)$.
Since the solution obtained in $[0,T_n]$ satisfies the assumption of Lemma \ref{lemma: energy estimate}, by \eqref{eqn: energy estimate of Stokes immersed boundary problem}, 
\begin{equation*}
\zeta_{X_{T_n}}^2=\|X_{T_n}\|_{\dot{H}^1}^2-\|(X_{T_n})_*\|_{\dot{H}^1}^2 \leq \|X_0\|_{\dot{H}^1}^2-\|X_{0*}\|_{\dot{H}^1}^2 = \zeta_{X_{0}}^2=T'^2.
\end{equation*}
Note that $\|(X_{T_n})_*\|_{\dot{H}^1} = \|X_{0*}\|_{\dot{H}^1}$.
Hence, 
$\|X-(X_{T_n})_*\|_{L^{\infty}_{[T_n,T_n+T']}\dot{H}^1} \leq (C_7+2)T'$.
To this end, we simply argue to show  $\exists \,t_{n+1}\in [T'/4,T']$, s.t.\;there exists a unique local solution in $C_{[0,T_{n+1}]}H^{5/2}\cap L^2_{T_{n+1}}H^{3}(\mathbb{T})$ with $T_{n+1} = T_n+t_{n+1}$ and $X(T_{n+1})\in S_{\varepsilon_*}$.
Estimates \eqref{eqn: estimates on the distance to the equilibrium for the global solution in first k-th time intervals} and \eqref{eqn: well-stretched constant estimates for the global solution in first k-th time intervals} in the new time interval $[0, T_{n+1}]$ follow as before.
Since $T_n \geq nT'/4$ with $T'>0$ being a universal constant, $T_n\rightarrow +\infty$ as $n\rightarrow \infty$.
The existence of global solution is thus established.
The uniqueness follows from Theorem \ref{thm: local in time uniqueness}.
That $X_t\in L^2_{[0,+\infty),loc}H^2(\mathbb{T})$ follows from Theorem \ref{thm: local in time existence}.
Estimates \eqref{eqn: estimates on the distance to the equilibrium for the global solution in all time intervals}, \eqref{eqn: well-stretched constant estimates for the global solution in all time intervals} and \eqref{eqn: uniform bound of H 2.5 norm for the global solution} are established in the induction.
\end{proof}
\begin{remark}
Instead of \eqref{eqn: definition of epsilon* in the global existence}, we may take arbitrary $\varepsilon_*\in(0,\varepsilon_0]$, and the same proof still works.
\qed
\end{remark}

The main idea in the proof of Theorem \ref{thm: global existence near equilibrium} is that when the string configuration is close to an equilibrium, $\|X_0-X_{0*}\|_{\dot{H}^1}$ sets a bound for $\|X-X_*\|_{\dot{H}^{5/2}}$ in an indirect way (at least within a short time).
In the same spirit, we prove the following corollary with refined estimates.
It will be useful in the proof of Theorem \ref{thm: exponential convergence}.

\begin{corollary}\label{coro: refined decay estimate of global solution}
Let $X_0\in H^{5/2}(\mathbb{T})$ satisfy all the assumptions of Theorem \ref{thm: global existence near equilibrium} and let $X$ be the unique global solution of \eqref{eqn: contour dynamic formulation of the immersed boundary problem} starting from $X_0$ obtained in Theorem \ref{thm: global existence near equilibrium}. Then for any given $\xi\in(0,\xi_*]$, if in addition
\begin{equation}
\|X_0(s) - X_{0*}(s)\|_{\dot{H}^{1}(\mathbb{T})}\leq \xi R_{X_0},
\label{eqn: closeness condition of H 1 norm in corollary}
\end{equation}
then the solution $X$ satisfies that for $\forall \,t\geq 0$, 
\begin{equation}
\|X-X_{*}\|_{\dot{H}^{5/2}(\mathbb{T})}(t)\leq \max \{2e^{-t/8}\|X_0-X_{0*}\|_{\dot{H}^{5/2}(\mathbb{T})},\varepsilon_\xi R_{X_0}\},
\label{eqn: refined H2.5 bound of global solution}
\end{equation}
with
\begin{equation}
\varepsilon_\xi \triangleq  2C_4(C_7+2) (|\ln (2\xi)|+1)^2(2\xi),\quad \xi>0,
\label{eqn: defintion of varepsilon_* in the corollary}
\end{equation}
where $C_4 = C_4(C_6,1/(2\pi))$ and $C_7$ are universal constants defined in \eqref{eqn: introducing C_4} and \eqref{eqn: bound for H1 norm of the difference to the equilibrium} respectively.
\begin{remark}
We only define $\varepsilon_\xi$ for $\xi>0$ in order to avoid abusing the notation $\varepsilon_0$ defined in the proof of Theorem \ref{thm: global existence near equilibrium}. 
\end{remark}
\begin{proof}
We follow exactly the proof of Theorem \ref{thm: global existence near equilibrium} until the definition of $T'$.
Now we define $T' = 2\xi$ instead.
%
It is worthwhile to note that
\begin{equation*}
\zeta_{X_0}^2\leq 2\|X_0-X_{0*}\|^2_{\dot{H}^1}\leq 2\xi^2 \leq T'^2\leq T_*^2.
\end{equation*}
For the solution $X(s,t)$ in $t\in[0,T']$, \eqref{eqn: a final bound in the proof of global existence for H1 norm of X-X0* in 0 to T'} still holds.
With $c_* = C_4(C_6,1/(2\pi))(|\ln T'|+1)^2$ as before, we have a similar estimate as \eqref{eqn: proof of the threshold of H2.5 norm in the proof of global existence}
%
\begin{equation}
c_*\|X-X_{0*}\|_{L^\infty_{T'}\dot{H}^1}\leq C_4(|\ln T'|+1)^2\cdot (C_7+2)T' = C_4(C_7+2) (|\ln (2\xi)|+1)^2(2\xi)= \varepsilon_\xi/2.
\label{eqn: new threshold in the corollary such that the assumption holds}
\end{equation}
Therefore, if $\|X_0-X_{0*}\|_{\dot{H}^{5/2}(\mathbb{T})} \geq \varepsilon_\xi/2$, the assumption \eqref{eqn: condition in the small energy lemma initial H2.5 norm is much larger than H1 norm on the whole interval} holds. 
By Lemma \ref{lemma: bound and decay for H2.5 difference when energy difference is small}, there exists $t_1\in [T'/4,T']$, s.t.
\begin{equation*}
\|X-X_*\|_{\dot{H}^{5/2}}(t_1)\leq e^{-t_1/8}\|X_0-X_{0*}\|_{\dot{H}^{5/2}}. 
\end{equation*}
Otherwise, if $\|X_0-X_{0*}\|_{\dot{H}^{5/2}(\mathbb{T})} \leq \varepsilon_\xi/2$, by Lemma \ref{lemma: bound and decay for H2.5 difference when energy difference is small}, there exists $t_1\in [T'/4,T']$, s.t.
\begin{equation*}
\|X-X_{*}\|_{\dot{H}^{5/2}(\mathbb{T})}(t_1)\leq \|X-X_{*}\|_{L^{\infty}_{T'}\dot{H}^{5/2}(\mathbb{T})}\leq \sqrt{2}\|X_0-X_{0*}\|_{\dot{H}^{5/2}(\mathbb{T})} \leq \varepsilon_\xi/\sqrt{2}.
\end{equation*}
This implies that there always exists $t_1\in [T'/4,T']$, s.t.
\begin{equation*}
\|X-X_{*}\|_{\dot{H}^{5/2}(\mathbb{T})}(t_1)\leq \max\{e^{-t_1/8}\|X_0-X_{0*}\|_{\dot{H}^{5/2}}, \varepsilon_\xi/\sqrt{2}\}.
\end{equation*}
Now suppose that we have found $t_k$'s for $k\leq n$, satisfying that for $\forall\, k =1,\cdots,n$, $t_k\in[T'/4,T']$, and
\begin{equation*}
\|X-X_{*}\|_{\dot{H}^{5/2}(\mathbb{T})}(T_k)\leq \max\{e^{-T_k/8}\|X_0-X_{0*}\|_{\dot{H}^{5/2}}, \varepsilon_\xi/\sqrt{2}\}.
\end{equation*}
where $T_k = \sum_{i=1}^k t_i$ for $i = 1,\cdots,n$.
Now we consider the equation in $t \in [T_n, T_n+T']$.
As in \eqref{eqn: crude form of the bound for H1 norm of the difference to the equilibrium for later time},
\begin{equation*}
\|X-(X_{T_n})_*\|_{L^{\infty}_{[T_n,T_n+T']}\dot{H}^1} \leq  C_7 T'+2\zeta_{X_{T_n}}\leq C_7 T'+2\zeta_{X_0} \leq (C_7+2)T'.
\end{equation*}
Here we used the energy estimate \eqref{eqn: energy estimate of Stokes immersed boundary problem} again.
Hence, as in \eqref{eqn: new threshold in the corollary such that the assumption holds}, we have $c_*\|X-X_{0*}\|_{L^\infty_{[T_n,T_n+T']}\dot{H}^1}\leq \varepsilon_\xi/2$.
We argue as before to find that there always exists $t_{n+1}\in[T'/4,T']$, s.t.
\begin{equation*}
\begin{split}
\|X-X_{*}\|_{\dot{H}^{5/2}(\mathbb{T})}(T_n+t_{n+1})\leq &\;\max\{e^{-t_{n+1}/8}\|X-X_*\|_{\dot{H}^{5/2}}(T_n), \varepsilon_\xi/\sqrt{2}\}\\
\leq &\;\max\{e^{-(T_n+t_{n+1})/8}\|X_0-X_{0*}\|_{\dot{H}^{5/2}}, \varepsilon_\xi/\sqrt{2}\}.
\end{split}
\end{equation*}
By induction, there exists a sequence $\{t_k\}_{k\in\mathbb{Z}_+}$, $t_k\in [T'/4, T']$, such that for $\forall\, k\in\mathbb{Z}_+$,
\begin{equation*}
\|X-X_{*}\|_{\dot{H}^{5/2}(\mathbb{T})}(T_k)\leq \max\{e^{-T_k/8}\|X_0-X_{0*}\|_{\dot{H}^{5/2}}, \varepsilon_\xi/\sqrt{2}\}.
\end{equation*}
where $T_k = \sum_{i=1}^k t_i\rightarrow +\infty$.
Since $t_k\leq T'\leq T_*\leq 1$, by Lemma \ref{lemma: bound and decay for H2.5 difference when energy difference is small},
\begin{equation*}
\begin{split}
\|X-X_{*}\|_{L^\infty_{[T_{k-1}, T_k]}\dot{H}^{5/2}(\mathbb{T})}\leq &\;\sqrt{2}\|X-X_{*}\|_{\dot{H}^{5/2}(\mathbb{T})}(T_{k-1})\\
\leq &\;\max\{\sqrt{2}e^{-T_{k-1}/8}\|X_0-X_{0*}\|_{\dot{H}^{5/2}(\mathbb{T})}, \varepsilon_\xi\}\\
\leq &\;\max\{\sqrt{2}e^{T'/8}e^{-T_{k}/8}\|X_0-X_{0*}\|_{\dot{H}^{5/2}(\mathbb{T})}, \varepsilon_\xi\}\\
\leq &\;\max\{2e^{-T_{k}/8}\|X_0-X_{0*}\|_{\dot{H}^{5/2}(\mathbb{T})}, \varepsilon_\xi\}.
\end{split}
\end{equation*}
Note that here we are abusing the notation $T_0$ by defining $T_0 = 0$; it does not refer to the $T_0$ in Theorem \ref{thm: local in time existence}. 
Using the fact that $X\in C_{[0,+\infty)}H^{5/2}(\mathbb{T})$, 
for $\forall\,t\in [T_{k-1}, T_k]$,
\begin{equation*}
\|X-X_{*}\|_{\dot{H}^{5/2}(\mathbb{T})}(t)\leq \max\{2e^{-t/8}\|X_0-X_{0*}\|_{\dot{H}^{5/2}(\mathbb{T})}, \varepsilon_\xi\}.
\end{equation*}
This completes the proof. 
\end{proof}
\end{corollary}

\section{Exponential Convergence to Equilibrium Configurations}\label{section: exp convergence}
In this section, we shall prove that the global-in-time solution near equilibriums obtained in Theorem \ref{thm: global existence near equilibrium} converges exponentially in the $H^s$-sense to an equilibrium configuration as $t\rightarrow +\infty$.
See the statement of Theorem \ref{thm: exponential convergence}.
In the sequel, we shall always consider the contour dynamic formulation \eqref{eqn: contour dynamic formulation of the immersed boundary problem}, with $X_0\in H^{5/2}(\mathbb{T})$ satisfying \eqref{eqn: closeness condition of H 2.5 norm} and \eqref{eqn: closeness condition of H 1 norm} with $\varepsilon_*, \xi_*>0$ found in Theorem \ref{thm: global existence near equilibrium}.
Without loss of generality, we assume $R_{X_0} = 1$.

\subsection{A lower bound of the rate of energy dissipation}\label{section: lower bound for energy dissipation rate}
A key step to prove the exponential convergence of the global solution near equilibrium is to establish a lower bound of the rate of energy dissipation $\int_{\mathbb{R}^2}|\nabla u_X|^2\,dx$ (see Lemma \ref{lemma: energy estimate}) in terms of $\|X-X_*\|_{\dot{H}^1}$ provided that the latter is sufficiently small. 

Let $S_\varepsilon$ be defined as in \eqref{eqn: def of data close to equilibrium}. Let $\varepsilon_*' \in(0,\varepsilon_*)$ to be determined.
Let $\Omega_X\subset \mathbb{R}^2$ denote the bounded open domain enclosed by $X(\mathbb{T})$ where $X\in  S_{\sqrt{2}\varepsilon'_*} $.
Here the constant $\sqrt{2}$ comes from the estimate \eqref{eqn: estimates on the distance to the equilibrium for the global solution in all time intervals} of the global solution.
Define the collection of all such domains to be
\begin{equation*}
\mathcal{M}_{\varepsilon_*'} = \{\Omega_X \Subset\mathbb{R}^2:\; \partial \Omega_X = X(\mathbb{T}), \;X\in  S_{\sqrt{2}\varepsilon'_*} \}.
\end{equation*}
We assume that $\varepsilon_*'$ is sufficiently small, such that domains in $\mathcal{M}_{\varepsilon_*'}$ satisfy uniform $C^{1}$-regularity condition with uniform constants (see \cite{adams2003sobolev} in \S\,4.10 for the rigorous definition).
Indeed, this is achievable due to the implicit function theorem and the Sobolev embedding $H^{5/2}(\mathbb{T})\hookrightarrow C^{1,\alpha}(\mathbb{T})$ for $\forall\, \alpha\in(0,1)$.

Let $u_X$ be the velocity field determined by a configuration $X\in  S_{\sqrt{2}\varepsilon'_*} $. Let
\begin{equation*}
(u_X)_{\Omega_X}= |\Omega_X|^{-1}\int_{\Omega_X} u_X\,dx,\quad (u_X)_{\partial\Omega_X} = |\partial \Omega_X|^{-1}\int_{\partial\Omega_X} u_X\,dl,
\end{equation*}
where $l$ is the arc-length parameter of $\partial\Omega_X$.
Then by the boundary trace theorem \cite{adams2003sobolev},
\begin{equation*}
\begin{split}
\int_{\mathbb{R}^2} |\nabla u_X|^2 \,dx \geq &\;\int_{\Omega_X} |\nabla u_X|^2 \,dx \geq C \int_{\partial\Omega_X} |u_X-(u_X)_{\Omega_X}|^2\,dl \\
\geq &\;C \int_{\partial\Omega_X} |u_X-(u_X)_{\partial\Omega_X}|^2\,dl = C\int_{\mathbb{T}} |u_X-(u_X)_{\partial\Omega_X}|^2 |X'(s)|\,ds.
\end{split}
\end{equation*}
Here we used the fact that $dl = |X'(s)|\,ds$, since $s$ is a monotone parameterization of $\partial\Omega_X$.
Thanks to the uniform $C^1$-regularity of $\Omega_X\in \mathcal{M}_{\varepsilon_*'}$, the constant $C$ is uniform for $\forall\, X\in  S_{\sqrt{2}\varepsilon'_*} $.
We derive that
\begin{equation*}
\begin{split}
\int_{\mathbb{T}} |u_X(X(s))-(u_X)_{\partial\Omega_X}|^2 |X'(s)|\,ds \geq &\;\int_{\mathbb{T}} |u_X(X(s))-(u_X)_{\partial\Omega_X}|^2 (|X_*'(s)|-|X_*'(s)-X'(s)|)\,ds\\
\geq &\;(1- C_5\varepsilon_*')\int_{\mathbb{T}} |u_X(X(s))-(u_X)_{\partial\Omega_X}|^2\,ds\\
\geq &\;(1- C_5\varepsilon_*')\int_{\mathbb{T}} |u_X(X(s))-\bar{u}_X|^2\,ds.
\end{split}
\end{equation*}
Here $\bar{u}_X = |\mathbb{T}|^{-1}\int_{\mathbb{T}} u_X(X(s))\,ds$.
Again, the constant $C_5$, which first showed up in the proof of Theorem \ref{thm: global existence near equilibrium}, comes from the Sobolev embedding $H^{5/2}(\mathbb{T})\hookrightarrow C^1(\mathbb{T})$ and is independent of $X\in S_{\sqrt{2}\varepsilon'_*}$.
Taking $\varepsilon_*'\leq (2C_5)^{-1}$, we obtain that
\begin{equation}
\int_{\mathbb{R}^2} |\nabla u_X|^2 \,dx \geq C\int_{\mathbb{T}} |u_X(X(s))-\bar{u}_X|^2\,ds
\label{eqn: energy dissipation rate can be bounded from below by the L2 norm of velocity oscillation}
\end{equation}
for some universal constant $C$ independent of $X\in  S_{\sqrt{2}\varepsilon'_*} $.

Now we turn to derive a lower bound for $\int_{\mathbb{T}} |u_X(X(s))-\bar{u}_X|^2\,ds$.
We are going to perform linearization of $u_X(X(s))$ around the equilibrium configuration $X_*$.
Fix $X\in S_{\sqrt{2}\varepsilon'_*} $, with $\varepsilon_*'\leq \min\{1,\varepsilon_*\}$ satisfying all the assumptions above and to be determined.
Let $D(s) = X(s)-X_*(s)$ and
\begin{equation}
X_{\eta}(\cdot) \triangleq  X_*(\cdot) +\eta D(\cdot),\quad\eta \in[0,1].
\label{eqn: defintion of X_eta}
\end{equation}
By definition, $\|D\|_{\dot{H}^{5/2}(\mathbb{T})}\leq \sqrt{2}\varepsilon_*'$.
It is easy to show that with $\varepsilon_*'$ sufficiently small,
\begin{align}
&\;\|X_\eta\|_{\dot{H}^{5/2}(\mathbb{T})}\leq C,\quad \forall\, \eta\in[0,1],\label{eqn: uniform H2.5 upper bound for the family of configurations near equilibrium}\\
&\;|X_\eta(s_1) - X_\eta(s_2)|\geq \frac{1}{\pi}|s_1-s_2|,\quad \forall\,s_1,s_2\in\mathbb{T},\;\forall\,\eta\in[0,1],\label{eqn: uniform stretching constant for the family of configurations near equilibrium}
\end{align}
where $C$ is a universal constant.
Note that with $\varepsilon_*'$ being sufficiently small, $X_\eta$ is also a non-self-intersecting string configuration, but it may not be in $ S_{\sqrt{2}\varepsilon'_*} $, as $R_{X_{\eta}} = 1$ is not necessarily true.

The following lemma shows that $u_X(X(s))$, as a function of $X$, can be well approximated by linearization around $X_*$.
\begin{lemma}\label{lemma: linearization of velocity field around equilibrium}
Assume $\varepsilon_*'\leq \min\{1,\varepsilon_*, (2C_5)^{-1}\}$ is sufficiently small such that domains in $\mathcal{M}_{\varepsilon_*'}$ satisfy uniform $C^{1}$-regularity condition with uniform constants, and \eqref{eqn: uniform H2.5 upper bound for the family of configurations near equilibrium} and \eqref{eqn: uniform stretching constant for the family of configurations near equilibrium} hold. Then
\begin{equation}
u_X(X(s)) = \left.\frac{\partial}{\partial\eta}\right|_{\eta = 0}u_{X_\eta}(X_\eta(s))+\mathcal{R}_X(s),
\label{eqn: first approximation by linearization of velocity around equilibrium}
\end{equation}
where
\begin{equation}
\|\mathcal{R}_X(s)\|_{L^\infty(\mathbb{T})}\leq C\varepsilon_*' \|D\|_{\dot{H}^1(\mathbb{T})},
\label{eqn: estimate of the higher order error term in estimating the velocity around equilibrium}
\end{equation}
with $C$ being a universal constant.
\begin{proof}
Recall that $u_X$ is given by \eqref{eqn: velocity of membrane}.
By \eqref{eqn: simplification of integrand of g_X part 1},
\begin{equation*}
\begin{split}
u_X(X(s)) = &\;\frac{1}{4\pi}\int_{\mathbb{T}}\frac{L_{X}\cdot X'(s')}{|L_{X}|^2}M_{X}-\frac{L_{X}\cdot M_{X}}{|L_{X}|^2}X'(s') -\frac{X'(s')\cdot M_{X}}{|L_{X}|^2}L_{X}\,ds'\\
&\;+\frac{1}{4\pi}\int_{\mathbb{T}}\frac{2L_{X}\cdot X'(s')L_{X}\cdot M_{X}}{|L_{X}|^4}L_{X}\,ds',
\end{split}
\end{equation*}
where $L_X = L_X(s,s')$ and $M_X = M_X(s,s')$ are defined in \eqref{eqn: definition of L M N} and \eqref{eqn: definition of L M N at s}. The subscripts stress that they are determined by $X$.
Since $u_{X_*} \equiv 0$, by mean value theorem with respect to $\eta$, there exists an $\eta_*\in[0,1]$ such that,
\begin{equation*}
u_X(X(s)) = u_X(X(s))  -u_{X_*}(X_*(s))  = \left.\frac{\partial}{\partial\eta}\right|_{\eta = \eta_*}u_{X_\eta}(X_\eta(s)).
\end{equation*}
In Lemma \ref{lemma: eta derivative and the integral in u_X commute} in the Appendix \ref{appendix section: auxiliary calculations}, we will show that the $\eta$-derivative and the integral in $u_{X_\eta}$ commute. Hence,
\begin{equation}
\begin{split}
u_X(X(s)) = &\;\frac{1}{4\pi}\int_{\mathbb{T}}\frac{L_{D}\cdot X'_{\eta_*}(s')}{|L_{X_{\eta_*}}|^2}M_{X_{\eta_*}} +\frac{L_{X_{\eta_*}}\cdot D'(s')}{|L_{X_{\eta_*}}|^2}M_{X_{\eta_*}} +\frac{L_{X_{\eta_*}}\cdot X'_{\eta_*}(s')}{|L_{X_{\eta_*}}|^2}M_{D}\,ds'\\
&\;+\frac{1}{4\pi}\int_{\mathbb{T}}-\frac{(L_{X_{\eta_*}}\cdot L_{D})(L_{X_{\eta_*}}\cdot X'_{\eta_*}(s'))}{|L_{X_{\eta_*}}|^4}M_{X_{\eta_*}}\,ds'\\
&\;-\frac{1}{4\pi}\int_{\mathbb{T}}\frac{L_{D}\cdot M_{X_{\eta_*}}}{|L_{X_{\eta_*}}|^2}X'_{\eta_*}(s')+\frac{L_{X_{\eta_*}}\cdot M_{D}}{|L_{X_{\eta_*}}|^2}X'_{\eta_*}(s')+\frac{L_{X_{\eta_*}}\cdot M_{X_{\eta_*}}}{|L_{X_{\eta_*}}|^2}D'(s')\,ds'\\
&\;+\frac{1}{4\pi}\int_{\mathbb{T}}\frac{2(L_{D}\cdot L_{X_{\eta_*}})(L_{X_{\eta_*}}\cdot M_{X_{\eta_*}})}{|L_{X_{\eta_*}}|^2}X'_{\eta_*}(s')\,ds'\\
&\;-\frac{1}{4\pi}\int_{\mathbb{T}}\frac{D'(s')\cdot M_{X_{\eta_*}}}{|L_{X_{\eta_*}}|^2}L_{X_{\eta_*}}+\frac{X'_{\eta_*}(s')\cdot M_{D}}{|L_{X_{\eta_*}}|^2}L_{X_{\eta_*}}+\frac{X'_{\eta_*}(s')\cdot M_{X_{\eta_*}}}{|L_{X_{\eta_*}}|^2}L_{D}\,ds'\\
&\;+\frac{1}{4\pi}\int_{\mathbb{T}}\frac{2(L_D\cdot L_{X_{\eta_*}})(X'_{\eta_*}(s')\cdot M_{X_{\eta_*}})}{|L_{X_{\eta_*}}|^2}L_{X_{\eta_*}}\,ds'\\
&\;+\frac{1}{4\pi}\int_{\mathbb{T}}\frac{2(L_{D}\cdot X'_{\eta_*}(s'))(L_{X_{\eta_*}}\cdot M_{X_{\eta_*}})}{|L_{X_{\eta_*}}|^4}L_{X_{\eta_*}} + \frac{2(L_{X_{\eta_*}}\cdot D'(s'))(L_{X_{\eta_*}}\cdot M_{X_{\eta_*}})}{|L_{X_{\eta_*}}|^4}L_{X_{\eta_*}}\,ds'\\
&\;+\frac{1}{4\pi}\int_{\mathbb{T}}\frac{2(L_{X_{\eta_*}}\cdot X'_{\eta_*}(s'))(L_{D}\cdot M_{X_{\eta_*}})}{|L_{X_{\eta_*}}|^4}L_{X_{\eta_*}} + \frac{2(L_{X_{\eta_*}}\cdot X'_{\eta_*}(s'))(L_{X_{\eta_*}}\cdot M_{D})}{|L_{X_{\eta_*}}|^4}L_{X_{\eta_*}}\,ds'\\
&\;+\frac{1}{4\pi}\int_{\mathbb{T}}\frac{2(L_{X_{\eta_*}}\cdot X'_{\eta_*}(s'))(L_{X_{\eta_*}}\cdot M_{X_{\eta_*}})}{|L_{X_{\eta_*}}|^4}L_{D}\,ds'\\
&\;-\frac{1}{4\pi}\int_{\mathbb{T}}\frac{8(L_{D}\cdot L_{X_{\eta_*}})(L_{X_{\eta_*}}\cdot X'_{\eta_*}(s'))(L_{D}\cdot M_{X_{\eta_*}})}{|L_{X_{\eta_*}}|^4}L_{X_{\eta_*}}\,ds'.
\end{split}
\label{eqn: representation of velocity close to equilibrium}
\end{equation}

We then replace all the $X_{\eta_*}$ in \eqref{eqn: representation of velocity close to equilibrium} by $X_*$, i.e.\;$\eta = 0$, and introduce some error denoted by $\mathcal{R}_X(s)$. In this way, we obtain the representation \eqref{eqn: first approximation by linearization of velocity around equilibrium}.
To show \eqref{eqn: estimate of the higher order error term in estimating the velocity around equilibrium}, for conciseness, we only look at one part of $\mathcal{R}_X(s)$, which is the error in approximating the first term on the right hand side of \eqref{eqn: representation of velocity close to equilibrium},
\begin{equation*}
\begin{split}
&\;\left\|\frac{1}{4\pi}\int_{\mathbb{T}}\frac{L_{D}\cdot X'_{\eta_*}(s')}{|L_{X_{\eta_*}}|^2}M_{X_{\eta_*}}-\frac{L_{D}\cdot X'_{*}(s')}{|L_{X_{*}}|^2}M_{X_*}\,ds'\right\|_{L^\infty(\mathbb{T})}\\
\leq &\;\left\|\frac{1}{4\pi}\int_{\mathbb{T}}\frac{L_{D}\cdot (X'_{\eta_*}(s') - X'_{*}(s'))}{|L_{X_{\eta_*}}|^2}M_{X_{\eta_*}}\,ds'\right\|_{L^\infty(\mathbb{T})}+\left\|\frac{1}{4\pi}\int_{\mathbb{T}}\frac{L_{D}\cdot X'_{*}(s')}{|L_{X_{\eta_*}}|^2}(M_{X_{\eta_*}}-M_{X_*})\,ds'\right\|_{L^\infty(\mathbb{T})}\\
&\;+\left\|\frac{1}{4\pi}\int_{\mathbb{T}}\frac{L_{D}\cdot X'_{*}(s')}{|L_{X_{*}}|^2}M_{X_*}\frac{|L_{X_{*}}|^2-|L_{X_{\eta_*}}|^2}{|L_{X_{\eta_*}}|^2}\,ds'\right\|_{L^\infty(\mathbb{T})}\\
\leq &\;\left\|\frac{1}{4\pi}\int_{\mathbb{T}}\frac{L_{D}\cdot \eta_* D'(s')}{|L_{X_{\eta_*}}|^2}M_{X_{\eta_*}}\,ds'\right\|_{L^\infty(\mathbb{T})}+\left\|\frac{1}{4\pi}\int_{\mathbb{T}}\frac{L_{D}\cdot X'_{*}(s')}{|L_{X_{\eta_*}}|^2}\eta_* M_{D}\,ds'\right\|_{L^\infty(\mathbb{T})}\\
&\;+\left\|\frac{1}{4\pi}\int_{\mathbb{T}}\frac{L_{D}\cdot X'_{*}(s')}{|L_{X_{*}}|^2}M_{X_*}\frac{(L_{X_{*}}+L_{X_{\eta_*}})\cdot \eta_* L_D}{|L_{X_{\eta_*}}|^2}\,ds'\right\|_{L^\infty(\mathbb{T})}.
\end{split}
\end{equation*}
Note that \eqref{eqn: lower bound for L} and \eqref{eqn: uniform stretching constant for the family of configurations near equilibrium} imply that $|L_{X_*}|, |L_{X_{\eta_*}}|\geq C$ for some universal constant $C$.
By Lemma \ref{lemma: estimates for L M N}, \eqref{eqn: uniform H2.5 upper bound for the family of configurations near equilibrium}, and \eqref{eqn: uniform stretching constant for the family of configurations near equilibrium},
\begin{equation*}
\begin{split}
&\;\left\|\frac{1}{4\pi}\int_{\mathbb{T}}\frac{L_{D}\cdot X'_{\eta_*}(s')}{|L_{X_{\eta_*}}|^2}M_{X_{\eta_*}}-\frac{L_{D}\cdot X'_{*}(s')}{|L_{X_{*}}|^2}M_{X_*}\,ds'\right\|_{L^\infty(\mathbb{T})}\\
\leq &\;C \|L_D\|_{L^\infty_s(\mathbb{T})L^2_{s'}(\mathbb{T})}\|D'\|_{L^\infty(\mathbb{T})}\|M_{X_{\eta_*}}\|_{L^\infty_s(\mathbb{T})L^2_{s'}(\mathbb{T})}\\
&\;+C \|L_D\|_{L^\infty_s(\mathbb{T})L^2_{s'}(\mathbb{T})}\|X_*'\|_{L^\infty(\mathbb{T})}\|M_{D}\|_{L^\infty_s(\mathbb{T})L^2_{s'}(\mathbb{T})}\\
&\;+C \|L_D\|_{L^\infty_s(\mathbb{T})L^2_{s'}(\mathbb{T})}\|X_*'\|_{L^\infty(\mathbb{T})}\|M_{X_*}\|_{L^\infty_s(\mathbb{T})L^2_{s'}(\mathbb{T})}
\|L_{X_{*}}+L_{X_{\eta_*}}\|_{L^\infty_s(\mathbb{T})L^\infty_{s'}(\mathbb{T})}\|L_{D}\|_{L^\infty_s(\mathbb{T})L^\infty_{s'}(\mathbb{T})}\\
\leq &\;C \|D'\|_{L^2(\mathbb{T})}\|D'\|_{L^\infty(\mathbb{T})}\|X_{\eta_*}''\|_{L^2(\mathbb{T})}+C \|D'\|_{L^2(\mathbb{T})}\|X_*'\|_{L^\infty(\mathbb{T})}\|D''\|_{L^2(\mathbb{T})}\\
&\;+C \|D'\|_{L^2(\mathbb{T})}\|X_*'\|_{L^\infty(\mathbb{T})}\|X_*''\|_{L^2(\mathbb{T})}
(\|X_{*}'\|_{L^\infty(\mathbb{T})}+\|X_{\eta_*}'\|_{L^\infty(\mathbb{T})})\|D'\|_{L^\infty(\mathbb{T})}\\
\leq &\;C \|D'\|_{L^2(\mathbb{T})}\|D\|_{\dot{H}^{5/2}(\mathbb{T})}\\
\leq &\;C\varepsilon_*' \|D'\|_{L^2(\mathbb{T})}.
\end{split}
\end{equation*}
In a similar manner, we can prove the same bound for the other terms in $\mathcal{R}_{X}$. Thus we proved \eqref{eqn: estimate of the higher order error term in estimating the velocity around equilibrium}.
\end{proof}
\end{lemma}

The following lemma calculates the leading term of $u_X(X(s))$ in \eqref{eqn: first approximation by linearization of velocity around equilibrium}.
\begin{lemma}\label{lemma: final representation of the linearization of velocity near equilibrium}
Assume $X_*(s) = (\cos s, \sin s)^T$. Then
\begin{equation}
\left.\frac{\partial}{\partial\eta}\right|_{\eta = 0}u_{X_\eta}(X_\eta(s))=-\frac{1}{4}\left(\begin{array}{cc}0&1\\-1&0\end{array}\right)\mathcal{H}D(s)-\frac{1}{4}\mathcal{H}D'(s).
\label{eqn: final representation of the linearization of velocity near equilibrium}
\end{equation}
Here $\mathcal{H}$ denotes the Hilbert transform on $\mathbb{T}$ \cite{grafakos2008classical}. 
\end{lemma}
The proof is simply a long calculation. We leave it to the Appendix \ref{appendix section: auxiliary calculations}.

\begin{lemma}\label{lemma: lower bound for the energy dissipation rate in terms of excess energy}
There is a universal $\varepsilon_{*}'>0$ and a universal constant $C>0$, such that
\begin{equation}
\|u_X(X(s)) - \bar{u}_X\|_{L^2(\mathbb{T})}\geq C\|X-X_*\|_{\dot{H}^1(\mathbb{T})},\quad\forall\, X\in  S_{\sqrt{2}\varepsilon_{*}'},
\label{eqn: lower bound for the L2 norm of velocity oscillation}
\end{equation}
where $S_\varepsilon$ is defined in \eqref{eqn: def of data close to equilibrium}.
In particular, this implies that
\begin{equation}
\int_{\mathbb{R}^2} |\nabla u_X|^2 \,dx\geq C\left(\|X\|_{\dot{H}^1(\mathbb{T})}^2-\|X_*\|_{\dot{H}^1(\mathbb{T})}^2\right),\quad\forall\, X\in  S_{\sqrt{2}\varepsilon_{*}'},
\label{eqn: lower bound for the energy dissipation rate in terms of excess energy}
\end{equation}
with some universal constant $C>0$.
\begin{proof}
We always assume that $\varepsilon_*'$ satisfies the assumptions in Lemma \ref{lemma: linearization of velocity field around equilibrium}.
By Lemma \ref{lemma: final representation of the linearization of velocity near equilibrium},
\begin{equation*}
\begin{split}
\left.\frac{\partial}{\partial\eta}\right|_{\eta = 0}u_{X_\eta}(X_\eta(s))=&\;-\frac{1}{4}\left(\begin{array}{cc}0&1\\-1&0\end{array}\right)\mathcal{H}D(s)-\frac{1}{4}\mathcal{H}D'(s)\\
=&\;-\frac{1}{4}\sum_{k\in\mathbb{Z}} \left(\begin{array}{c}-i\cdot\mathrm{sgn}(k)\hat{D}_{k,2}+|k| \hat{D}_{k,1} \\i\cdot\mathrm{sgn}(k) \hat{D}_{k,1}+|k| \hat{D}_{k,2}\end{array}\right)  \mathrm{e}^{iks}.
\end{split}
\end{equation*}
Obviously,
\begin{equation}
\int_{\mathbb{T}}\left.\frac{\partial}{\partial\eta}\right|_{\eta = 0}u_{X_\eta}(X_\eta(s))\,ds = 0.
\label{eqn: linearization of velocity field has mean zero}
\end{equation}
By Parseval's identity and the fact that $\hat{D}_{-k} = \overline{\hat{D}_{k}}$,
\begin{equation}
\begin{split}
\left\|\left.\frac{\partial}{\partial\eta}\right|_{\eta = 0}u_{X_\eta}(X_\eta(s))\right\|_{L^2(\mathbb{T})}^2
=&\;\frac{\pi}{8}\sum_{k\in\mathbb{Z}} \left|\left(\begin{array}{c}-i\cdot\mathrm{sgn}(k)\hat{D}_{k,2}+|k| \hat{D}_{k,1} \\i\cdot\mathrm{sgn}(k) \hat{D}_{k,1}+|k| \hat{D}_{k,2}\end{array}\right)\right|^2\\
\geq &\;\frac{\pi}{8}\left|\left(\begin{array}{c}-i\hat{D}_{1,2}+ \hat{D}_{1,1} \\i \hat{D}_{1,1}+\hat{D}_{1,2}\end{array}\right)\right|^2+\frac{\pi}{8}\left|\left(\begin{array}{c}i\hat{D}_{-1,2}+ \hat{D}_{-1,1} \\-i \hat{D}_{-1,1}+\hat{D}_{-1,2}\end{array}\right)\right|^2\\
&\;+\frac{\pi}{8}\sum_{k\in\mathbb{Z}\atop |k|\geq 2} \left|(|k|-1)|\hat{D}_k|\right|^2\\
\geq  &\;\frac{\pi}{2}\left[(\Re \hat{D}_{1,1}+\Im \hat{D}_{1,2})^2+(\Im \hat{D}_{1,1}-\Re \hat{D}_{1,2})^2\right]+\frac{\pi}{32}\sum_{k\in\mathbb{Z}\atop |k|\geq 2} |k|^2|\hat{D}_k|^2.
\end{split}
\label{eqn: a crude lower bound for the linearized velocity L2 norm with mode 1 -1 unhandled}
\end{equation}
Recall that $D(s)$ satisfies the constraints \eqref{eqn: constraint on deviation from volume conservation} and \eqref{eqn: simplified equation for the optimal approximated equilibrium}, with $Y_*$ replaced by $X_*$.
\eqref{eqn: simplified equation for the optimal approximated equilibrium} imples that $(\Im \hat{D}_{1,1}-\Re \hat{D}_{1,2})^2 = 2(\Im \hat{D}_{1,1})^2+2(\Re \hat{D}_{1,2})^2$;
\eqref{eqn: constraint on deviation from volume conservation} together with \eqref{eqn: inner product of D and Y_star} implies that
\begin{equation*}
|\Re \hat{D}_{1,1}-\Im \hat{D}_{1,2}| \leq C\left|\int_\mathbb{T}D\cdot Y_* \right|\leq C\left|\int_\mathbb{T}D\times D'\,ds\right|\leq C\|D\|_{L^2(\mathbb{T})}\|D'\|_{L^2(\mathbb{T})}\leq C\varepsilon_*'\|D'\|_{L^2(\mathbb{T})}.
\end{equation*}
Here we used the fact that $D$ has mean zero on $\mathbb{T}$, and thus $\|D\|_{L^2(\mathbb{T})}\leq C\|D\|_{\dot{H}^{5/2}(\mathbb{T})}\leq C\varepsilon_*'$.
Hence, we use \eqref{eqn: a crude lower bound for the linearized velocity L2 norm with mode 1 -1 unhandled} to derive that
\begin{equation}
\begin{split}
&\;\left\|\left.\frac{\partial}{\partial\eta}\right|_{\eta = 0}u_{X_\eta}(X_\eta(s))\right\|_{L^2(\mathbb{T})}^2\\
\geq &\;\frac{\pi}{2}\left[(\Re \hat{D}_{1,1}+\Im \hat{D}_{1,2})^2+(\Re \hat{D}_{1,1}-\Im \hat{D}_{1,2})^2+2(\Im \hat{D}_{1,1})^2+2(\Re \hat{D}_{1,2})^2\right]\\
&\;+\frac{\pi}{32}\sum_{k\in\mathbb{Z}\atop |k|\geq 2} |k|^2|\hat{D}_k|^2 - C\varepsilon_*'^2\|D'\|^2_{L^2(\mathbb{T})}\\
\geq &\;\frac{\pi}{32}\sum_{k\in\mathbb{Z}} |k|^2|\hat{D}_k|^2 - C\varepsilon_*'^2\|D'\|^2_{L^2(\mathbb{T})}= \left(\frac{1}{64} - C\varepsilon_*'^2\right)\|D'\|^2_{L^2(\mathbb{T})}.
\end{split}
\label{eqn: a lower bound for the linearized velocity L2 norm with mode 1 -1 unhandled}
\end{equation}
Here $C$ is a universal constant.
To this end, we use Lemma \ref{lemma: linearization of velocity field around equilibrium} and \eqref{eqn: linearization of velocity field has mean zero} to derive that
\begin{equation*}
\begin{split}
\|u_X(X(s)) - \bar{u}_{X}\|_{L^2(\mathbb{T})} \geq &\; \left\|\left.\frac{\partial}{\partial\eta}\right|_{\eta = 0}u_{X_\eta}(X_\eta(s))\right\|_{L^2(\mathbb{T})} - \left\|\mathcal{R}_X(s) - \overline{\mathcal{R}_X}\right\|_{L^2(\mathbb{T})}\\
\geq &\;\left(\frac{1}{64} - C\varepsilon_*'^2\right)^{1/2}\|D'\|_{L^2(\mathbb{T})} - \|\mathcal{R}_X(s)\|_{L^2(\mathbb{T})}\\
\geq &\;\left[\left(\frac{1}{64} - C\varepsilon_*'^2\right)^{1/2}-C\varepsilon_*'\right]\|D'\|_{L^2(\mathbb{T})}.
\end{split}
\end{equation*}
Again, $C$ is a universal constant. Taking $\varepsilon_*'$ sufficiently small, but still universal, we proved the desired lower bound \eqref{eqn: lower bound for the L2 norm of velocity oscillation} with some universal constant $C$.
\eqref{eqn: lower bound for the energy dissipation rate in terms of excess energy} follows immediately from \eqref{eqn: energy dissipation rate can be bounded from below by the L2 norm of velocity oscillation}, \eqref{eqn: lower bound for the L2 norm of velocity oscillation} and Lemma \ref{lemma: estimates concerning closest equilbrium}.
\end{proof}
\end{lemma}

With Lemma \ref{lemma: lower bound for the energy dissipation rate in terms of excess energy}, we conclude by \eqref{eqn: energy estimate on each time slice simplified version} and Lemma \ref{lemma: estimates concerning closest equilbrium} that
\begin{corollary}\label{coro: exponential decay of H1 distance from equilibrium when it is in a small H2.5 neighborhood}
Let $X_0$ satisfy all the assumptions in Theorem \ref{thm: global existence near equilibrium} so that $X$ is the unique global-in-time solution of \eqref{eqn: contour dynamic formulation of the immersed boundary problem} starting from $X_0$. Then there exist universal constants $\varepsilon_{*}',\alpha>0$, such that if in addition $X(\cdot, t)\in  S_{\sqrt{2}\varepsilon_{*}'}$,
\begin{align*}
\|X\|^2_{\dot{H}^1(\mathbb{T})}(t)-\|X_*\|^2_{\dot{H}^1(\mathbb{T})}(t)\leq &\;e^{-2\alpha t}\left(\|X_0\|^2_{\dot{H}^1(\mathbb{T})}-\|X_{0*}\|^2_{\dot{H}^1(\mathbb{T})}\right),\\
\|X-X_*\|_{\dot{H}^1(\mathbb{T})}(t)\leq &\;2\sqrt{2}e^{-\alpha t}\|X_0-X_{0*}\|_{\dot{H}^1(\mathbb{T})},
\end{align*}
where $S_\varepsilon$ is defined in \eqref{eqn: def of data close to equilibrium}.
\end{corollary}

\subsection{Proof of exponential convergence to equilibrium configurations}\label{section: proof of exponential convergence to equilibrium configurations}
Before we prove Theorem \ref{thm: exponential convergence}, we first state the following simple lemma.
\begin{lemma}\label{lemma: property of epsilon_xi}
Let $\varepsilon_\xi$ be defined as in \eqref{eqn: defintion of varepsilon_* in the corollary}, i.e.\;$\varepsilon_\xi = 2C_4(C_6,1/(2\pi))(C_7+2) (|\ln (2\xi)|+1)^2(2\xi)$, where $C_4$, $C_6$ and $C_7$ are universal constant. Then $\varepsilon_\xi$ is increasing on $\xi \in(0,1/(2e)]$. Moreover, for $\forall\,\xi \in(0,1/(2e)]$ and $\forall \,c\geq e$,
\begin{equation*}
\frac{1}{c}\varepsilon_{\xi}\leq \varepsilon_{(\xi/c)}\leq \frac{1}{c}\left(\frac{2+\ln c}{2}\right)^2\varepsilon_{\xi}\triangleq \beta_c\varepsilon_\xi.
\end{equation*}
The first inequality is  true even for $\forall\, c\geq 1$. $\beta_c$ is decreasing in $c\geq e$ and $\beta_c\leq \frac{9}{4e}<1$ for $\forall\, c\geq e$.
\end{lemma}
Its proof is a simple calculation, which we shall omit.
Now we are able to prove Theorem \ref{thm: exponential convergence}.

\begin{proof}[Proof of Theorem \ref{thm: exponential convergence}]
As before, we assume $R_{X_0} = 1$.
For convenience, we denote
\begin{equation*}
\mathcal{F}(t) = \|X-X_*\|_{\dot{H}^1(\mathbb{T})}(t),\quad \mathcal{G}(t) = \|X-X_*\|_{\dot{H}^{5/2}(\mathbb{T})}(t).
\end{equation*}
Note that the case $\mathcal{F}(0) = 0$ is trivial; we shall only discuss the case $\mathcal{F}(0)>0$ in the sequel.

Let $\varepsilon_*'$ be defined as in Corollary \ref{coro: exponential decay of H1 distance from equilibrium when it is in a small H2.5 neighborhood}. We may assume $\varepsilon_*'\leq 1$. Take $\xi_{**}\leq 1/(16e)$ such that
\begin{equation}
\varepsilon_{8\xi_{**}} = 2C_4(C_6,1/(2\pi))(C_7+2) (|\ln (16\xi_{**})|+1)^2(16\xi_{**}) = \sqrt{2}\varepsilon_*',
\label{eqn: definition of xi_**}
\end{equation}
where $C_4$, $C_6$ and $C_7$ are universal constants defined in \eqref{eqn: introducing C_4}, \eqref{eqn: uniform bound of the family of solution} and \eqref{eqn: bound for H1 norm of the difference to the equilibrium} respectively.
Such $\xi_{**}$ is uniquely achievable as a universal constant by the assumptions $C_4\geq 1$ and $\varepsilon_*'\leq 1$, since it is required that
\begin{equation*}
(|\ln (16\xi_{**})|+1)^2(16\xi_{**}) =\frac{\sqrt{2}\varepsilon_*'}{2C_4(C_7+2) }\leq\frac{\sqrt{2}}{4},
\end{equation*}
while $x(|\ln x|+1)^2$ monotonically maps $(0,1/e]$ onto $(0,4/e]$.
Hence, by Corollary \ref{coro: refined decay estimate of global solution}, the solution $X$ satisfies that for $\forall\, t\geq 0$,
\begin{equation}
\mathcal{G}(t)\leq \max\{2\mathcal{G}(0)e^{-t/8},\varepsilon_{\mathcal{F}(0)}\}\leq \max \{2\mathcal{G}(0)e^{-t/8},\sqrt{2}\varepsilon_*'\}.
\label{eqn: decay of G in the first time interval crude form}
\end{equation}
Here we used Lemma \ref{lemma: property of epsilon_xi} to find that $\varepsilon_{\mathcal{F}(0)} \leq \varepsilon_{8\xi_{**}}= \sqrt{2}\varepsilon_*'$.
If $2\mathcal{G}(0)\leq \varepsilon_{\mathcal{F}(0)}$, we take $t_* = 0$; otherwise, take $t_*$ such that
\begin{equation}
2\mathcal{G}(0)e^{-t_*/8}=\varepsilon_{\mathcal{F}(0)}.
\label{eqn: first constraints on t_* in exp convergence}
\end{equation}
Hence, we have $X(t)\in S_{\varepsilon_{\mathcal{F}(0)}}\subset S_{\sqrt{2}\varepsilon_*'}$ if $t\geq t_*$, which allows us to apply Corollary \ref{coro: exponential decay of H1 distance from equilibrium when it is in a small H2.5 neighborhood} for $t\geq t_*$.
By Lemma \ref{lemma: energy estimate} and Lemma \ref{lemma: estimates concerning closest equilbrium}, we derive that
\begin{equation*}
\mathcal{F}(t_*)\leq 2 \left(\|X\|^2_{\dot{H}^1(\mathbb{T})} - \|X_*\|^2_{\dot{H}^1(\mathbb{T})}\right)^{1/2}(t_*)\leq 2\left(\|X_0\|^2_{\dot{H}^1(\mathbb{T})} - \|X_{0*}\|^2_{\dot{H}^1(\mathbb{T})}\right)^{1/2}\leq 2\sqrt{2}\mathcal{F}(0).
\end{equation*}
By Corollary \ref{coro: exponential decay of H1 distance from equilibrium when it is in a small H2.5 neighborhood}, for some universal $\alpha>0$, and $\forall\, t>0$,
\begin{equation*}
\mathcal{F}(t_*+t)\leq 2\sqrt{2}\mathcal{F}(t_*)e^{-\alpha t}\leq 8\mathcal{F}(0)e^{-\alpha t}.
\end{equation*}
Note that $8\mathcal{F}(0)\in (0,1/(2e)]$ by the assumption on $\xi_{**}$, in which interval $\varepsilon_\xi$ is increasing in $\xi$. 
Without loss of generality, we may assume $\alpha <\frac{1}{8}$.
We additionally take $t_{**}>0$ to be a universal constant such that
\begin{equation}
e^{-\alpha t_{**}}\leq \frac{1}{8},\quad e^{\left(\frac{1}{8}-\alpha\right)t_{**}}\geq 2.
\label{eqn: constraint on t_**}
\end{equation}

To this end, we shall use mathematical induction to show \eqref{eqn: exp convergence in H2.5 norm}.
Let us summarize what has been proved so far: 
\begin{enumerate}
\item For $\forall\, t\geq 0$, $\mathcal{G}(t)\leq \max \{2e^{-t/8}\mathcal{G}(0),\varepsilon_{\mathcal{F}(0)}\}$.
\item Wit the choice of $t_*$ in \eqref{eqn: first constraints on t_* in exp convergence}, for $\forall\, t\in[t_*, t_*+2t_{**}]$, $\mathcal{G}(t)\leq \varepsilon_{\mathcal{F}(0)}\leq \sqrt{2}\varepsilon_*'$.
\item For $\forall\, t>0$, $\mathcal{F}(t_*+t)\leq 8\mathcal{F}(0)e^{-\alpha t}$. In particular, with the choice of $t_{**}$ in \eqref{eqn: constraint on t_**}, for $\forall\,k\in\mathbb{Z}_+$, $k\geq 2$, $\mathcal{F}(t_*+kt_{**})\leq e^{-\alpha (k-1)t_{**}}\mathcal{F}(0)$.
\end{enumerate}

Suppose that
\begin{equation}
\mathcal{G}(t_*+kt_{**})\leq \varepsilon_{e^{-(k-2)\alpha t_{**}}\mathcal{F}(0)}
\label{eqn: decay of G in the exp convergence}
\end{equation}
has been proved for some $k\geq 2$, $k\in\mathbb{Z}_+$ (indeed, the case $k =2$ has been established above).
By Corollary \ref{coro: refined decay estimate of global solution}, $\forall\, t\in[0,t_{**}]$,
\begin{equation}
\begin{split}
\mathcal{G}(t_*+kt_{**}+t)\leq &\;\max\{2e^{-t/8}\mathcal{G}(t_*+kt_{**}),\varepsilon_{\mathcal{F}(t_*+kt_{**})}\}\\
\leq &\;\max\{2e^{-t/8}\varepsilon_{e^{-(k-2)\alpha t_{**}}\mathcal{F}(0)},\varepsilon_{e^{-(k-1)\alpha t_{**}}\mathcal{F}(0)}\}.
\end{split}
\label{eqn: estimates of G inside the time interval}
\end{equation}
In particular,
\begin{equation*}
\mathcal{G}(t_*+(k+1)t_{**})\leq \max\{2e^{-t_{**}/8}\varepsilon_{e^{-(k-2)\alpha t_{**}}\mathcal{F}(0)},\varepsilon_{e^{-(k-1)\alpha t_{**}}\mathcal{F}(0)}\}.
\end{equation*}
We claim that with the choice of $t_{**}$ in \eqref{eqn: constraint on t_**}, the first term on the right hand side is always smaller.
Indeed, by \eqref{eqn: constraint on t_**}, and the lower bound in Lemma \ref{lemma: property of epsilon_xi},
\begin{equation*}
\frac{2e^{-t_{**}/8}\varepsilon_{e^{-(k-2)\alpha t_{**}}\mathcal{F}(0)}}{\varepsilon_{e^{-(k-1)\alpha t_{**}}\mathcal{F}(0)}} \leq \frac{2e^{-t_{**}/8}\varepsilon_{e^{-(k-2)\alpha t_{**}}\mathcal{F}(0)}}{e^{-\alpha t_{**}}\varepsilon_{e^{-(k-2)\alpha t_{**}}\mathcal{F}(0)}} = 2e^{-\left(\frac{1}{8}-\alpha\right)t_{**}}\leq 1.
\end{equation*}
Hence, we proved that
\begin{equation*}
\mathcal{G}(t_*+(k+1)t_{**})\leq \varepsilon_{e^{-(k-1)\alpha t_{**}}\mathcal{F}(0)}.
\end{equation*}
Therefore, by induction, \eqref{eqn: decay of G in the exp convergence} is true for all $k\in\mathbb{Z}_+$, $k\geq 2$; so is \eqref{eqn: estimates of G inside the time interval}.
With the choice of $t_{**}$, we use \eqref{eqn: estimates of G inside the time interval} and the upper bound in Lemma \ref{lemma: property of epsilon_xi} to derive that for $\forall\, t\in[0,t_{**}]$,
\begin{equation*}
\begin{split}
\mathcal{G}(t_*+kt_{**}+t)\leq &\;\max\{2e^{-t/8}\varepsilon_{e^{-(k-2)\alpha t_{**}}\mathcal{F}(0)},\varepsilon_{e^{-(k-1)\alpha t_{**}}\mathcal{F}(0)}\}\\
\leq &\;\max\{2e^{-t/8}\beta_{e^{\alpha t_{**}}}^{k-2}\varepsilon_{\mathcal{F}(0)},\beta_{e^{\alpha t_{**}}}^{k-1}\varepsilon_{\mathcal{F}(0)}\}\\
\leq &\;\beta_{8}^{k-2}\varepsilon_{\mathcal{F}(0)} \max\{2e^{-t/8},\beta_{8}\} \leq 2\beta_{8}^{k-2}\varepsilon_{\mathcal{F}(0)}.
\end{split}
\end{equation*}
Note that $t_{**}$ and $\beta_8<1$ are both universal constants.
Hence, combining this with the fact that $\mathcal{G}(t_*+t) \leq \varepsilon_{\mathcal{F}(0)}$ for $\forall\, t\geq 0$, we find that there exist universal constants $\alpha_*\leq 1/8$ and $C>1$, such that for $\forall\, t\geq 0$,
\begin{equation}
\mathcal{G}(t_*+t)\leq C e^{-\alpha_* t}\varepsilon_{\mathcal{F}(0)}.
\label{eqn: exp decay when t is larger than t_*}
\end{equation}

If $t_* = 0$, we readily proved that for $\forall\, t\geq 0$,
\begin{equation}
\mathcal{G}(t)\leq C e^{-\alpha_* t}\varepsilon_{\mathcal{F}(0)},
\label{eqn: exp decay when G is small}
\end{equation}
where $C$ and $\alpha_*$ are universal constants.
If $t_*>0$, by \eqref{eqn: first constraints on t_* in exp convergence} and the fact that $\alpha_*\leq 1/8$, $\varepsilon_{\mathcal{F}(0)} = 2e^{-t_*/8}\mathcal{G}(0)\leq 2e^{-\alpha_{*}t_*}\mathcal{G}(0)$.
Hence, by \eqref{eqn: exp decay when t is larger than t_*}, for $\forall\, t\geq 0$,
\begin{equation}
\mathcal{G}(t_*+t)\leq C e^{-\alpha_* (t_*+t)}\mathcal{G}(0).
\label{eqn: decay of G in the latter time interval}
\end{equation}
On the other hand, since $\alpha_*\leq 1/8$, we also know that for $t\in[0,t_*]$,
\begin{equation}
\mathcal{G}(t)\leq 2e^{-t/8}\mathcal{G}(0)\leq 2e^{-\alpha_{*}t}\mathcal{G}(0).
\label{eqn: decay of G in the first time interval}
\end{equation}
Combining \eqref{eqn: decay of G in the latter time interval} and \eqref{eqn: decay of G in the first time interval}, we proved that
\begin{equation}
\mathcal{G}(t)\leq Ce^{-\alpha_{*}t}\mathcal{G}(0).
\label{eqn: exp decay when G is large}
\end{equation}
with some universal constants $\alpha_*$ and $C$.
Combining \eqref{eqn: exp decay when G is small} and \eqref{eqn: exp decay when G is large}, we complete the proof of \eqref{eqn: exp convergence in H2.5 norm}.

In order to prove \eqref{eqn: exp convergence to a fixed configuration}, we use the fact $u_{X_*}(x) \equiv 0$ to derive that
\begin{equation*}
\|u_X(X(s))\|_{H^1(\mathbb{T})} = \|u_X(X(s))- u_{X_*}(X_*(s))\|_{H^1(\mathbb{T})} \leq \|\mathcal{L}X-\mathcal{L}X_*\|_{H^1(\mathbb{T})}+\|g_X-g_{X_*}\|_{H^1(\mathbb{T})}.
\end{equation*}
By Corollary \ref{coro: L2 estimate for g_X1-g_X2} and Corollary \ref{coro: H1 estimate for g_X1-g_X2},
\begin{equation*}
\|X_t(s,t)\|_{H^1(\mathbb{T})}  = \|u_X(X(s),t)\|_{H^1(\mathbb{T})} \leq C \|X-X_*\|_{\dot{H}^2(\mathbb{T})}(t) \leq C \|X-X_*\|_{\dot{H}^{5/2}(\mathbb{T})}(t).
\end{equation*}
Here $C$ is a universal constant thanks to the uniform estimates of solutions obtained in Theorem \ref{thm: global existence near equilibrium}.
Hence, by \eqref{eqn: exp convergence in H2.5 norm},
\begin{equation}
\|X(s,t)-X(s,t')\|_{H^1(\mathbb{T})} \leq \int_{t}^{t'} \|X_t(s,\tau)\|_{H^1(\mathbb{T})}\,d\tau \leq CB(X_0)\int_{t}^{t'} e^{-\alpha_*\tau}\,d\tau.
\label{eqn: X(t) is a Cauchy sequence in H^1 given the exp decay}
\end{equation}
Here $B(X_0)$ is defined in \eqref{eqn: exp convergence in H2.5 norm} and $C$ is a universal constant.
This implies that $X(s,t)$ is a Cauchy sequence in $H^1(\mathbb{T})$, which converges to some $X_\infty(s)\in H^1(\mathbb{T})$.
Take $t'\rightarrow +\infty$ in \eqref{eqn: X(t) is a Cauchy sequence in H^1 given the exp decay} and we find
\begin{equation}
\|X(s,t)-X_\infty(s)\|_{H^1(\mathbb{T})} \leq CB(X_0)e^{-\alpha_*t}.
\label{eqn: exp convergence to a fixed configuration in H^1 norm}
\end{equation}
On the other hand, by virtue of the bound \eqref{eqn: uniform bound of H 2.5 norm for the global solution} of $\|X\|_{\dot{H}^{5/2}(\mathbb{T})}(t)$, we may take $\tilde{X}_{w,\infty}\in H^{5/2}(\mathbb{T})$ as an arbitrary weak limit (up to a subsequence) of $\{\tilde{X}(t)\}_{t\geq 0}$ in $H^{5/2}(\mathbb{T})$.
Note that we only have the bound on the $H^{5/2}$-seminorm of $X(t)$, so we can only extract weak limit for $\{\tilde{X}(t)\}_{t\geq 0}$ instead of $\{X(t)\}_{t\geq 0}$ at this moment.
By compact Sobolev embedding, $\tilde{X}_{w,\infty}$ is a strong $H^1(\mathbb{T})$-limit of a subsequence of $\{\tilde{X}(t)\}_{t\geq 0}$.
Since $\tilde{X}(t)\rightarrow \tilde{X}_{\infty}$ in $H^1(\mathbb{T})$, one must have $\tilde{X}_{w,\infty} = \tilde{X}_\infty$.
And this is true for all weak limits of $\{\tilde{X}(t)\}_{t\geq 0}$.
Hence, $X_{\infty} \in H^{5/2}(\mathbb{T})$ and satisfies $\|X_\infty\|_{\dot{H}^{5/2}(\mathbb{T})}\leq C$, with the same universal constant $C$ as in \eqref{eqn: uniform bound of H 2.5 norm for the global solution}.
By \eqref{eqn: exp convergence in H2.5 norm} and the convergence $X(t)\rightarrow X_\infty$ in $H^1(\mathbb{T})$, we know that $\|X_\infty - X_{\infty,*}\|_{\dot{H}^1(\mathbb{T})} = 0$. Hence $X_\infty = X_{\infty,*}$ is an equilibrium configuration.

To this end, we derive  \eqref{eqn: exp convergence to a fixed configuration} as follows
\begin{equation*}
\begin{split}
\|X(t)-X_\infty\|_{H^{5/2}} \leq &\; C\|X(t)-X_\infty\|_{H^1}+C\|X(t)-X_\infty\|_{\dot{H}^{5/2}}\\
\leq &\; C\|X(t)-X_\infty\|_{H^1}+C\|X(t)-X_*(t)\|_{\dot{H}^{5/2}}+C\|X_\infty(t)-X_*(t)\|_{\dot{H}^{5/2}}.
\end{split}
\end{equation*}
Note that both $X_*$ and $X_\infty$ are equilibrium configurations. Since $X_\infty(s,t)-X_*(s,t)$ as a function of $s\in\mathbb{T}$ only contains Fourier modes with wave numbers $0$ and $\pm 1$, we can replace the $H^{5/2}$-seminorm in the last term by $H^1$-seminorm without changing its value, i.e.
\begin{equation*}
\begin{split}
\|X(t)-X_\infty\|_{H^{5/2}} \leq &\; C\|X(t)-X_\infty\|_{H^1}+C\|X(t)-X_*(t)\|_{\dot{H}^{5/2}}+C\|X_\infty(t)-X_*(t)\|_{\dot{H}^{1}}\\
\leq &\; C\|X(t)-X_\infty\|_{H^1}+C\|X(t)-X_*(t)\|_{\dot{H}^{5/2}}+C\|X(t)-X_\infty(t)\|_{\dot{H}^{1}}\\
&\;+C\|X(t)-X_*(t)\|_{\dot{H}^{1}}\\
\leq &\; C\|X(t)-X_\infty\|_{H^1}+C\|X(t)-X_*(t)\|_{\dot{H}^{5/2}}\\
\leq &\; CB(X_0)e^{-\alpha_*t}.
\end{split}
\end{equation*}
In the last inequality, we used \eqref{eqn: exp convergence in H2.5 norm} and \eqref{eqn: exp convergence to a fixed configuration in H^1 norm}.
This completes the proof of \eqref{eqn: exp convergence to a fixed configuration}.
\end{proof}

\section{Conclusion and Discussion}
In this paper, we transform the Stokes immersed boundary problem \eqref{eqn: stokes equation}-\eqref{eqn: kinematic equation of membrane} in two dimensions into a contour dynamic formulation \eqref{eqn: contour dynamic formulation of the immersed boundary problem} via the fundamental solution of the Stokes equation.
We proved that there exists a unique local solution of the contour dynamic formulation (Theorem \ref{thm: local in time existence} and Theorem \ref{thm: local in time uniqueness}), provided that the initial data is an $H^{5/2}$-function in Lagrangian coordinate and satisfies the well-stretched condition \eqref{eqn: well_stretched assumption}.
If in addition the initial configuration is sufficiently close to an equilibrium, the solution should exist globally in time (Theorem \ref{thm: global existence near equilibrium}), and it converges exponentially to an equilibrium as $t\rightarrow +\infty$ (Theorem \ref{thm: exponential convergence}).
Regularity of the ambient flow field can thus be recovered through the fundamental solution of the Stokes equation (Lemma \ref{lemma: the velocity field is continuous} and Lemma \ref{lemma: energy estimate}).

In the contour dynamic formulation \eqref{eqn: contour dynamic formulation of the immersed boundary problem}, the string motion is given by a singular integral, which depends nonlinearly functional on the string configuration.
The starting point of the proofs in this paper is that the principal part of the singular integral in the contour dynamic formulation introduces dissipation, which essentially results from the dissipation in the Stokes flow.
Then it suffices to show that the remainder term is regular in some sense and can be well-controlled by the dissipation.
The same approach may also apply to the higher dimensional case, where a 2-D closed membrane is immersed and moving in a 3-D Stokes flow, although the description of the 2-D membrane needs some extra efforts.
Note that the equilibrium shape of the membrane may not necessarily be a sphere.
We shall address this problem in a forthcoming work.

In this paper, we only consider the simplest case where the 1-D string is modeled by a Hookean material with zero resting length in the force-free state.
See the local elastic energy density \eqref{eqn: elastic energy density}.
In particular, the material always tends to shorten its length in all time.
Other types of elastic constitutive law can be also considered.
In fact, most of the discussion in this paper may also apply to more general elastic energy of other forms.
We do not dig deep into this topic here, but it would be interesting to find out what conditions are needed for the energy density so that the current approach still works.

\appendix
\section{Appendix}
\subsection{Study of the Flow Field}\label{appendix section: study of the flow field}
In this section, we shall prove Lemma \ref{lemma: the velocity field is continuous} and Lemma \ref{lemma: energy estimate} that characterize properties of the flow field $u_X$.
Roughly speaking, Lemma \ref{lemma: the velocity field is continuous} claims that under certain assumptions on $X$, $u_X\in C(\mathbb{R}^2)$ and $\nabla u_X \in L^2(\mathbb{R}^2)$; while Lemma \ref{lemma: energy estimate} proves an energy estimate of the whole system, which says that under certain assumptions on $X$, the decrease in the elastic energy of the string is fully accounted by the energy dissipation in the Stokes flow.
See their precise statements in Section \ref{section: energy estimate}.
\begin{proof}[Proof of Lemma \ref{lemma: the velocity field is continuous}]
If $x\not\in\Gamma_t$, since $G(x-X(s',t))$ in \eqref{eqn: expression for velocity field} is smooth at $x$, $u_X$ is also smooth at $x$.
We then turn to show continuity of $u_X$ at $X(s,t)\in \Gamma_t$.

Take any arbitrary $x\in\mathbb{R}^2$, and let $s_x$ be defined as in \eqref{eqn: definition of s_x}. Note that $s_x$ may not be unique; take an arbitrary one if it is the case.
We first show that
\begin{equation}
|x-X(s')| \geq \frac{\lambda}{2}|s'-s_x|.
\label{eqn: lower bound for distance between x and X(s')}
\end{equation}
Indeed, if $|s'-s_x| \leq 2\lambda^{-1}\mathrm{dist}(x,X(\mathbb{T}))$, then by definition of $s_x$,
\begin{equation*}
|x-X(s')| \geq |x-X(s_x)| = \mathrm{dist}(x,X(\mathbb{T})) \geq \frac{\lambda}{2}|s'-s_x|.
\end{equation*}
If $|s'-s_x| \geq 2\lambda^{-1}\mathrm{dist}(x,X(\mathbb{T}))$, by triangle inequality,
\begin{equation*}
|x-X(s')| \geq |X(s')-X(s_x)| - |x-X(s_x)| \geq \lambda|s'-s_x| - \mathrm{dist}(x,X(\mathbb{T}))\geq \frac{\lambda}{2}|s'-s_x|.
\end{equation*}
This proves \eqref{eqn: lower bound for distance between x and X(s')}.
If $s_x$ and $s_x'$ both satisfy \eqref{eqn: definition of s_x}, \eqref{eqn: lower bound for distance between x and X(s')} implies that $|s_x-s_x'|\leq 2\lambda^{-1}\mathrm{dist}(x,X(\mathbb{T}))$.

Next, we shall take the limit $x\rightarrow X(s,t)$ and apply dominated convergence theorem to \eqref{eqn: 2D velocity field}. Here we do not assume $x\not\in \Gamma_t$.
For $s'\not = s$, using the formula in \eqref{eqn: 2D velocity field} and \eqref{eqn: velocity of membrane}, it is easy to find that
\begin{equation*}
\partial_{s'} [G(x-X(s'))](X'(s')-X'(s_x)) \rightarrow \partial_{s'} [G(X(s)-X(s'))](X'(s')-X'(s)).
\end{equation*}
Here we used the fact that $X'(s_x)\rightarrow X'(s)$ as $x\rightarrow X(s,t)$. This is because \eqref{eqn: lower bound for distance between x and X(s')} implies that $s_x \rightarrow s$ and $X\in H^2(\mathbb{T})$ implies that $X'\in C^{1/2}(\mathbb{T})$.

On the other hand, by \eqref{eqn: lower bound for distance between x and X(s')},
\begin{equation*}
\begin{split}
|\partial_{s'} [G(x-X(s'))](X'(s')-X'(s_x))|\leq &\; C\frac{|X'(s')||X'(s')-X'(s_x)|}{|X(s')-x|}\\
\leq &\;\frac{C|X'(s')|}{\lambda |s_x-s'|}\int_{s'}^{s_x} |X''(\tau)| \,d\tau\\
\leq &\;C\lambda^{-1}|X'(s')||\mathcal{M}X''(s')|.
\end{split}
\end{equation*}
Here $\mathcal{M}$ is the centered Hardy-Littlewood maximal operator on $\mathbb{T}$ \cite{grafakos2008classical}.
Note that the bound is independent of $x$.
Since $\mathcal{M}$ is bounded from $L^2(\mathbb{T})$ to $L^2(\mathbb{T})$, $C\lambda^{-1}|X'(s')||\mathcal{M}X''(s')|\in L^1(\mathbb{T})$.
Therefore, by dominated convergence theorem, $u_X(x)\rightarrow u_X(X(s))$ as $x\rightarrow X(s)$.
This completes the proof of the continuity of $u_X$.

To show $\nabla u_X\in L^2(\mathbb{R}^2)$, we first introduce a mollifier $\varphi(x)\in C^\infty_0(\mathbb{R}^2)$, such that $\varphi\geq 0$, $\mathrm{supp}\, \varphi \subset B(0,1)$ and $\int_{\mathbb{R}^2}\varphi(x)\,dx = 1$. Define $\varphi_\varepsilon(x) = \varepsilon^{-2}\varphi(x/\varepsilon)$.
Let $f_\varepsilon = \varphi_\varepsilon * f$ and let $(u_\varepsilon, p_\varepsilon)$ solves the Stokes equation with $\varepsilon <1$,
\begin{equation}
\begin{split}
&\;-\Delta u_\varepsilon  + \nabla p_\varepsilon = f_\varepsilon,\\
&\;\mathrm{div} u_\varepsilon = 0,\\
&\;|u_\varepsilon|,|p_\varepsilon|\rightarrow 0\mbox{ as }|x|\rightarrow \infty.
\end{split}
\label{eqn: regularized stokes equation}
\end{equation}
It is obvious that $u_\varepsilon = \varphi_\varepsilon*u_X$ and $p_\varepsilon = \varphi_\varepsilon *p_X$.
On the other hand, since $f\in \mathscr{M}(\mathbb{R}^2)$ under the assumption of the lemma, $f_\varepsilon$ is smooth and so are $u_\varepsilon$, and $p_\varepsilon$.

We also introduce a cut-off function $\phi\in C^\infty_0(\mathbb{R}^2)$, such that $\phi\geq 0$; $\phi(y)= 1$ for $|y|\leq 1$; $\phi(y)= 0$ for $|y|\geq 2$; and $|\nabla\phi(y)|\leq C$ for some universal constant $C$. Define $\phi_r(x) \triangleq \phi(x/r)$.
For given $t$, assume $\Gamma_t\subset B_R(0)=\{x\in\mathbb{R}^2:\;|x|< R\}$ with some $R>2$.
We multiply the regularized Stokes equation \eqref{eqn: regularized stokes equation} on both sides by $\phi_r u_\varepsilon$, with $r= 2R+1$, and take integral on $\mathbb{R}^2$.
By integration by parts, we obtain that
\begin{equation}
\int_{\mathbb{R}^2}\phi_r(x)|\nabla u_\varepsilon(x,t)|^2\,dx + \int_{\mathbb{R}^2}(u_{\varepsilon,i}\partial_j \phi_r \partial_j u_{\varepsilon,i}-u_{\varepsilon,i} \partial_i \phi_r p_\varepsilon)\,dx = \int_{\mathbb{R}^2}\phi_r (x) u_\varepsilon(x,t)f_\varepsilon(x,t)\,dx.
\label{eqn: take inner product to get the energy estimate for the regularized equation}
\end{equation}
Note that $\nabla \phi$ is supported on $B_{2r}(0)\backslash B_r(0)$, which is away from $X(\cdot,t)$.
In the region $B_{6R}(0)\backslash B_{2R}(0)$, which contains an $\varepsilon$-neighborhood of $B_{2r}(0)\backslash B_r(0)$ since $R>2>2\varepsilon$, $u_X$, $\nabla u_X$ and $p_X$ have the following $L^\infty$-bound due to \eqref{eqn: 2D velocity field} and \eqref{eqn: 2D pressure field} with $C_x = 0$.
\begin{align*}
|u_X(x)|\leq &\;CR^{-1}\|X\|_{\dot{H}^1(\mathbb{T})}^2,\\
|\nabla u_X(x)|\leq &\;CR^{-2}\|X\|_{\dot{H}^1(\mathbb{T})}^2,\\
|p_X(x)|\leq &\;CR^{-2}\|X\|_{\dot{H}^1(\mathbb{T})}^2.
\end{align*}
Therefore, the regularized solutions also enjoy similar bounds in $B_{2r}(0)\backslash B_r(0)$, namely
\begin{align}
|u_\varepsilon(x)|\leq &\;Cr^{-1}\|X\|_{\dot{H}^1(\mathbb{T})}^2,\label{eqn: far field estimates of the regularized u}\\
|\nabla u_\varepsilon(x)|\leq &\;Cr^{-2}\|X\|_{\dot{H}^1(\mathbb{T})}^2,\label{eqn: far field estimates of the regularized grad u}\\
|p_\varepsilon(x)|\leq &\;Cr^{-2}\|X\|_{\dot{H}^1(\mathbb{T})}^2.\label{eqn: far field estimates of the regularized p}
\end{align}
Note that the constants $C$'s are uniform in $\varepsilon$.
Applying these estimates in \eqref{eqn: take inner product to get the energy estimate for the regularized equation}, we find that
\begin{equation*}
\int_{\mathbb{R}^2}\phi_r(x)|\nabla u_\varepsilon(x,t)|^2\,dx \leq \int_{\mathbb{R}^2}\phi_r (x) u_\varepsilon(x,t)f_\varepsilon(x,t)\,dx + C r^{-2}\|X\|_{\dot{H}^1(\mathbb{T})}^4.
\end{equation*}
It is known that
\begin{equation}
u_\varepsilon \rightarrow u_X \mbox{ in }C_{loc}(\mathbb{R}^2),\quad f_\varepsilon \rightarrow f \mbox{ in }\mathscr{M}(\mathbb{R}^2).
\label{eqn: convergence of the regularized solution to the original solution}
\end{equation}
This gives
\begin{equation*}
\left|\int_{\mathbb{R}^2}\phi_r (x) u_\varepsilon(x,t)f_\varepsilon(x,t)\,dx\right| \rightarrow \left|\int_{\mathbb{R}^2}\phi_r (x) u_X(x,t)f(x,t)\,dx\right|\mbox{ as }\varepsilon \rightarrow 0^+.
\end{equation*}
Therefore, by \eqref{eqn: a trivial L^infty bound for velocity},
\begin{equation}
\begin{split}
\limsup_{\varepsilon\rightarrow 0^+}\int_{\mathbb{R}^2}\phi_r(x)|\nabla u_\varepsilon(x,t)|^2\,dx \leq &\;\left|\int_{\mathbb{R}^2}\phi_r (x) u_X(x,t)f(x,t)\,dx\right| + C r^{-2}\|X\|_{\dot{H}^1(\mathbb{T})}^4\\
\leq &\;\int_{\mathbb{T}}|u(X(s,t),t)X_{ss}(s,t)|\,ds+ C r^{-2}\|X\|_{\dot{H}^1(\mathbb{T})}^4\\
\leq &\;C\lambda^{-1}\|X\|_{\dot{H}^2(\mathbb{T})}^3+ C r^{-2}\|X\|_{\dot{H}^1(\mathbb{T})}^4.
\label{eqn: a local bound for the dissipation rate of the regularized solution}
\end{split}
\end{equation}
Here we used the fact that $\Gamma_t\subset B_r(0)$.
This implies that for fixed $r$ and any sequence $\varepsilon_k\rightarrow 0^+$, there exists a subsequence $\varepsilon_{k_l}\rightarrow 0^+$ and $\tilde{u}\in H^1(B_r(0))$ such that $\nabla u_{\varepsilon_{k_l}}\rightharpoonup \nabla \tilde{u}$ in $L^2(B_r(0))$. Combining this with the fact $u_\varepsilon \rightarrow u_X$ in $C(B_r(0))$, we can even obtain $u_{\varepsilon_{k_l}}\rightarrow \tilde{u}$ in $L^2(B_r(0))$ and thus $\tilde{u} = u_X\in H^1(B_r(0))$.
Since the sequence $\{\varepsilon_k\}$ is arbitrary and $r$ (or equivalently $R$) can be arbitrarily large, we conclude that $u_X\in H^1_{loc}(\mathbb{R}^2)$ and $\nabla u_\varepsilon \rightarrow \nabla u_X$ in $L^2_{loc}(\mathbb{R}^2)$. Here we obtain local strong convergence as a property of the mollification.
This together with \eqref{eqn: a local bound for the dissipation rate of the regularized solution} implies that
\begin{equation*}
\int_{\mathbb{R}^2}\phi_r(x)|\nabla u_X(x,t)|^2\,dx \leq C\lambda^{-1}\|X\|_{\dot{H}^2(\mathbb{T})}^3+ C r^{-2}\|X\|_{\dot{H}^1(\mathbb{T})}^4.
\end{equation*}
Take $r\rightarrow \infty$ and we find
\begin{equation}
\int_{\mathbb{R}^2}|\nabla u_X(x,t)|^2\,dx \leq C\lambda^{-1}\|X\|_{\dot{H}^2(\mathbb{T})}^3.
\label{eqn: a trivial bound for the energy dissipation rate or H1 semi norm of velocity field}
\end{equation}
This completes the proof.
\end{proof}

As is mentioned before, Lemma \ref{lemma: energy estimate} is devoted to an energy estimate of the system, which will be used in the proof of Theorem \ref{thm: global existence near equilibrium} and Theorem \ref{thm: exponential convergence}.
\begin{proof}[Proof of Lemma \ref{lemma: energy estimate}]
Since $X\in C_TH^2(\mathbb{T})$, $\Gamma_t$ stays in a bounded set in $t\in[0,T]$. We may assume $\Gamma_t\subset B_R(0)=\{x\in\mathbb{R}^2:\;|x|< R\}$ for $t\in [0,T]$ with some $R>2$.
Again let $r = 2R+1$.

We start from the local energy estimate of the regularized solution in the proof of Lemma \ref{lemma: the velocity field is continuous}. By \eqref{eqn: take inner product to get the energy estimate for the regularized equation} and the decay estimates \eqref{eqn: far field estimates of the regularized u}-\eqref{eqn: far field estimates of the regularized p}, we find that
\begin{equation*}
\limsup_{\varepsilon\rightarrow 0^+}\left|\int_{\mathbb{R}^2}\phi_r(x)|\nabla u_\varepsilon(x,t)|^2\,dx - \int_{\mathbb{R}^2}\phi_r (x) u_\varepsilon(x,t)f_\varepsilon(x,t)\,dx\right| \leq C r^{-2}\|X\|_{\dot{H}^1(\mathbb{T})}^4.
\end{equation*}
By the convergence \eqref{eqn: convergence of the regularized solution to the original solution} and $\nabla u_\varepsilon \rightarrow \nabla u_X$ in $L^2_{loc}(\mathbb{R}^2)$, it becomes
\begin{equation*}
\left|\int_{\mathbb{R}^2}\phi_r(x)|\nabla u_X(x,t)|^2\,dx - \int_{\mathbb{R}^2}\phi_r (x) u_X(x,t)f(x,t)\,dx\right| \leq C r^{-2}\|X\|_{\dot{H}^1(\mathbb{T})}^4.
\end{equation*}
Take $r\rightarrow \infty$ and we find
\begin{equation}
\begin{split}
\int_{\mathbb{R}^2}|\nabla u_X(x,t)|^2\,dx = &\;\int_{\mathbb{R}^2} u_X(x,t)f(x,t)\,dx\\
=&\;\int_{\mathbb{R}^2}\int_{\mathbb{T}}u_X(x,t)F(s,t)\delta(x-X(s,t))\,dsdx=\int_{\mathbb{T}}u_X(X(s,t),t)F(s,t)\,ds\\
=&\;\int_{\mathbb{T}}X_t(s,t)X''(s,t)\,ds =-\int_{\mathbb{T}}X'_t(s,t)X'(s,t)\,ds\\
=&\;-\frac{1}{2}\frac{d}{dt}\int_{\mathbb{T}}|X'(s,t)|^2\,ds.
\end{split}
\label{eqn: energy estimate on each time slice}
\end{equation}
The last equality, estabilished in \cite{temam1984navier} in Chapter III, \S\,1.4, holds in the scalar distribution sense.
By some limiting argument and the assumption that $X\in C_T H^2(\mathbb{T})$, we may take integral on both sides of \eqref{eqn: energy estimate on each time slice} in $t$ from $0$ to $T$ to obtain \eqref{eqn: energy estimate of Stokes immersed boundary problem}.
\end{proof}

\subsection{A Priori Estimates Involving $\mathcal{L}$}\label{appendix section: estimates involving L}
We state several a priori estimates involving the operator $\mathcal{L}$ without proofs.
\begin{lemma}\label{lemma: improved Hs estimate and Hs continuity of semigroup solution}
For $\forall\, v_0\in H^l(\mathbb{T})$ with arbitrary $l\in\mathbb{R}_+$,
\begin{enumerate}
\item $\|\mathrm{e}^{t\mathcal{L}}v_0\|_{\dot{H}^l(\mathbb{T})} \leq \mathrm{e}^{-t/4}\|v_0\|_{\dot{H}^l(\mathbb{T})}$;
\item $\mathrm{e}^{t\mathcal{L}}v_0\in C([0,+\infty);H^l(\mathbb{T}))$;
\item $\mathrm{e}^{t\mathcal{L}}v_0\rightarrow v_0$ in $H^l(\mathbb{T})$ as $t\rightarrow 0^+$.
\end{enumerate}
\end{lemma}

\begin{lemma}\label{lemma: a priori estimate of nonlocal eqn}
Given $T>0$, let $h\in L^2_T H^l(\mathbb{T})$.
The model equation
\begin{equation}
\partial_t v(s,t) = \mathcal{L}v(s,t) +h(s,t),\quad v(s,0) = v_0(s),\quad s\in \mathbb{T},\; t\geq 0
\label{eqn: model nonlocal equation}
\end{equation}
has a unique solution $v\in L_T^\infty H^{l+1/2}\cap L_T^2 H^{l+1}(\mathbb{T})$ with $v_t \in L_T^2 H^l(\mathbb{T})$, satisfying the following a priori estimates: for $\forall\,t\in[0,T]$,
\begin{enumerate}
\item
\begin{equation}
\|v\|_{\dot{H}^{l+1/2}(\mathbb{T})}(t)\leq \mathrm{e}^{-t/4}\|v_0\|_{\dot{H}^{l+1/2}(\mathbb{T})}+ \sqrt{2}\|h\|_{L_{[0,t]}^2 \dot{H}^{l}(\mathbb{T})},
\label{eqn: a priori estimate for nonlocal eqn L_infty_in_time estimate}
\end{equation}
\item
\begin{equation}
\|v\|_{\dot{H}^{l+1/2}(\mathbb{T})}^2(t)+\frac{1}{4}\|v\|_{L^2_{[0,t]} \dot{H}^{l+1}(\mathbb{T})}^2 \leq \|v_0\|_{\dot{H}^{l+1/2}(\mathbb{T})}^2+ 4\|h\|_{L_{[0,t]}^2 \dot{H}^l(\mathbb{T})}^2.
\label{eqn: a priori estimate for nonlocal eqn L_2_in_time estimate}
\end{equation}
Hence,
\begin{equation*}
\|v\|_{L^\infty_{[0,t]}\dot{H}^{l+1/2}\cap L^2_{[0,t]}\dot{H}^{l+1}(\mathbb{T})}\leq 3\|v_0\|_{\dot{H}^{l+1/2}(\mathbb{T})}+ 6\|h\|_{L^2_{[0,t]}\dot{H}^l(\mathbb{T})}.
\end{equation*}
In particular,
\begin{equation*}
\|\mathrm{e}^{t\mathcal{L}}v_0\|_{L^\infty_{[0,t]}\dot{H}^{l+1/2}\cap L^2_{[0,t]}\dot{H}^{l+1}(\mathbb{T})}\leq 3\|v_0\|_{\dot{H}^{l+1/2}(\mathbb{T})}.
\end{equation*}
\item
\begin{equation*}
\|\partial_t v\|_{L^2_{[0,t]} \dot{H}^l(\mathbb{T})} \leq \frac{1}{2}\|v_0\|_{\dot{H}^{l+1/2}(\mathbb{T})}+\|h\|_{L^2_{[0,t]}\dot{H}^l(\mathbb{T})}.
\end{equation*}
\end{enumerate}
\end{lemma}

\subsection{Auxiliary Calculations}\label{appendix section: auxiliary calculations}
The following lemma is used to derive a simplification of $\Gamma_1(s,s')$ used in Section \ref{section: a priori estimates of the immersed boundary problem}.
\begin{lemma}\label{lemma: simplification of Gamma_1(s,s')}
Let $\Gamma_1(s,s')$ be defined by \eqref{eqn: introduce the notation Gamma_1}, i.e.
\begin{equation*}
\Gamma_1(s,s') = \left(-\partial_{ss'}[G(X(s)-X(s'))]-\frac{Id}{16\pi\sin^2\left(\frac{s'-s}{2}\right)}\right)(X'(s')-X'(s)).
\end{equation*}
Then with the notations introduced in \eqref{eqn: definition of L M N}, for $\forall\,s,s'\in\mathbb{T}$, $s'\not = s$, we have \eqref{eqn: simplified Gamma order 1}.

\begin{proof}
We shall simplify $\Gamma_1(s,s')$ by exploring cancelations between its terms. Using the notations introduced in \eqref{eqn: definition of L M N}, we calculate that
\begin{equation*}
\begin{split}
&\;4\tau\pi\Gamma_1(s,s')\\
=&\;-\frac{X'(s)\cdot X'(s')}{|L|^2}M + \frac{2(X'(s)\cdot L)(X'(s')\cdot L)}{|L|^4}M + \frac{X'(s)\cdot M}{|L|^2}X'(s')-\frac{2(X'(s)\cdot L)( L\cdot M)}{|L|^4}X'(s')\\
&\;+\frac{X'(s')\cdot M}{|L|^2}X'(s)-\frac{2(X'(s')\cdot M)( X'(s)\cdot L)}{|L|^4}L-\frac{2(L\cdot X'(s'))(L\cdot M)}{|L|^4}X'(s)\\
&\;-\frac{2(L\cdot X'(s'))(X'(s)\cdot M)}{|L|^4}L -\frac{2(L\cdot M)(X'(s)\cdot X'(s'))}{|L|^4}L\\
&\;+\frac{8 (L\cdot M) (L\cdot X'(s')) (L\cdot X'(s))}{|L|^6}L -\frac{\tau^2}{4\sin^2(\frac{\tau}{2})}M.
\end{split}
\end{equation*}
We simplify the first four terms by plugging in $X'(s') = X'(s)+\tau M$,
\begin{equation*}
\begin{split}
&\; -\frac{X'(s)\cdot (X'(s)+\tau M)}{|L|^2}M + \frac{2(X'(s)\cdot L)((X'(s)+\tau M)\cdot L)}{|L|^4}M\\
&\; + \frac{X'(s)\cdot M}{|L|^2}(X'(s)+\tau M)-\frac{2(X'(s)\cdot L)( L\cdot M)}{|L|^4}(X'(s)+\tau M)\\
=&\; -\frac{|X'(s)|^2}{|L|^2}M + \frac{2(X'(s)\cdot L)^2}{|L|^4}M + \frac{X'(s)\cdot M}{|L|^2}X'(s)-\frac{2(X'(s)\cdot L)( L\cdot M)}{|L|^4}X'(s).
\end{split}
\end{equation*}
Hence,
\begin{equation*}
\begin{split}
4\pi\Gamma_1(s,s') =&\; \frac{1}{\tau}\left(-\frac{|X'(s)|^2}{|L|^2}M + \frac{2(X'(s)\cdot L)^2}{|L|^4}M-\frac{\tau^2}{4\sin^2(\frac{\tau}{2})}M\right)\\
&\;+ \frac{1}{\tau}\left(\frac{X'(s)\cdot M}{|L|^2}X'(s)-\frac{2(X'(s)\cdot L)( L\cdot M)}{|L|^4}X'(s)+\frac{X'(s')\cdot M}{|L|^2}X'(s)\right)\\
&\;+\frac{1}{\tau}\left(-\frac{2(X'(s')\cdot M)( X'(s)\cdot L)}{|L|^4}L-\frac{2(L\cdot X'(s'))(L\cdot M)}{|L|^4}X'(s)\right.\\
&\;-\frac{2(L\cdot X'(s'))(X'(s)\cdot M)}{|L|^4}L -\frac{2(L\cdot M)(X'(s)\cdot X'(s'))}{|L|^4}L\\
&\;\left.+\frac{8 (L\cdot M) (L\cdot X'(s')) (L\cdot X'(s))}{|L|^6}L \right)\\
\triangleq &\;A_1(s,s')+A_2(s,s')+A_3(s,s').
\end{split}
\end{equation*}
Using $X'(s) = L-\tau N$, we calculate that
\begin{equation}
\begin{split}
A_1 = &\;\frac{1}{\tau}\left(-\frac{|X'(s)|^2}{|L|^2}M + \frac{2(X'(s)\cdot L)}{|L|^2}M - \frac{2\tau(N\cdot L)(X'(s)\cdot L)}{|L|^4}M -M - \left(\frac{\tau^2}{4\sin^2(\frac{\tau}{2})}-1\right)M\right)\\
= &\;\frac{1}{\tau}\left(-\frac{|X'(s)|^2}{|L|^2}M + \frac{2X'(s)\cdot L}{|L|^2}M -M \right) - \frac{2(N\cdot L)(X'(s)\cdot L)}{|L|^4}M - \left(\frac{\tau^2 - 4\sin^2(\frac{\tau}{2})}{4\tau\sin^2(\frac{\tau}{2})}\right)M\\
=&\;-\frac{1}{\tau}\frac{|X'(s)-L|^2}{|L|^2}M - \frac{2(N\cdot L)(X'(s)\cdot L)}{|L|^4}M - \left(\frac{\tau^2 - 4\sin^2(\frac{\tau}{2})}{4\tau\sin^2(\frac{\tau}{2})}\right)M\\
=&\;\frac{(X'(s)-L)\cdot N}{|L|^2}M - \frac{2(N\cdot L)(X'(s)\cdot L)}{|L|^4}M - \left(\frac{\tau^2 - 4\sin^2(\frac{\tau}{2})}{4\tau\sin^2(\frac{\tau}{2})}\right)M.
\end{split}
\label{eqn: C 1 beta estimate integrand part 1}
\end{equation}
Similarly,
\begin{equation}
\begin{split}
A_2 = &\;\frac{1}{\tau}\left(\frac{X'(s)\cdot M}{|L|^2}X'(s)-\frac{2L\cdot M}{|L|^2}X'(s)+\frac{2\tau(N\cdot L)( L\cdot M)}{|L|^4}X'(s)+\frac{X'(s')\cdot M}{|L|^2}X'(s)\right)\\
= &\;\frac{1}{\tau}\frac{(X'(s)+X'(s')-2L)\cdot M}{|L|^2}X'(s)+\frac{2(N\cdot L)( L\cdot M)}{|L|^4}X'(s)\\
= &\;\frac{(M-2N)\cdot M}{|L|^2}X'(s)+\frac{2(N\cdot L)( L\cdot M)}{|L|^4}X'(s).
\end{split}
\label{eqn: C 1 beta estimate integrand part 2}
\end{equation}
For $A_3$, we split the last term into four and look for cancellations with the other four terms. That is,
\begin{equation*}
\begin{split}
\tau A_3 = &\;\frac{2 (L\cdot M) (L\cdot X'(s')) (L\cdot X'(s))}{|L|^6}L -\frac{2(X'(s')\cdot M)( X'(s)\cdot L)}{|L|^4}L\\
&\;+\frac{2 (L\cdot M) (L\cdot X'(s')) (L\cdot X'(s))}{|L|^6}L-\frac{2(L\cdot X'(s'))(L\cdot M)}{|L|^4}X'(s)\\
&\;+\frac{2 (L\cdot M) (L\cdot X'(s')) (L\cdot X'(s))}{|L|^6}L-\frac{2(L\cdot X'(s'))(X'(s)\cdot M)}{|L|^4}L\\
&\;+\frac{2 (L\cdot M) (L\cdot X'(s')) (L\cdot X'(s))}{|L|^6}L -\frac{2(L\cdot M)(X'(s)\cdot X'(s'))}{|L|^4}L\\
=&\;\frac{2\tau (L\cdot M) (L\cdot (M-N)) (L\cdot X'(s))}{|L|^6}L+\frac{2 (L\cdot M)(L\cdot X'(s))}{|L|^4}L -\frac{2(X'(s')\cdot M)(X'(s)\cdot L)}{|L|^4}L\\
&\;-\frac{2\tau (L\cdot M) (L\cdot X'(s')) (L\cdot N)}{|L|^6}L+\frac{2 (L\cdot M) (L\cdot X'(s'))}{|L|^4}L-\frac{2(L\cdot X'(s'))(L\cdot M)}{|L|^4}X'(s)\\
&\;-\frac{2\tau (L\cdot M) (L\cdot X'(s')) (L\cdot N)}{|L|^6}L+\frac{2 (L\cdot M) (L\cdot X'(s'))}{|L|^4}L-\frac{2(L\cdot X'(s'))(X'(s)\cdot M)}{|L|^4}L\\
&\;-\frac{2 \tau(L\cdot M) (L\cdot X'(s')) (L\cdot N)}{|L|^6}L+\frac{2 (L\cdot M) (L\cdot X'(s'))}{|L|^4}L -\frac{2(L\cdot M)(X'(s)\cdot X'(s'))}{|L|^4}L\\
=&\;\frac{2\tau (L\cdot M) (L\cdot (M-N)) (L\cdot X'(s))}{|L|^6}L+\frac{2\tau ((N-M)\cdot M)(L\cdot X'(s))}{|L|^4}L\\
&\;-\frac{2\tau (L\cdot M) (L\cdot X'(s')) (L\cdot N)}{|L|^6}L+\frac{2 (L\cdot M) (L\cdot X'(s'))}{|L|^4}\tau N\\
&\;-\frac{2\tau (L\cdot M) (L\cdot X'(s')) (L\cdot N)}{|L|^6}L+\frac{2\tau (N\cdot M) (L\cdot X'(s'))}{|L|^4}L\\
&\;-\frac{2 \tau(L\cdot M) (L\cdot X'(s')) (L\cdot N)}{|L|^6}L+\frac{2\tau (L\cdot M) (N\cdot X'(s'))}{|L|^4}L.
\end{split}
\end{equation*}
Here we used $X'(s) = L-\tau N$ and $X'(s') = L+\tau(M-N)$. Therefore,
\begin{equation}
\begin{split}
A_3 =&\;\frac{2 (L\cdot M) (L\cdot (M-N)) (L\cdot X'(s))}{|L|^6}L+\frac{2 ((N-M)\cdot M)(L\cdot X'(s))}{|L|^4}L\\
&\;-\frac{6 (L\cdot M) (L\cdot X'(s')) (L\cdot N)}{|L|^6}L+\frac{2 (L\cdot M) (L\cdot X'(s'))}{|L|^4} N\\
&\;+\frac{2 (N\cdot M) (L\cdot X'(s'))}{|L|^4}L+\frac{2 (L\cdot M) (N\cdot X'(s'))}{|L|^4}L.
\end{split}
\label{eqn: C 1 beta estimate integrand part 3}
\end{equation}
Combining \eqref{eqn: C 1 beta estimate integrand part 1}, \eqref{eqn: C 1 beta estimate integrand part 2} and \eqref{eqn: C 1 beta estimate integrand part 3}, we obtain the desired simplification \eqref{eqn: simplified Gamma order 1}.
\end{proof}
\end{lemma}

The following lemma states that the $\eta$-derivative of $u_{X_\eta}(X_{\eta}(s))$ can commute with the integral in the representation of $u_{X_\eta}(X_{\eta}(s))$. It will be used in the proofs of Lemma \ref{lemma: linearization of velocity field around equilibrium} and Lemma \ref{lemma: final representation of the linearization of velocity near equilibrium}.

\begin{lemma}\label{lemma: eta derivative and the integral in u_X commute}
Let $\eta\in[0,1]$ and let $X_\eta$ be defined as in \eqref{eqn: defintion of X_eta}. Let
\begin{equation*}
\begin{split}
u_{X_\eta}(X_\eta(s)) = &\;\frac{1}{4\pi}\int_{\mathbb{T}}\frac{L_{X_\eta}\cdot X_\eta'(s')}{|L_{X_\eta}|^2}M_{X_\eta}-\frac{L_{X_\eta}\cdot M_{X_\eta}}{|L_{X_\eta}|^2}X'_\eta(s') -\frac{X'_\eta(s')\cdot M_{X_\eta}}{|L_{X_\eta}|^2}L_{X_\eta}\,ds'\\
&\;+\frac{1}{4\pi}\int_{\mathbb{T}}\frac{2L_{X_\eta}\cdot X'_\eta(s')L_{X_\eta}\cdot M_{X_\eta}}{|L_{X_\eta}|^4}L_{X_\eta}\,ds'\\
\triangleq &\; \int_{\mathbb{T}} h_\eta(s,s')\,ds'.
\end{split}
\end{equation*}
Then under the same assumptions as in Lemma \ref{lemma: linearization of velocity field around equilibrium}, for $\forall\,\eta\in[0,1]$,
\begin{equation*}
\frac{\partial}{\partial \eta}u_{X_\eta}(X_\eta(s)) = \int_{\mathbb{T}} \frac{\partial}{\partial \eta}h_\eta(s,s')\,ds'.
\end{equation*}
\begin{proof}
By definition,
\begin{equation}
\frac{\partial}{\partial \eta}u_{X_\eta}(X_\eta(s)) = \lim_{\eta'\rightarrow \eta}\int_{\mathbb{T}} \frac{h_{\eta'}(s,s')-h_\eta(s,s')}{\eta'-\eta}\,ds'.
\label{eqn: def of eta derivative of u}
\end{equation}
We shall check the conditions of the dominated convergence theorem to show that the limit and the integral commute.
In particular, we need to show that for $\forall\,s\in\mathbb{T}$, there exists an $s'$-integrable function $h(s,s')$, such that for $\forall\,\eta_1,\eta_2\in[0,1]$, $\eta_1\not = \eta_2$,
\begin{equation}
|h_{\eta_1}(s,s')-h_{\eta_2}(s,s')|\leq |\eta_1-\eta_2|h(s,s').
\label{eqn: condition of the DCT}
\end{equation}
By definition, we calculate that
\begin{equation}
\begin{split}
&\;4\pi|h_{\eta_1}(s,s')-h_{\eta_2}(s,s')|\\
\leq &\;\left|\frac{L_{X_{\eta_1}}\cdot X_{\eta_1}'(s')}{|L_{X_{\eta_1}}|^2}M_{X_{\eta_1}}-\frac{L_{X_{\eta_2}}\cdot X_{\eta_2}'(s')}{|L_{X_{\eta_2}}|^2}M_{X_{\eta_2}}\right|+\left|\frac{L_{X_{\eta_1}}\cdot M_{X_{\eta_1}}}{|L_{X_{\eta_1}}|^2}X'_{\eta_1}(s')-\frac{L_{X_{\eta_2}}\cdot M_{X_{\eta_2}}}{|L_{X_{\eta_2}}|^2}X'_{\eta_2}(s')\right|\\
&\;+\left|\frac{X'_{\eta_1}(s')\cdot M_{X_{\eta_1}}}{|L_{X_{\eta_1}}|^2}L_{X_{\eta_1}}-\frac{X'_{\eta_2}(s')\cdot M_{X_{\eta_2}}}{|L_{X_{\eta_2}}|^2}L_{X_{\eta_2}}\right|\\
&\;+\left|\frac{2L_{X_{\eta_1}}\cdot X'_{\eta_1}(s')L_{X_{\eta_1}}\cdot M_{X_{\eta_1}}}{|L_{X_{\eta_1}}|^4}L_{X_{\eta_1}}-\frac{2L_{X_{\eta_2}}\cdot X'_{\eta_2}(s')L_{X_{\eta_2}}\cdot M_{X_{\eta_2}}}{|L_{X_{\eta_2}}|^4}L_{X_{\eta_2}}\right|.
\end{split}
\label{eqn: difference between h_1 and h_2}
\end{equation}
For conciseness, we only show estimate of one of the terms above. Thanks to the uniform estimates \eqref{eqn: uniform H2.5 upper bound for the family of configurations near equilibrium} and \eqref{eqn: uniform stretching constant for the family of configurations near equilibrium},
\begin{equation*}
\begin{split}
&\;\left|\frac{L_{X_{\eta_1}}\cdot X_{\eta_1}'(s')}{|L_{X_{\eta_1}}|^2}M_{X_{\eta_1}}-\frac{L_{X_{\eta_2}}\cdot X_{\eta_2}'(s')}{|L_{X_{\eta_2}}|^2}M_{X_{\eta_2}}\right|\\
\leq &\;\left|\frac{(L_{X_{\eta_1}}-L_{X_{\eta_2}})\cdot X_{\eta_1}'(s')}{|L_{X_{\eta_1}}|^2}M_{X_{\eta_1}}\right|+\left|\frac{L_{X_{\eta_2}}\cdot (X_{\eta_1}'(s')-X_{\eta_2}'(s'))}{|L_{X_{\eta_1}}|^2}M_{X_{\eta_1}}\right|\\
&\;+\left|\frac{L_{X_{\eta_2}}\cdot X_{\eta_2}'(s')}{|L_{X_{\eta_1}}|^2}(M_{X_{\eta_1}}-M_{X_{\eta_2}})\right|+\left|\frac{L_{X_{\eta_2}}\cdot X_{\eta_2}'(s')}{|L_{X_{\eta_1}}|^2}M_{X_{\eta_2}}\frac{|L_{X_{\eta_2}}|^2-|L_{X_{\eta_1}}|^2}{|L_{X_{\eta_2}}|^2}\right|\\
\leq &\;\left|\frac{(\eta_1-\eta_2)L_{D}\cdot X_{\eta_1}'(s')}{|L_{X_{\eta_1}}|^2}M_{X_{\eta_1}}\right|+\left|\frac{L_{X_{\eta_2}}\cdot (\eta_1-\eta_2)D'(s')}{|L_{X_{\eta_1}}|^2}M_{X_{\eta_1}}\right|\\
&\;+\left|\frac{L_{X_{\eta_2}}\cdot X_{\eta_2}'(s')}{|L_{X_{\eta_1}}|^2}(\eta_1-\eta_2)M_{D}\right|+\left|\frac{L_{X_{\eta_2}}\cdot X_{\eta_2}'(s')}{|L_{X_{\eta_1}}|^2}M_{X_{\eta_2}}\frac{(L_{X_{\eta_2}}+L_{X_{\eta_1}})\cdot (\eta_1-\eta_2)L_D}{|L_{X_{\eta_2}}|^2}\right|\\
\leq &\;C|\eta_1-\eta_2|(\|L_D\|_{L^{\infty}_{s'}}\|X_{\eta_1}'\|_{L^\infty}|M_{X_{\eta_1}}|+\|L_{X_{\eta_2}}\|_{L^\infty_{s'}}\|D'\|_{L^\infty}|M_{X_{\eta_1}}|\\
&\;\quad +\|L_{X_{\eta_2}}\|_{L^{\infty}_{s'}}\|X_{\eta_2}'\|_{L^{\infty}}|M_{D}|+\|X_{\eta_2}'\|_{L^{\infty}}|M_{X_{\eta_2}}|\|L_D\|_{L^{\infty}_{s'}})\\
\leq &\;C|\eta_1-\eta_2|[\|D'\|_{L^\infty}(\|X_{\eta_1}'\|_{L^\infty}+\|X_{\eta_2}'\|_{L^\infty})(|M_{X_{\eta_1}}|+|M_{X_{\eta_2}}|) + \|X_{\eta_2}'\|^2_{L^\infty}|M_D|]\\
\leq &\;C|\eta_1-\eta_2|(|M_{X_{\eta_1}}|+|M_{X_{\eta_2}}|+|M_D|)\\
\leq &\;C|\eta_1-\eta_2|(|M_{X_*}|+|M_D|),
\end{split}
\end{equation*}
where $C$ is a universal constant. The other terms in \eqref{eqn: difference between h_1 and h_2} can be estimated in a similar fashion. Hence, we obtain that, for $s'\not = s$,
\begin{equation*}
\begin{split}
|h_{\eta_1}(s,s')-h_{\eta_2}(s,s')|\leq &\;C|\eta_1-\eta_2|(|M_{X_*}|+|M_D|)\\
\leq &\;\frac{C|\eta_1-\eta_2|}{|s'-s|}\int_s^{s'}|X''_*(\omega)|+|D''(\omega)|\,d\omega\\
\leq &\;C|\eta_1-\eta_2|(|\mathcal{M}X''_*(s')|+|\mathcal{M}D''(s')|),
\end{split}
\end{equation*}
where $\mathcal{M}$ is again the centered Hardy-Littlewood maximal operator on $\mathbb{T}$.
Hence \eqref{eqn: condition of the DCT} is proved with $h(s,s') = C(|\mathcal{M}X''_*(s')|+|\mathcal{M}D''(s')|)\in L^1_{s'}(\mathbb{T})$. Note that $h$ is independent of $\eta_1$, $\eta_2$ and $s$.
By dominated convergence theorem, the limit and the integral in \eqref{eqn: def of eta derivative of u} commute, which proves the Lemma.
\end{proof}
\end{lemma}

Finally, we come to prove Lemma \ref{lemma: final representation of the linearization of velocity near equilibrium}, which calculates the leading term of $u_X(X(s))$ in \eqref{eqn: first approximation by linearization of velocity around equilibrium}.
\begin{proof}[Proof of Lemma \ref{lemma: final representation of the linearization of velocity near equilibrium}]
This time, we used \eqref{eqn: expression for velocity field} as the representation of $u_{X_{\eta}}(X_\eta(s))$,
\begin{equation*}
u_{X_\eta} = \frac{1}{4\pi}\int_{\mathbb{T}} \left(-\ln |X_\eta(s')-X_{\eta}(s)| Id +\frac{(X_\eta(s')-X_{\eta}(s)) \otimes (X_\eta(s')-X_{\eta}(s))}{|X_\eta(s')-X_{\eta}(s)|^2}\right)X_{\eta}''(s') \,ds'.
\end{equation*}
It has been showed before (see Section \ref{section: proof of contour dynamic formulation}) that, with the well-stretched condition \eqref{eqn: well_stretched assumption}, the integral with logarithmic singularity is well-defined.
Hence, by virtue of Lemma \ref{lemma: eta derivative and the integral in u_X commute},
\begin{equation*}
\begin{split}
&\;\left.\frac{\partial}{\partial\eta}\right|_{\eta = 0}u_{X_\eta}(X_\eta(s))\\
=&\;\frac{1}{4\pi}\int_{\mathbb{T}}  \left[-\frac{(X_*(s')-X_*(s))\cdot (D(s')-D(s))}{|X_*(s')-X_*(s)|^2} Id \right.\\
&\; +\frac{(D(s')-D(s)) \otimes (X_*(s')-X_*(s))}{|X_*(s')-X_*(s)|^2}+\frac{(X_*(s')-X_*(s)) \otimes (D(s')-D(s))}{|X_*(s')-X_*(s)|^2}\\
&\;\left. -\frac{(X_*(s')-X_*(s)) \otimes (X_*(s')-X_*(s))\cdot 2(X_*(s')-X_*(s))\cdot(D(s')-D(s))}{|X_*(s')-X_*(s)|^4}\right]X_{*}''(s')\,ds'\\
&\;+\frac{1}{4\pi}\int_{\mathbb{T}}\left(-\ln |X_*(s')-X_*(s)| Id +\frac{(X_*(s')-X_*(s)) \otimes (X_*(s')-X_*(s))}{|X_*(s')-X_*(s)|^2}\right)D''(s')\,ds'.
\end{split}
\end{equation*}
We split $X_{*}''(s')$ into two terms, namely, $X_{*}''(s') = -X_*(s') = -\frac{1}{2}(X_*(s')-X_*(s)) - \frac{1}{2}(X_*(s')+X_*(s))$.
\begin{equation}
\begin{split}
&\;\left.\frac{\partial}{\partial\eta}\right|_{\eta = 0}u_{X_\eta}(X_\eta(s))\\
=&\;\frac{1}{8\pi}\int_{\mathbb{T}}  \frac{(X_*(s')-X_*(s))\cdot (D(s')-D(s))}{|X_*(s')-X_*(s)|^2} (X_*(s')-X_*(s))\\
&\; -(D(s')-D(s)) -\frac{(X_*(s')-X_*(s)) \cdot (D(s')-D(s))}{|X_*(s')-X_*(s)|^2}(X_*(s')-X_*(s))\\
&\; +\frac{2(X_*(s')-X_*(s))\cdot(D(s')-D(s))}{|X_*(s')-X_*(s)|^2}(X_*(s')-X_*(s))\,ds'\\
&\;+\frac{1}{8\pi}\int_{\mathbb{T}}  \frac{(X_*(s')-X_*(s))\cdot (D(s')-D(s))}{|X_*(s')-X_*(s)|^2} (X_*(s')+X_*(s)) \\
&\;-\frac{(X_*(s')+X_*(s))\cdot (X_*(s')-X_*(s))}{|X_*(s')-X_*(s)|^2}(D(s')-D(s))\\
&\;-\frac{(X_*(s')+X_*(s))\cdot (D(s')-D(s))}{|X_*(s')-X_*(s)|^2}(X_*(s')-X_*(s))\\
&\;+\frac{(X_*(s')-X_*(s))\cdot (X_*(s')+X_*(s)) \cdot 2(X_*(s')-X_*(s))\cdot(D(s')-D(s))}{|X_*(s')-X_*(s)|^4}(X_*(s')-X_*(s))\,ds'\\
&\;+\frac{1}{4\pi}\int_{\mathbb{T}}\left(-\ln |X_*(s')-X_*(s)| Id +\frac{(X_*(s')-X_*(s)) \otimes (X_*(s')-X_*(s))}{|X_*(s')-X_*(s)|^2}\right)D''(s')\,ds'\\
=&\;\frac{1}{8\pi}\int_{\mathbb{T}} -(D(s')-D(s)) +\frac{2(X_*(s')-X_*(s))\cdot(D(s')-D(s))}{|X_*(s')-X_*(s)|^2}(X_*(s')-X_*(s))\,ds'\\
&\;+\frac{1}{8\pi}\int_{\mathbb{T}} \frac{(X_*(s')-X_*(s))\cdot (D(s')-D(s))}{|X_*(s')-X_*(s)|^2} (X_*(s')+X_*(s))\,ds'\\
&\;-\frac{1}{8\pi}\int_{\mathbb{T}} \frac{(X_*(s')+X_*(s))\cdot (D(s')-D(s))}{|X_*(s')-X_*(s)|^2}(X_*(s')-X_*(s))\,ds'\\
&\;+\frac{1}{4\pi}\int_{\mathbb{T}} \left(-\ln |X_*(s')-X_*(s)| Id +\frac{(X_*(s')-X_*(s)) \otimes (X_*(s')-X_*(s))}{|X_*(s')-X_*(s)|^2}\right)D''(s')\,ds'\\
=&\;\frac{1}{4}D(s)+\frac{1}{4\pi}\int_{\mathbb{T}} \frac{(X_*(s')-X_*(s))\cdot (D(s')-D(s))}{|X_*(s')-X_*(s)|^2}X_*(s')\,ds'\\
&\;-\frac{1}{4\pi}\int_{\mathbb{T}} \frac{X_*(s)\cdot(D(s')-D(s))}{|X_*(s')-X_*(s)|^2}(X_*(s')-X_*(s))\,ds'\\
&\;+\frac{1}{4\pi}\int_{\mathbb{T}} -\ln |X_*(s')-X_*(s)| D''(s') +\frac{(X_*(s')-X_*(s)) \otimes (X_*(s')-X_*(s))}{|X_*(s')-X_*(s)|^2}D''(s')\,ds'.
\end{split}
\label{eqn: representation of the first variation of velocity field}
\end{equation}
Here we used the fact that $(X_*(s')-X_*(s))\cdot (X_*(s')+X_*(s))=0$ and $\int_{\mathbb{T}}D(s')\,ds' = 0$.
Since $X_*(s) = (\cos s, \sin s)^T$, we plug this into \eqref{eqn: representation of the first variation of velocity field} and find that
\begin{equation*}
\begin{split}
&\;\frac{1}{4\pi}\int_{\mathbb{T}} \frac{(X_*(s')-X_*(s))\cdot (D(s')-D(s))}{|X_*(s')-X_*(s)|^2}X_*(s')- \frac{X_*(s)\cdot(D(s')-D(s))}{|X_*(s')-X_*(s)|^2}(X_*(s')-X_*(s))\,ds'\\
=&\;\frac{1}{4\pi}\int_{\mathbb{T}} \left[\frac{X_*(s')\otimes(X_*(s')-X_*(s)) - (X_*(s')-X_*(s))\otimes X_*(s')}{|X_*(s')-X_*(s)|^2}\right](D(s')-D(s))\,ds'\\
&\; +\frac{1}{4\pi} \int_{\mathbb{T}} \frac{(X_*(s')-X_*(s))\otimes (X_*(s')-X_*(s))}{|X_*(s')-X_*(s)|^2}(D(s')-D(s))\,ds'\\
=&\;\frac{1}{4\pi}\int_{\mathbb{T}}\frac{1}{2}\cot\left(\frac{s'-s}{2}\right) \left(\begin{array}{cc}0&1\\-1&0\end{array}\right)(D(s')-D(s))\,ds'\\
&\;+\frac{1}{4\pi}\int_{\mathbb{T}}\left(\begin{array}{cc}\frac{1}{2}(1-\cos(s'+s))&-\frac{1}{2}\sin(s'+s)\\ -\frac{1}{2}\sin(s'+s)&\frac{1}{2}(1+\cos(s'+s))\end{array}\right)(D(s')-D(s))\,ds'\\
=&\;-\frac{1}{4}\left(\begin{array}{cc}0&1\\-1&0\end{array}\right)\mathcal{H}D(s)+\frac{1}{4\pi}\int_{\mathbb{T}}-\frac{1}{2}\left(\begin{array}{cc}\cos(s'+s)&\sin(s'+s)\\ \sin(s'+s)&-\cos(s'+s)\end{array}\right)(D(s')-D(s))\,ds'\\
&\; + \frac{1}{8\pi}\int_{\mathbb{T}}(D(s')-D(s))\,ds'\\
=&\;-\frac{1}{4}\left(\begin{array}{cc}0&1\\-1&0\end{array}\right)\mathcal{H}D(s)-\frac{1}{8\pi}\int_{\mathbb{T}}\left(\begin{array}{cc}\cos(s'+s)&\sin(s'+s)\\ \sin(s'+s)&-\cos(s'+s)\end{array}\right)D(s')\,ds'-\frac{1}{4}D(s).
\end{split}
\end{equation*}
Here $\mathcal{H}$ is the Hilbert transform on $\mathbb{T}$ \cite{grafakos2008classical}. 
Moreover,
\begin{equation*}
\begin{split}
\frac{1}{4\pi}\int_{\mathbb{T}} -\ln |X_*(s')-X_*(s)| D''(s') \,ds' =&\; -\frac{1}{8\pi} \int_{\mathbb{T}} \ln \left[4\sin^2\left(\frac{s'-s}{2}\right)\right] D''(s') \,ds'\\
=&\; \frac{1}{8\pi}\mathrm{p.v.} \int_{\mathbb{T}} \cot\left(\frac{s'-s}{2}\right) D'(s') \,ds'\\
=&\; -\frac{1}{4}\mathcal{H}D'(s),
\end{split}
\end{equation*}
and
\begin{equation*}
\begin{split}
&\;\frac{1}{4\pi}\int_{\mathbb{T}} \frac{(X_*(s')-X_*(s)) \otimes (X_*(s')-X_*(s))}{|X_*(s')-X_*(s)|^2}D''(s')\,ds'\\
=&\;\frac{1}{4\pi}\int_{\mathbb{T}} \left(\begin{array}{cc}\frac{1}{2}(1-\cos(s'+s))&-\frac{1}{2}\sin(s'+s)\\ -\frac{1}{2}\sin(s'+s)&\frac{1}{2}(1+\cos(s'+s))\end{array}\right)D''(s')\,ds'\\
=&\;-\frac{1}{8\pi}\int_{\mathbb{T}} \left(\begin{array}{cc}\cos(s'+s)&\sin(s'+s)\\ \sin(s'+s)&-\cos(s'+s)\end{array}\right)D''(s')\,ds'\\
=&\;\frac{1}{8\pi}\int_{\mathbb{T}} \left(\begin{array}{cc}\cos(s'+s)&\sin(s'+s)\\ \sin(s'+s)&-\cos(s'+s)\end{array}\right)D(s')\,ds'.
\end{split}
\end{equation*}
In the last line, we used the fact that only the Fourier modes of $D''(s')$ with wave numbers $\pm 1$ contributes to the integral, and thus replacing $D''(s')$ by $-D(s')$ does not change the integral.
Combining the above calculations with \eqref{eqn: representation of the first variation of velocity field}, we find the desired result \eqref{eqn: final representation of the linearization of velocity near equilibrium}.
\end{proof}

\section*{Acknowledgement}
We want to thank Prof.\;Michael Shelley and Prof.\;Wei Wang for helpful discussions. We also would like to thank Prof.\;Jim Portegies for his early notes that contain some preliminary but inspiring calculations on this problem.
We also thank anonymous referee for helpful comments and suggestions.
This work was partially supported by National Science Foundation under Award Number DMS-1501000.

\noindent Fang-Hua Lin\\
Courant Institute\\
251 Mercer St.\\
New York, NY 10012\\
USA\\
E-mail: linf@cims.nyu.edu\\

\noindent Jiajun Tong\\
Courant Institute\\
251 Mercer St.\\
New York, NY 10012\\
USA\\
E-mail: jiajun@cims.nyu.edu

\end{document}